\documentclass[11pt,reqno]{amsart}
\usepackage{mathrsfs}
\usepackage[breaklinks]{hyperref}
\usepackage[dvipsnames]{xcolor} 
\usepackage{diagbox}
\usepackage{latexsym}
\usepackage{upref, eucal}

\usepackage[headheight=110pt,top=1in, bottom=.9in, left=0.8in, right=0.8in]{geometry}
\allowdisplaybreaks
\usepackage{url}
\usepackage{amssymb}
\newcommand{\dd}{\mathrm{d}}
\newcommand{\df}{\dfrac}
\newcommand{\tf}{\tfrac}

\newcommand{\Res}{\operatorname{Res}}
\renewcommand{\Re}{\operatorname{Re}}

\newcommand{\s}{{\sigma}}

\renewcommand{\a}{\alpha}
\renewcommand{\b}{\beta}

\newcommand{\G}{\Gamma}

\newcommand{\z}{\zeta}

\renewcommand{\(}{\left\(}
\renewcommand{\)}{\right\)}
\renewcommand{\[}{\left\[}
\renewcommand{\]}{\right\]}
\newcommand{\di}{i} 
\let\dotlessi=\i
\renewcommand{\i}{\infty}
\numberwithin{equation}{section}
\theoremstyle{plain}
\newtheorem{theorem}{Theorem}[section]
\newtheorem{lemma}[theorem]{Lemma}
\newtheorem{corollary}[theorem]{Corollary}

\newtheorem{remark}[]{Remark}

\makeatletter
\def\proof{\@ifnextchar[{\@oproof}{\@nproof}}
\def\@oproof[#1][#2]{\trivlist\item[\hskip\labelsep\textit{#2 Proof of\
		#1.}~]\ignorespaces}
\def\@nproof{\trivlist\item[\hskip\labelsep\textit{Proof.}~]\ignorespaces}

\makeatother

\makeatletter
\def\@tocline#1#2#3#4#5#6#7{\relax
	\ifnum #1>\c@tocdepth 
	\else
	\par \addpenalty\@secpenalty\addvspace{#2}%
	\begingroup \hyphenpenalty\@M
	\@ifempty{#4}{%
		\@tempdima\csname r@tocindent\number#1\endcsname\relax
	}{%
		\@tempdima#4\relax
	}%
	\parindent\z@ \leftskip#3\relax \advance\leftskip\@tempdima\relax
	\rightskip\@pnumwidth plus4em \parfillskip-\@pnumwidth
	#5\leavevmode\hskip-\@tempdima
	\ifcase #1
	\or\or \hskip 1em \or \hskip 2em \else \hskip 3em \fi%
	#6\nobreak\relax
	\dotfill\hbox to\@pnumwidth{\@tocpagenum{#7}}\par
	\nobreak
	\endgroup
	\fi}
\makeatother

\allowdisplaybreaks

\begin{document}
	
\title[Voronoi summation formulas, oscillations of Riesz sums and identities]{Voronoi summation formulas, oscillations of Riesz sums, and  Ramanujan-Guinand and Cohen type identities}
\author{Shashank Chorge and Atul Dixit}
\address{Department of Mathematics, Indian Institute of Technology Gandhinagar, Palaj, Gandhinagar, Gujarat 382355, India} \email{shashankashok.c@iitgn.ac.in; adixit@iitgn.ac.in}
\begin{abstract}
	We derive Vorono\"{\dotlessi} summation formulas for the Liouville function $\lambda(n)$, the M\"{o}bius function $\mu(n)$, and for $d^{2}(n)$, where $d(n)$ is the divisor function. The formula for $\lambda(n)$ requires explicit evaluation of certain infinite series for which the use of the Vinogradov-Korobov zero-free region of the Riemann zeta function is indispensable. Several results of independent interest are obtained as special cases of these formulas. For example, a special case of the one for $\mu(n)$ is a famous result of Ramanujan, Hardy, and Littlewood. 
	 Cohen type and Ramanujan-Guinand type identities are established for $\lambda(n)$ and $\sigma_a(n)\sigma_b(n)$, where $\sigma_s(n)$ is the generalized divisor function. As expected, infinite series over the non-trivial zeros of $\zeta(s)$ now form an essential part of all of these formulas. A series involving $\sigma_a(n)\sigma_b(n)$ and product of  modified Bessel functions occurring in one of our identities has appeared in a recent work of Dorigoni and Treilis in string theory. Lastly, we obtain results on oscillations of Riesz sums associated to $\lambda(n), \mu(n)$ and of the error term of Riesz sum of $d^2(n)$ under the assumption of the Riemann Hypothesis, simplicity of the zeros of $\zeta(s)$, the Linear Independence conjecture, and a weaker form of the Gonek-Hejhal conjecture.
\end{abstract}
\maketitle
\tableofcontents

\section{Introduction}\label{intro}

Two of the famous unsolved problems in number theory are the Gauss circle problem and the Dirichlet divisor problem. Let $d(n)$ denote the number of positive divisors of $n$. The Dirichlet divisor problem aims at finding the optimal error bound on $\sum_{n\leq x}d(n)$.

Almost all of the improvements on the error term of $\sum_{n\leq x}d(n)$ rely on a formula of Vorono\"{\dotlessi} \cite{voronoi}, now bearing his name, and specifically known as the \emph{Vorono\"{\dotlessi} summation formula}. For example, Vorono\"{\dotlessi} himself used his result to improve the error bound from $O(\sqrt{x})$ to $O(x^{1/3}\log(x))$. A more general version of this formula, studied by Vorono\"{\dotlessi} himself, involves a test function $\phi$ satisfying certain conditions, and is a representation for the sum $\sum_{n\leq x}d(n)\phi(n)$. Putting further restrictions on $\phi$, one can even obtain a Vorono\"{\dotlessi} summation formula for the infinite series $\sum_{n=1}^{\infty}d(n)\phi(n)$ in which case the formula takes the form
\begin{align}\label{voronoi d(n)}
\sum_{n=1}^{\infty}d(n)\phi(n)=\int_{0}^{\infty}(2\gamma+\log(t))\phi(t)\dd t+2\pi\sum_{n=1}^{\infty}d(n)\int_{0}^{\infty}\phi(t)\left(\dfrac{2}{\pi}K_{0}(4\pi\sqrt{nt})-Y_{0}(4\pi\sqrt{nt})\right)\dd t,
\end{align}
where $\gamma $ is Euler's constant and $Y_{\nu}(\xi)$ is the Bessel function of the second kind of order $\nu$ defined by \cite[p.~64, Equation (1)]{watson-1966a}
\begin{align*}
	Y_{\nu}(\xi)=\frac{J_{\nu}(\xi)\cos(\pi \nu)-J_{-\nu}(\xi)}{\sin(\pi \nu)}
\end{align*}
for $\nu\notin\mathbb{Z}$, and by $Y_{n}(\xi)=\lim_{\nu\to n}Y_{\nu}(\xi)$ for $n\in\mathbb{Z}$, and where $J_{\nu}(\xi)$ is the Bessel function of the first kind of order $\nu$ defined by \cite[p.~40]{watson-1966a}
\begin{align}\label{sumbesselj}
	J_{\nu}(\xi):=\sum_{m=0}^{\infty}\frac{(-1)^m(\xi/2)^{2m+\nu}}{m!\Gamma(m+1+\nu)}, \quad |\xi|<\infty.
\end{align}
Moreover, $K_{\nu}(\xi)$ is the modified Bessel function of the second kind of order $\nu$ defined by \cite[p.~78, eq.~(6)]{watson-1966a}
\begin{align}\label{knus}
	K_{\nu}(\xi)&:=\frac{\pi}{2}\frac{\left(I_{-\nu}(\xi)-I_{\nu}(\xi)\right)}{\sin(\pi \nu)}\hspace{7mm} (\nu\in\mathbb{C}\backslash\mathbb{Z}),\nonumber\\
	K_{n}(\xi)&:=\lim_{\nu\to n}K_{\nu}(\xi)\hspace{21mm}(n\in\mathbb{Z}),
\end{align}
where $I_{\nu}(\xi)$ is the modified Bessel function of the first kind is defined by
\begin{equation}\label{besseli}
	I_{\nu}(\xi):=
	\begin{cases}
		e^{-\frac{1}{2}\pi \nu i}J_{\nu}(e^{\frac{1}{2}\pi i}\xi), & \text{if $-\pi<$ arg $\xi\leq\frac{\pi}{2}$,}\\
		e^{\frac{3}{2}\pi \nu i}J_{\nu}(e^{-\frac{3}{2}\pi i}\xi), & \text{if $\frac{\pi}{2}<$ arg $\xi\leq \pi$}.
	\end{cases}
\end{equation}
The reader is referred to the excellent survey articles \cite{ddp-survey}, \cite{popov2024} on this topic.

Different authors have obtained different conditions of varying generality and applicability under which \eqref{voronoi d(n)} is valid, for example, \eqref{voronoi d(n)} holds if $\phi:\mathbb{R}\to\mathbb{C}$ is a smooth compactly supported function on $(0, \infty)$.

A generalization of \eqref{voronoi d(n)} for the generalized divisor function $\sigma_s(n):=\sum_{d|n}d^s$ can be obtained by letting\footnote{The interchange of the order of limit and summation can be justified under certain extra hypotheses on $\phi$ than the ones given in \cite[Theorem 6.3]{bdrz1}. For example, it is definitely valid for $\phi$ belonging to the Schwartz class.} $\beta\to\infty$ in \cite[Theorem 6.3]{bdrz1}, thereby obtaining
\begin{align}\label{voronoi sigma_s(n)}
	\sum_{n=1}^{\infty}\sigma_{-s}(n)\phi(n)&=-\frac{1}{2}\phi(0+)\z(s)+\int_{0}^{\beta}(\zeta(1+s)+t^{-s}\zeta(1-s))\phi(t)\, dt\nonumber\\
	&\quad+2\pi\sum_{n=1}^{\infty}\sigma_{-s}(n)n^{\frac{s}{2}}\int_{0}^{\beta}t^{-\frac{s}{2}}\phi(t)
	\bigg\{\left(\frac{2}{\pi}K_{s}(4\pi\sqrt{nt})-Y_{s}(4\pi\sqrt{nt})\right)\nonumber\\
	&\quad\quad\quad\quad\quad\quad\times\cos\left(\frac{\pi s}{2}\right)-J_{s}(4\pi\sqrt{n t})\sin\left(\frac{\pi s}{2}\right)\bigg\}\, dt.
\end{align}
where $-1/2<\textup{Re}(s)<1/2$ and $\phi$ satisfies certain conditions, for example, see \cite[Equation (6.9)]{bdrz1}. This formula, in turn, was recently generalized in \cite[Theorem 2.4]{dmv} for $\sigma_z^{(k)}(n):=\sum_{d^k|n}d^z$, where $k\in\mathbb{N}$, and where $\phi$ is a function belonging to the Schwartz class.

Besides playing a vital role in improving the error estimates in the lattice point problems such as the Dirichlet divisor problem, the Vorono\"{\dotlessi} summation formulas \eqref{voronoi d(n)} and \eqref{voronoi sigma_s(n)} are also useful in obtaining some important modular transformations. For example, the special case of \eqref{voronoi sigma_s(n)} when $\phi(x)=e^{-xy}, \textup{Re}(y)>0$, and $s$ is an odd integer, encapsulates the modular transformations satisfied by Eisenstein series on $\textup{SL}_2(\mathbb{Z})$ and their Eichler integrals; see \cite{dkk} for more details. Also, the case $\phi(x)=e^{-xy}$ of \eqref{voronoi d(n)} is used in the study of moments of the Riemann zeta function $\zeta(s)$; see \cite{betcon} for the same.

A formula of the type \eqref{voronoi d(n)} or \eqref{voronoi sigma_s(n)} for the series $\sum_{n=1}^{\infty}a(n)\phi(n)$ (or, for that matter, for $\sum_{n\leq x}a(n)\phi(n)$), where $a(n)$ is some arithmetic function, $\phi$ satisfies certain hypotheses, and whose right-hand side involves an infinite series of an  integral transform of $\phi$ is called the \emph{Vorono\"{\dotlessi} summation formula for $a(n)$}. Today such formulas are known to exist for a large class of arithmetic functions. The reader is referred to the excellent survey \cite{miller4} on this topic as well as for a recent survey  \cite{bdgz}. 

Among other things, this paper aims at obtaining Vorono\"{\dotlessi} summation formula for an arithmetic function whose Dirichlet series evaluates to a quotient (or reciprocal) of the Riemann zeta function in a certain half-plane. It is clear that the non-trivial zeros of the zeta function ought to then play a vital role. In Theorems \ref{vsf dn squared}, and \ref{vsf lambda(n)}, we obtain such formulas for the square of the divisor function $d^{2}(n)$, the M\"{o}bius function $\mu(n)$ and the Liouville function $\lambda(n)$ respectively. Let $n=p_1^{a_1}p_2^{a_2}\cdots p_k^{a_k}$, where $p_i, 1\leq i\leq k$ are distinct primes. Then the latter two functions are defined by
\begin{align*}
\mu(n)&:=\begin{cases}
	1,\hspace{8mm}\text{if}\hspace{1mm}n=1,\\
	(-1)^k,\hspace{0.5mm}\text{if} \hspace{1mm}a_1=\cdots=a_k=1,\\
	0,\hspace{8mm}\text{else},
\end{cases}	\\
 \lambda(n)&:=(-1)^{a_1+a_2+\cdots+a_k}.
\end{align*}
For Re$(s)>1$, the Dirichlet series of $\lambda(n), \mu(n)$ and $d^2(n)$ are respectively given by
	\begin{align}\label{lambda ds}
	\sum_{n=1}^{\infty}\frac{\lambda(n)}{n^s}=\frac{\zeta(2s)}{\zeta(s)},\hspace{6mm}
		\sum_{n=1}^{\infty}\frac{\mu(n)(n)}{n^s}=\frac{1}{\zeta(s)},
	\hspace{6mm}
	\sum_{n=1}^{\infty}\frac{d^2(n)}{n^s}=\frac{\zeta^4(s)}{\zeta(2s)}.
\end{align}	
The formula for $d^{2}(n)$ respects the operation of squaring in the sense that the kernel associated to the integral transform of the test function $\phi$ contains two copies of $\frac{2}{\pi}K_0\big(4\sqrt{x}\big)-Y_0\big(4\sqrt{x})$ in its triple integral representation; see Theorem \ref{vsf dn squared} below. The corresponding formulas for $\lambda(n)$ and $\mu(n)$ given in Theorems \ref{vsf lambda(n)} and \ref{voronoi mu(n)} have simple kernels involving the sine function. 

As far as the M\"{o}bius function is concerned, certain identities of such kind are known. For example. let  $\tilde{\phi}(x)$ and $\tilde{\psi}(x)$ be pair of reciprocal functions in the cosine kernel, that is,
\begin{align}\label{bef nmt}
	\tilde{\psi}(x)=\frac{2}{\sqrt{\pi}}\int_{0}^{\infty}\tilde{\phi}(u)\cos(2ux)\dd u,\hspace{5mm} \tilde{\phi}(x)=\frac{2}{\sqrt{\pi}}\int_{0}^{\infty}\tilde{\psi}(u)\cos(2ux)\dd u.
\end{align}
 Let
 \begin{align}\label{nmt}
 Z_1(s)=\frac{1}{\Gamma(s)}\int_{0}^{\infty}x^{s-1}\tilde{\phi}(x)\dd x,\hspace{5mm} Z_2(s)=\frac{1}{\Gamma(s)}\int_{0}^{\infty}x^{s-1}\tilde{\psi}(x)\dd x
 \end{align}
be the normalized Mellin transforms of $\tilde{\phi}$ and $\tilde{\psi}$ respectively. Then an identity indicated by Ramanujan to Hardy and Littlewood \cite[p.~160, Equation (2.535)]{hl} reads
\begin{align}\label{mr gen}
\sqrt{\a}\sum_{n=1}^{\infty}\frac{\mu(n)}{n}\tilde{\phi}\left(\frac{\a}{n}\right)-\sqrt{\b}\sum_{n=1}^{\infty}\frac{\mu(n)}{n}\tilde{\psi}\left(\frac{\b}{n}\right)&=\frac{1}{\sqrt{\a}}\sum_{\rho}\frac{\Gamma(1-\rho)Z_1(1-\rho)\a^{\rho}}{\zeta'(\rho)}\nonumber\\
&=	-\frac{1}{\sqrt{\b}}\sum_{\rho}\frac{\Gamma(1-\rho)Z_2(1-\rho)\b^{\rho}}{\zeta'(\rho)}	,
\end{align}
where $\alpha\beta=\pi$, and where $\rho$ denotes a non-trivial zero of $\zeta(s)$. The special case $\tilde{\phi}(x)=\tilde{\psi}(x)=e^{-x^2}$, again due to Ramanujan, and discussed in \cite[p.~158, Equation (2.516)]{hl}, which has received a renewed attention in recent years  \cite{dixthet}, \cite{kuhn-robles-roy}, \cite{agm}, \cite{chirre-gonek}, \cite{gupta-vatwani} is
\begin{align*}
	\sqrt{\alpha}\sum_{n=1}^{\infty}\frac{\mu(n)}{n}e^{-\pi\alpha^2/n^2}-\sqrt{\beta}\sum_{n=1}^{\infty}\frac{\mu(n)}{n}e^{-\pi\beta^2/n^2}
	=\frac{1}{2\sqrt{\a}}\sum_{\rho}\frac{\Gamma((1-\rho)/2)}{\zeta^{'}(\rho)}\a^{\rho}\nonumber\\
	=-\frac{1}{2\sqrt{\beta}}\sum_{\rho}\frac{\Gamma((1-\rho)/2)}{\zeta^{'}(\rho)}\beta^{\rho}.
\end{align*}
Here, the series over the non-trivial zeros are not yet known to be convergent in the usual sense. For $m\in\mathbb{Z}$, let $\rho_m:=\beta_m+i\gamma_m$ denote the $m^{\textup{th}}$ non-trivial zero of $\zeta(s)$, where $\rho_{-m}=\beta_m-i\gamma_m$. If we bracket the terms of the series in such a way that the terms for which 
\begin{equation*}
	|\gamma_m-\gamma'_m|<\exp\left(-c\hspace{0.5mm}\gamma_m/\log(\gamma_m)\right)+\exp\left(-c\hspace{0.5mm}\gamma'_m/\log(\gamma'_m)\right),
\end{equation*}
where $c>0$, are included in the same bracket, then the series converges; see \cite[p.~220]{titch}. It is believed that the series are not merely convergent but rather rapidly convergent and so bracketing may not be required. However, this is unproven as of now even upon the assumption of the Riemann Hypothesis (RH). The bracketing requirement can be expressed by rephrasing the series $\sum_{\rho}\frac{\Gamma((1-\rho)/2)}{\zeta^{'}(\rho)}\a^{\rho}$ in the form $\lim_{T_n\to\infty}\sum_{|\gamma_m|<T_n}\frac{\Gamma((1-\rho_m)/2)}{\zeta^{'}(\rho_m)}\a^{\rho_m}$, where, $\{T_n\}$ is a  sequence tending to infinity such that $|T_n-\gamma_m|>\exp{(-A\gamma_m/\log(\gamma_m))}$ for every ordinate $\gamma_m$ of a zero of $s$ \cite[p.~219]{titch}.

The Vorono\"{\dotlessi} summation formula for $\mu(n)$ that we obtain in Theorem \ref{voronoi mu(n)} gives \eqref{mr gen} (and hence \eqref{mr}) as special cases. This formula as well as most of our results in this paper involve infinite series over the non-trivial zeros of $\zeta(s)$ for which we use the same notation, that is, $\lim_{T_n\to\infty}\sum_{|\gamma_m|<T_n}$ to indicate bracketing with $\{T_n\}$ being a certain sequence.

While both the infinite series occurring  \eqref{voronoi d(n)} contain the same arithmetic function $d(n)$ (and likewise, $\sigma_{-s}(n)$ in the case of  \eqref{voronoi sigma_s(n)}), this is not true in general. In the case of $\lambda(n)$,  the corresponding function we get is the function $c(n)$ defined by
\begin{align}\label{c(n) ds}
	\sum_{n=1}^{\infty}\frac{c(n)}{n^{s}}:=\frac{\zeta(2s-1)}{\zeta(s)}
\end{align}
where Re$(s)>1$. The series is absolutely convergent in this half-plane. It is easy to see that
 \begin{align}\label{c(n) defn}
 	c(n)=m\mu(k)\hspace{1mm}\text{if}\hspace{1mm}n=m^2k,\hspace{1mm}\text{where}\hspace{1mm}k\hspace{1mm}\text{is squarefree}.
 \end{align}
Moreover, since $|\mu(k)|=\mu^{2}(k)$, in the same half-plane, we have
\begin{align}\label{abs c(n) ds}
\sum_{n=1}^{\infty}\frac{|c(n)|}{n^{s}}=\sum_{m=1}^{\infty}\frac{m\cdot\mu^{2}(k)}{(m^{2}k)^s}=\zeta(2s-1)\prod_{p\hspace{1mm}\text{prime}}\left(1+p^{-s}\right)\
=\frac{\zeta(s)\zeta(2s-1)}{\zeta(2s)}.
\end{align}
In the course of proving the Vorono\"{\dotlessi} summation formula for $\lambda(n)$, that is, Theorem \ref{vsf lambda(n)}, we need to evaluate certain infinite series involving $c(n)$. For example, we prove that
\begin{align}\label{series involving c(n)}
\sum_{n=1}^{\infty}\frac{c(n)}{n}=\frac{1}{2},\hspace{5mm}\sum_{n=1}^{\infty}\frac{c(n)\log n}{n}=-\frac{\gamma}{2},
\end{align}
An interesting feature here is that one needs  to make use of the Vinogradov-Korobov  zero-free region for $\zeta(s)$ \cite{vinogradov}, \cite{korobov} to evaluate them. The standard zero-free region is not sufficient. See Theorems \ref{c(n) over n thm} and \ref{c(n) log(n) over n thm}. Note that both series in \eqref{series involving c(n)} are conditionally convergent. Several interesting corollaries of the Vorono\"{\dotlessi} summation formula give interesting transformations for series involving $c(n)$ and $\lambda(n)$. See Corollaries \ref{lambda simple pole} - \ref{lambda riesz}.

 Also, in the expression for $\sum_{n=1}^{\infty}d^2(n)\phi(n)$ that we obtain in the Vorono\"{\dotlessi} summation formula for $d^2(n)$ that is, Theorem \ref{vsf dn squared}, we encounter another arithmetic function $b(n)$ which is defined by $$\sum_{n=1}^{\infty}\frac{b(n)}{n^{s}}:=\frac{\zeta^4(s)}{\zeta(2s-1)}\hspace{8mm}(\textup{Re}(s)>1).$$ As shown in Lemma \ref{bn lemma}, 
 $b(n)$ is multiplicative and for $n=p^k$, where  $p$ is a prime and $k\geq0$,
 \begin{equation}\label{bn exp}
 b(p^k)=\binom{k+3}{3}-p\binom{k+1}{3}.
\end{equation}
In his seminal work, Vorono\"{\dotlessi} obtained the identity
\cite[Equation (5), (6)]{voronoi} (see also \cite[p.~254]{lnb})
\begin{equation}\label{voronoi dn}
	2\sum_{n=1}^{\infty}d(n)K_{0}(4\pi\sqrt{nx})=\frac{x}{\pi^2}\sum_{n=1}^{\infty}\frac{d(n)\log(x/n)}{x^2-n^2}-\frac{\gamma}{2}-\left(\frac{1}{4}+\frac{1}{4\pi^2x}\right)\log(x)-\frac{\log(2\pi)}{2\pi^2x},
\end{equation}
where $|\arg(x)|<\pi$. It can be derived from \eqref{voronoi d(n)} by letting $\phi(t)=K_0(4\pi\sqrt{tx})$. An application of this identity is in obtaining the following identity of Koshliakov \cite[Equation (5)]{koshliakov}, which is, in turn, used to derive a simpler proof of the Vorono\"{\dotlessi} summation formula for $d(n)$ \cite{koshliakov}:
 \begin{equation}\label{obe}
	2\sum_{n=1}^{\infty}d(n)\left(K_{0}\left(4\pi e^{\frac{i\pi}{4}}\sqrt{nx}\right)+K_{0}\left(4\pi e^{-\frac{i\pi}{4}}\sqrt{nx}\right)\right)=-\gamma-\frac{1}{2}\log x-\frac{1}{4\pi x}+\frac{x}{\pi}\sum_{n=1}^{\infty}\frac{d(n)}{x^2+n^2},
\end{equation}
where $x\in\mathbb{C}\backslash\{z\in\mathbb{C}: \textup{Re}(z)=0\}$. Soni \cite{soni} proved that both \eqref{voronoi dn} and \eqref{obe} are equivalent to \eqref{voronoi d(n)}.
Another important identity associated with $d(n)$, and termed as \emph{Koshliakov's formula} \cite{bls} (see \cite{koshliakov}), is
\begin{equation}\label{koshfor}
	\gamma-\log\left(\df{4\pi}{x}\right)+4\sum_{n=1}^{\infty}d(n)K_0(2\pi
	nx)=\df{1}{x}\left(\gamma-\log(4\pi x)+4\sum_{n=1}^{\infty}d(n)K_0\left(\df{2\pi n}{x}\right)\right),
\end{equation}
where Re$(x)>0$. Observe that the arguments of the modified Bessel function occurring in the series on the left-hand sides of \eqref{voronoi dn} and \eqref{koshfor} involve $\sqrt{x}$ and $x$ respectively. It is known \cite{soni} that \eqref{koshfor} is also equivalent to \eqref{voronoi d(n)}. 

A generalization of \eqref{voronoi dn} in the setting of $\sigma_{-s}(n)$ was given by Cohen \cite[Theorem 3.4]{cohen} who proved that for $x>0$ and $s\notin\mathbb{Z}$, where $\textup{Re}(s)\geq 0$ \footnote{As
	mentioned in \cite{cohen}, the condition $\sigma\geq 0$ is not restrictive
	since $K_{-s}(w)=K_s(w)$ implies that the left-hand side of \eqref{cohenres}
	is invariant if we replace $s$ by $-s$.}, and $k\in\mathbb{Z}$ such that $k\geq\lfloor\left(\textup{Re}(s)+1\right)/2\rfloor$,
{\allowdisplaybreaks\begin{align}\label{cohenres}
		&8\pi x^{s/2}\sum_{n=1}^{\infty}\sigma_{-s}(n)n^{s/2}K_{s}(4\pi\sqrt{n x})=
		A(s, x)\zeta(s)+B(s, x)\zeta(s+1)\nonumber\\
		&\quad+\frac{2}{\sin\left(\pi s/2\right)}\left(\sum_{1\leq j\leq k}\zeta(2j)\zeta(2j-s)x^{2j-1}+x^{2k+1}\sum_{n=1}^{\infty}\sigma_{-s}(n)\frac{n^{s-2k}-x^{s-2k}}{n^2-x^2}\right),
\end{align}}%
where
\begin{align*}
	A(s, x)=\frac{x^{s-1}}{\sin\left(\pi s/2\right)}-(2\pi)^{1-s}\Gamma(s),\hspace{5mm}
	B(s, x)=\frac{2}{x}(2\pi)^{-s-1}\Gamma(s+1)-\frac{\pi x^{s}}{\cos\left(\pi s/2\right)}.
\end{align*}
Equation \eqref{obe} can also be generalized in this setting; see \cite[p.~844, Equation (7.3)]{bdrz1}.

Lastly, \eqref{koshfor} can also be generalized to get what is known as the \emph{Ramanujan-Guinand} formula \cite[p.~253]{lnb}, \cite{guinand}, \cite[Theorem 1.2]{transf}, namely, for Re$(x)>0$,
\begin{align}\label{mainagain}
&(\pi x)^{-s/2}\Gamma\left(\frac{s}{2}\right)\zeta(s)+(\pi x)^{s/2}\Gamma\left(\frac{-s}{2}\right)\zeta(-s)+4\sum_{n=1}^{\infty}\sigma_{-s}(n)n^{s/2}K_{s/2}(2\pi nx)\nonumber\\
&=\frac{1}{x}\left((\pi /x)^{-s/2}\Gamma\left(\frac{s}{2}\right)\zeta(s)+(\pi/ x)^{s/2}\Gamma\left(\frac{-s}{2}\right)\zeta(-s)+4\sum_{n=1}^{\infty}\sigma_{-s}(n)n^{s/2}K_{s/2}\left(\frac{2\pi n}{x}\right)\right).
\end{align}
This has a connection with the Fourier expansion of non-holomorphic Eisenstein series on $\textup{SL}_2\left(\mathbb{Z}\right)$ and its functional equation; see \cite[p.~60]{cohen}.
Note that Koshliakov's formula \eqref{koshfor} can be viewed as the Ramanujan-Guinand formula corresponding to $d(n)$.

One may wonder how the underlying strategy for deriving the Vorono\"{\dotlessi}-Ramanujan-Cohen identity \eqref{voronoi dn} or Cohen's identity \eqref{cohenres} differs from that used for proving Koshliakov's formula \eqref{koshfor} or the Ramanujan-Guinand formula \eqref{mainagain} as all of the identities in these formulas involve infinite series of the modified Bessel function but with different arguments. The answer is, \eqref{voronoi dn} or \eqref{cohenres} entail the asymmetric form of the functional equation of $\zeta(w)$ in their derivations, that is, \cite[p.~24]{titch},
\begin{equation}\label{zetafe asym}	
	\Gamma(w)\zeta(w)=\frac{1}{2}(2\pi)^w\zeta(1-w)\sec\left(\frac{\pi w}{2}\right)
\end{equation}
whereas, \eqref{koshfor} or \eqref{mainagain} require the symmetric form of the functional equation, namely,
\begin{align}\label{zetafe sym}
\pi^{-\frac{w}{2}}\Gamma\left(\frac{w}{2}\right)\zeta(w)=\pi^{-\frac{(1-w)}{2}}\Gamma\left(\frac{1-w}{2}\right)\zeta(1-w).
\end{align}
For example, using \cite[p.~115, Formula 11.1]{ober}, for $c=\textup{Re}(z)>\pm\frac{1}{2}\textup{Re}(s)$,
\begin{equation*}
	\frac{1}{2\pi i}\int_{(c)}\Gamma\left(z-\frac{s}{2}\right)\Gamma\left(z+\frac{s}{2}\right)(4\pi^2nx)^{-z}\dd z=2K_s(4\pi\sqrt{nx}),
\end{equation*} 
(where, here, and throughout the paper, by $\int_{(c)}$, we mean $\int_{c-i\infty}^{c+i\infty}$), we see that
\begin{align*}
\sum_{n=1}^{\infty}\sigma_{-s}(n)n^{s/2}K_{s}(4\pi\sqrt{n x})&=\frac{1}{4\pi i}\int_{(c)}\Gamma\left(z-\frac{s}{2}\right)\Gamma\left(z+\frac{s}{2}\right)\left(\sum_{n=1}^{\infty}\frac{\sigma_{-s}(n)n^{s/2}}{n^{z}}\right)(4\pi^2x)^{-z}\dd z\nonumber\\
&=\frac{1}{4\pi i}\int_{(c)}\Gamma\left(z-\frac{s}{2}\right)\Gamma\left(z+\frac{s}{2}\right)\zeta\left(z-\frac{s}{2}\right)\zeta\left(z+\frac{s}{2}\right)(4\pi^2x)^{-z}\dd z,
\end{align*}
provided, we now take $c=\textup{Re}(z)>1\pm\frac{1}{2}\textup{Re}(s)$. In the last step, we used the well-known result 
\begin{equation*}
	\sum_{n=1}^{\infty}\frac{\sigma_{-s}(n)n^{s/2}}{n^{z}}=\zeta\left(z-\frac{s}{2}\right)\zeta\left(z+\frac{s}{2}\right).
\end{equation*}
The last integral can now be transformed using \eqref{zetafe asym}. But if we consider the series occurring on the left-hand side of \eqref{mainagain}, then performing a calculation similar to the above, for $c=\textup{Re}(z)>\pm\frac{1}{2}\textup{Re}(s)$,
\begin{align*}
\sum_{n=1}^{\infty}\sigma_{-s}(n)n^{s/2}K_{s/2}(2\pi nx)&=\frac{1}{8\pi i}\int_{(c)}\Gamma\left(\frac{z}{2}-\frac{s}{4}\right)\zeta\left(z-\frac{s}{2}\right)\Gamma\left(\frac{z}{2}+\frac{s}{4}\right)\zeta\left(z+\frac{s}{2}\right)(\pi x)^{-z}\dd z,
\end{align*}
which needs \eqref{zetafe sym} for simplifying the integrand.

With this mind, one may obtain new identities of these kinds by working with different arithmetic functions. Following the terminology in \cite{banerjee-maji} and \cite{banerjee-khurana}, we call identities of the type \eqref{voronoi dn} or \eqref{voronoi sigma_s(n)} as \emph{Cohen type} identities. Also, we term those of the type \eqref{koshfor} or \eqref{mainagain} as \emph{Ramanujan-Guinand type} identities. 

Our second goal of this paper is to obtain Cohen type and Ramanujan-Guinand type identities, whenever possible, for functions whose Dirichlet series can be represented as quotients of Riemann zeta function. While we again concentrate on $\lambda(n)$ and $d^2(n)$, we also obtain such identities for $\sigma_a(n)\sigma_b(n)$. These can be seen in Theorems \ref{cohen lambda thm}, \ref{rg liouville}, \ref{voronoi mu(n)}, \ref{cohen sigma ab}, \ref{rg sigma ab} and corollaries \ref{cohen d(n) squared} and \ref{rg d(n) squared cor}. In Remark \ref{not possible}, we explain why one cannot obtain such identities for $\mu(n)$. 

We emphasize that there are very few studies on infinite series involving $\sigma_a(n)\sigma_b(n)$. Wigert \cite{wigert-zeros} undertook a brief study of $\sum_{n=1}^{\infty}\sigma_{-1}(n)\sigma_{-\alpha}(n)e^{-nx}$, where $\alpha>3$, and obtained a transformation for it \cite[Equation (8)]{wigert-zeros}. In their recent work in string theory, Dorigoni and Treilis \cite[Equation (4.29)]{dorigoni-treilis} encountered the series $\sum_{n=1}^{\infty}\sigma_{a}(n)\sigma_{b}(n)n^{-(a+b)/2}K_{a/2}(2nx)K_{b/2}(2nx)$, which is the exact series that we transform thereby obtaining the Ramanujan-Guinand identity for $\sigma_a(n)\sigma_b(n)$. Indeed, the Laplace equation considered by Dorigoni and Treilis for which the exact solution is the non-perturbative part of the Fourier zero-mode sector has precisely this series as its inhomogeneous part \cite[pp.~36-37]{dorigoni-treilis}. This showcases the importance of such series not only in number theory but also in physics. It is our hope that other series such as the ones in \eqref{cohen sigma ab} will eventually find their use in similar analyses.

We now come to the third and the last goal of our paper, namely, obtaining results on oscillations of certain weighted sums involving the arithmetic functions $\mu(n), \lambda(n)$, and of the error term corresponding to the Riesz-type sum associated to $d^{2}(n)$.

Historically, people have been interested in studying oscillations of functions associated to $ \mu(n), \lambda(n)$ and $d(n)$, namely, $\frac{1}{\sqrt{x}}\sum_{n\leq x}\mu(n)$, $\frac{1}{\sqrt{x}}\sum_{n\leq x}\lambda(n)$ and $\frac{1}{x^{1/4}}\left(\sum_{n\leq x}d(n)-x\log(x)-(2\gamma-1)x\right)$ as $x\to\infty$. For example, Ingham \cite{ingham-ajm} showed that 
\begin{align*}
\underline{\lim}_{x\to\infty}\frac{\sum_{n\leq x}\mu(n)}{\sqrt{x}}=-\infty\hspace{4mm}\text{and}\hspace{4mm}\overline{\lim}_{x\to\infty}\frac{\sum_{n\leq x}\mu(n)}{\sqrt{x}}=\infty
\end{align*}
assuming RH and the \emph{Linear Independence Conjecture} (LI) which states that the positive imaginary ordinates of the zeros of $\zeta(s)$ are linearly independent over $\mathbb{Q}$.

Using a method similar to the one that we have used for deriving our Vorono\"{\dotlessi} summation formulas, we obtain oscillation results on the weighted sums of $\mu(n)$ and $\lambda(n)$ where the weights are functions of Riesz and Schwartz type. For the corresponding Riesz type sum on $d^2(n)$, we obtain a result on the oscillation of its error term. These results are stated and proved in Section \ref{sign changes}.
	
\section{Main results}\label{mr}

\subsection{Results on the Liouville function $\lambda(n)$ }\label{lambda results}

 Let $\Phi(s)$ be holomorphic in $-1<\Re s<2$ except for a possible pole of order two or less at $s=0$, and such that 
 \begin{align}\label{delta defn}
 	\Phi(\sigma+it)\ll t^{-1-\delta}\hspace{1mm}\text{as}\hspace{1mm}t\to \pm\infty
 \end{align}
for some $\delta>0$ and $-1<\sigma<2$. 

\begin{theorem}\label{vsf lambda(n)}
	Let $\lambda(n)$ denote the Liouville function and recall the expression for $c(n)$ from \eqref{c(n) defn}. For $0<\sigma<2$, define $\phi(x):=\frac{1}{2\pi i}\int_{\sigma-i\infty}^{\sigma+i\infty}\Phi(s)x^{-s}\dd s$, where $\Phi(s)$ satisfies \eqref{delta defn} and the condition before \eqref{delta defn}. For $m\in\mathbb{Z}$, let $\rho_m:=\beta_m+i\gamma_m$ denote the $m^{\textup{th}}$ non-trivial zero of $\zeta(s)$, where $\rho_{-m}=\beta_m-i\gamma_m$. Assume that the non-trivial zeros of $\zeta(s)$ are simple.  Then there exists a sequence of numbers $\{T_n\}_{n=1}^\infty$ with $T_n\to\infty$ such that
	\begin{equation}\label{vsf lambda eqn}
		\begin{split}
			\sum_{n=1}^{\infty}\lambda(n)\phi(n)&=\frac{1}{2\zeta(\frac{1}{2})}\int_{0}^{\infty}\frac{\phi(x)}{\sqrt{x}}\dd x+ \lim_{T_n\to\infty}\sum_{|\gamma_m|<T_n}\frac{\zeta(2\rho_m)}{\zeta'(\rho_m)}\int_{0}^{\infty}\phi(x)x^{\rho_m-1}\dd x\\
			&\quad+\sqrt{2}\sum_{n=1}^{\infty}\frac{c(n)}{\sqrt{n}}\int_{0}^{\infty}\frac{\phi(x)}{\sqrt{x}}\sin\bigg(\frac{\pi nx}{2}+\frac{\pi}{4}\bigg)\dd x.
			\end{split}
		\end{equation}
		\end{theorem}
\begin{remark}
We allow $\Phi(s)$ to have a pole at $s=0$ in order to accommodate well-known functions $\phi$ belonging to Schwartz class such as $e^{-x}$ and $K_0(x)$ as well as the functions used to construct Riesz sums.

The condition that $\Phi(s)$ has at most a double pole at $s=0$ can, of course, be relaxed to have a pole of any order at $0$, and the result should remain as above. However, the calculations for justifying this become increasingly complicated and hence we refrain from doing so. 
\end{remark}	
	\begin{corollary}\label{lambda simple pole}
		Let $y>0$. Under the assumptions of Theorem \ref{vsf lambda(n)}, we have
		\begin{align}\label{lambda simple pole eqn}
		\sum_{n=1}^{\infty}\lambda(n)e^{-ny}&=
		\frac{\sqrt{\pi}}{2\sqrt{y}\zeta\left(\frac{1}{2}\right)}+ \lim_{T_n\to\infty}\sum_{|\gamma_m|<T_n}\frac{\zeta(2\rho_m)\Gamma(\rho_m)}{\zeta'(\rho_m)}y^{-\rho_m}\nonumber\\
		&\quad+\sqrt{\pi}\sum_{n=1}^{\infty}\frac{c(n)}{\sqrt{n}}\frac{\sqrt{\sqrt{4y^2+\pi^2n^2}-2y}+\sqrt{\sqrt{4y^2+\pi^2n^2}+2y}}{\sqrt{4y^2+\pi^2n^2}}.
		\end{align}
		In addition, if we assume RH and the absolute convergence of the above series over the non-trivial zeros of $\zeta(s)$, then, as $y\to0^{+}$,
		\begin{align}\label{lambda simple pole bound}
	\sum_{n=1}^{\infty}\lambda(n)e^{-ny}=O\left(\frac{1}{\sqrt{y}}\right).
		\end{align} 
	\end{corollary}

\begin{corollary}\label{lambda gaussian}
	Let $I_{s}(\xi)$ be defined in \eqref{besseli} and let $y>0$. Under the assumptions of Theorem \ref{vsf lambda(n)}, we have
		\begin{align}\label{lambda gaussian eqn}
		\sum_{n=1}^{\infty}\lambda(n)e^{-n^2y}&=
		\frac{\Gamma\left(\frac{5}{4}\right)}{y^{\frac{1}{4}}\zeta\left(\frac{1}{2}\right)}+ \frac{1}{2}\lim_{T_n\to\infty}\sum_{|\gamma_m|<T_n}\frac{\zeta(2\rho_m)}{\zeta'(\rho_m)}\Gamma\left(\frac{\rho_m}{2}\right)y^{-\frac{\rho_m}{2}}\nonumber\\
		&\quad+\frac{\pi^{3/2}}{4\sqrt{y}}\sum_{n=1}^{\infty}c(n)e^{-\frac{\pi^2n^2}{32y}}\left(I_{-\frac{1}{4}}\left(\frac{\pi^2n^2}{32y}\right)+I_{\frac{1}{4}}\left(\frac{\pi^2n^2}{32y}\right)\right).
		\end{align}
			In addition, if we assume RH and the absolute convergence of the above series over the non-trivial zeros of $\zeta(s)$, then, as $y\to0^{+}$,
		\begin{align}\label{gaussian bound}
			\sum_{n=1}^{\infty}\lambda(n)e^{-n^2y}=O\left(\frac{1}{y^{1/4}}\right).
		\end{align} 
\end{corollary}

\begin{corollary}\label{lambda hypergeometric}
Let $K_{s}(\xi)$ be defined in \eqref{knus} and let $y>0$. Under the assumptions of Theorem \ref{vsf lambda(n)}, we have
\begin{align}\label{lambda hypergeometric eqn}
\sum_{n=1}^{\infty}\lambda(n)K_0(ny)&=\frac{2\sqrt{2}}{\sqrt{y}\zeta\left(\frac{1}{2}\right)}\Gamma^{2}\left(\frac{5}{4}\right)+ \frac{1}{4}\lim_{T_n\to\infty}\sum_{|\gamma_m|<T_n}\frac{\zeta(2\rho_m)}{\zeta'(\rho_m)}\Gamma^{2}\left(\frac{\rho_m}{2}\right)\left(\frac{2}{y}\right)^{\rho_m}\nonumber\\
&\quad+\pi^{\frac{3}{2}}\sum_{n=1}^{\infty}\frac{c(n)}{\sqrt{n}(\pi^2n^2+4y^2)^{\frac{1}{4}}}{}_2F_{1}\left(\frac{1}{2},\frac{1}{2};1;\frac{1}{2}+\frac{\pi n}{2\sqrt{\pi^2n^2+4y^2}}\right).
\end{align}
	In addition, if we assume RH and the absolute convergence of the above series over the non-trivial zeros of $\zeta(s)$, then, as $y\to0^{+}$,
\begin{align}\label{lambda K big-O}
	\sum_{n=1}^{\infty}\lambda(n)K_0(ny)=O\left(\frac{1}{\sqrt{y}}\right).
\end{align} 
\end{corollary}

\begin{corollary}\label{lambda riesz}
	Let $J_s(\xi)$ be defined in \eqref{sumbesselj} and let $y>0$. Under the assumptions of Theorem \ref{vsf lambda(n)},
	\begin{align}\label{lambda riesz eqn}
	&\sum_{n\leq y}\lambda(n)\left(1-\tfrac{n}{y}\right)^{\frac{1}{2}}\nonumber\\
	&=\frac{\pi \sqrt{y}}{4\zeta\left(\frac{1}{2}\right)}+\frac{1}{2}\sqrt{\pi }\lim_{T_n\to\infty}\sum_{|\gamma_m|<T_n}\frac{\zeta(2\rho_m)}{\zeta'(\rho_m)}\frac{\Gamma(\rho_m)}{\Gamma\left(\rho_m+\frac{3}{2}\right)}y^{\rho_m}\nonumber\\
	&\quad+\frac{\pi \sqrt{y}}{2}\sum_{n=1}^{\infty}\frac{c(n)}{\sqrt{n}}\bigg\{J_{0}\left(\frac{\pi n y}{4}\right)\left(\sin\left(\frac{\pi n y}{4}\right)+\cos\left(\frac{\pi n y}{4}\right)\right)+J_{1}\left(\frac{\pi n y}{4}\right)\left(\sin\left(\frac{\pi n y}{4}\right)-\cos\left(\frac{\pi n y}{4}\right)\right)\bigg\}.
	\end{align}
In addition, if we assume RH and the absolute convergence of the above series over the non-trivial zeros of $\zeta(s)$, then, as $y\to\infty$,
\begin{align}\label{riesz bound}
	\sum_{n\leq y}\lambda(n)\left(1-\tfrac{n}{y}\right)^{\frac{1}{2}}=O\left(\sqrt{y}\right).
\end{align} 
\end{corollary}
\begin{remark}
There exists a transformation for the sum $\sum_{n\leq y}\lambda(n)\left(1-n/y\right)^{k}$, where $k>-1$,  which generalizes the above result as the integral inside the series on the right-hand side of \eqref{vsf lambda eqn} can be explicitly evaluated. However, the evaluation in terms of generalized hypregeometric functions is complicated  as opposed to the elegant Bessel function evaluation in \eqref{lambda riesz eqn}. Hence we refrain from giving the same.
\end{remark}

We now obtain Cohen-type identity associated with $\lambda(n)$.

\begin{theorem}\label{cohen lambda thm}
Let $x>0$ and let $c(n)$ be given in \eqref{c(n) defn}. For $m\in\mathbb{Z}$, let $\rho_m:=\beta_m+i\gamma_m$ denote the $m^{\textup{th}}$ non-trivial zero of $\zeta(s)$, where $\rho_{-m}=\beta_m-i\gamma_m$. Assume that the non-trivial zeros of $\zeta(s)$ are simple. Then there exists a sequence of numbers $\{T_n\}_{n=1}^\infty$ with $T_n\to\infty$ such that
\begin{align}\label{cohen lambda eqn}
\sum_{n=1}^{\infty}\frac{\lambda(n)}{n(x^2+n^2)}=-\pi x^{-5/2}\sum_{n=1}^{\infty}\frac{c(n)}{\sqrt{2n}}e^{-\pi n x/2}-\frac{\pi x^{-5/2}}{2\sqrt{2}\zeta\left(\frac{1}{2}\right)}+\frac{2\pi^2}{x^2}\lim_{T_n\to\infty}\sum_{|\gamma_m|<T_n}\frac{\zeta(2\rho_m-1)\Gamma(2\rho_m-1)}{\zeta'(\rho_m)\Gamma(\rho_m)(2\pi x)^{\rho_m}}.
\end{align}
\end{theorem}
\begin{remark}
	This identity should be compared with  \eqref{obe}.
\end{remark}

We conclude this subsection by stating the Ramanujan-Guinand type identity associated with $\lambda(n)$ and $c(n)$.
\begin{theorem}\label{rg liouville}
There exists a sequence of numbers $\{T_n\}_{n=1}^\infty$ with $T_n\to\infty$ such that
\begin{align}\label{rg liouville eqn}
\frac{x}{2}\sum_{n=1}^{\infty}n\lambda(n)e^{-\pi n^2x^2/4}&=\frac{1}{4\sqrt{2}x^2}\sum_{n=1}^{\infty}nc(n)e^{-\pi n^2/(8x^2)}\left(K_{1/4}\bigg(\frac{\pi n^2}{8x^2}\bigg)+K_{3/4}\bigg(\frac{\pi n^2}{8x^2}\bigg)\right)+\frac{\pi^{1/4}}{2\sqrt{x}\Gamma\left(\frac{1}{4}\right)\zeta\left(\frac{1}{2}\right)}\nonumber\\
&\quad+\lim_{T_n\to\infty}\sum_{|\gamma_m|<T_n}\frac{\zeta(2\rho_m)}{\zeta'(\rho_m)}\frac{\Gamma(\rho_m)}{\Gamma\left(\rho_m/2\right)}(\sqrt{\pi} x)^{-\rho_m}.
\end{align}
\end{theorem}

\subsection{Results on the square of the divisor function $d^{2}(n)$}\label{d(n) squared results}

Let $\Phi(s)$ be a holomorphic function in the strip\footnote{Unlike in Theorem \ref{vsf lambda(n)}, we do not allow the function $\Phi(s)$ in Theorem \ref{vsf dn squared} to have a pole at $s=0$ so as to not compromise on the elegance of the result, but, in principle, this can be done.} $-1<\Re s<2$ and such that
\begin{equation}\label{delta defn2}
\Phi(\s+it)\ll t^{-3-\delta}
\end{equation}
 as $t\to \infty$ for some $\delta>0$ and $-1<\s<2$.
\begin{theorem}\label{vsf dn squared}
	Let $d(n)$ be the divisor function and recall the expression for $b(n)$ from \eqref{bn exp}. For $-1<\sigma<2$, define $\phi(x):=\frac{1}{2\pi i}\int_{\sigma-i\infty}^{\sigma+i\infty}\Phi(s)x^{-s}\dd s$, where $\Phi(s)$ satisfies the aforementioned conditions. For $m\in\mathbb{Z}$, let $\rho_m:=\beta_m+i\gamma_m$ denote the $m^{\textup{th}}$ non-trivial zero of $\zeta(s)$, where $\rho_{-m}=\beta_m-i\gamma_m$. Assume that the non-trivial zeros of $\zeta(s)$ are simple. Then there exists a sequence of numbers $\{T_n\}_{n=1}^\infty$ with $T_n\to\infty$ such that
\begin{equation}\label{vsf dn squared eqn}
	\begin{split}
		\sum_{n=1}^{\infty}d^2(n)\phi(n)
		&=\int_{0}^{\infty}(A_0+A_1\log(x)+A_2\log^2(x)+A_3\log^3(x))\phi(x)\dd x\\
		&+\lim_{T_n\to\infty}\sum_{|\gamma_m|\leq T_n}\frac{\zeta^4(\frac{\rho_m}{2})}{2\zeta'(\rho_m)}\int_{0}^{\infty}\phi(x)x^{\frac{\rho_m}{2}-1}dx\\
		&+\frac{4}{\pi^2}\sum_{n=1}^{\infty}\frac{b(n)}{n}\int_{0}^{\infty}\int_{0}^{\infty}\int_{0}^{\infty}\bigg(\frac{2}{\pi}K_0\big(4\sqrt{x}\big)-Y_0\big(4\sqrt{x}\big)\bigg)\bigg(\frac{2}{\pi}K_0\big(4\sqrt{y}\big)-Y_0\big(4\sqrt{y}\big)\bigg)\\
		&\qquad\qquad\qquad\qquad\qquad\qquad\times\frac{\phi(z)}{z}\cos\bigg(\frac{2\sqrt{xy}}{\pi\sqrt{nz}}\bigg)
		{\dd z \dd x \dd y},
	\end{split}
\end{equation}
	where
		\begin{align}\label{A0-A3}
		A_0&=\frac{1}{\pi^8}\big(24 \gamma^3 \pi^6 -  72 \gamma\pi^6 \gamma_1 + 12 \pi^6 \gamma_2- 432 \gamma^2 \pi^4 \zeta'(2)+288 \pi^4 \gamma_1 \zeta'(2)+ 
		3456 \gamma \pi^2 \zeta'(2)^2- 
		10368 \zeta'(2)^3\nonumber\\ 
		&\quad- 
		288 \gamma \pi^4 \zeta''(2)+ 
		1728 \pi^2 \zeta'(2) \zeta''(2)-48\pi^4\zeta'''(2)\big),\nonumber\\
		A_1&=\frac{1}{\pi^8}\left(36 \gamma^2 \pi^6- 24 \pi^6 \gamma_1 - 
		288 \gamma \pi^4\zeta'(2)+ 
		864 \pi^2  \zeta'(2)^2 - 
		72 \pi^4\zeta''(2)\right),\nonumber\\
		A_2&= \frac{1}{\pi^8}\left(12 \gamma \pi^6 -
		36 \pi^4 \zeta'(2)\right),\nonumber\\
		A_3&=\frac{1}{\pi^2},
	\end{align}
	with $\gamma_1$ being the first Stieltjes constant.
\end{theorem}

\subsection{Results on the M\"{o}bius function $\mu(n)$}

Let $\Phi(s)$ be holomorphic in $-1<\Re s<2$ except for a possible simple pole at $s=0$, and such that
 \begin{equation}\label{delta mu(n)}
 \Phi(\sigma+it)\ll t^{-1-\delta}
\end{equation} 
  as $t\to \pm\infty$ for some $\delta>0$ and $-1<\delta<2$. 

\begin{theorem}\label{voronoi mu(n)}
	For $0<\sigma<2$, define $\phi(x):=\frac{1}{2\pi i}\int_{\sigma-i\infty}^{\sigma+i\infty}\Phi(s)x^{-s}\dd s$, where $\Phi(s)$ satisfies the aforementioned conditions. For $m\in\mathbb{Z}$, let $\rho_m:=\beta_m+i\gamma_m$ denote the $m^{\textup{th}}$ non-trivial zero of $\zeta(s)$, where $\rho_{-m}=\beta_m-i\gamma_m$. Assume that the non-trivial zeros of $\zeta(s)$ are simple. Then there exists a sequence of numbers $\{T_n\}_{n=1}^\infty$ with $T_n\to\infty$ such that
	\begin{equation*}
		\begin{split}
			\sum_{n=1}^{\infty}\mu(n)\phi(n)= -2k+\lim_{T_n\to \infty}\sum_{|\gamma_m|<T_n}\frac{1}{\zeta'(\rho_m)}\int_{0}^{\infty}\phi(x)x^{\rho_m-1}\dd x-4\sum_{n=1}^{\infty}\frac{\mu(n)}{n}\int_{0}^{\infty}(\phi(t)-k)\bigg(\frac{\sin^2(\pi/nt)}{t}\bigg)\dd t, \end{split}
		\end{equation*}
		where $k$ is the residue of $\Phi(s)$ at $s=0$. 
		\end{theorem}
		
\noindent
In Section \ref{mr gen section}, we briefly sketch the proof of the above theorem since it is similar to that of Theorem \ref{vsf lambda(n)}, and since the latter is proved in complete detail in Section \ref{4.2}.
\begin{corollary}\label{mr gen corollary}
Ramanujan's identity \eqref{mr gen} holds.
\end{corollary}

\begin{remark}\label{not possible}
One does not get Ramanujan-Guinand type and Cohen-type identities for $\mu(n)$ since the corresponding integrals $\frac{1}{2\pi i}\int_{(c)}\frac{x^{-s}\dd s}{\pi^{-s/2}\Gamma(s/2)\zeta(s)}$ and  $\frac{1}{2\pi i}\int_{(d)}\frac{x^{-s}\dd s}{\Gamma(s)\zeta(s)}$ diverge along any vertical line. 
\end{remark}
\subsection{Cohen and Ramanujan-Guinand type identities for $\sigma_a(n)\sigma_b(n)$}\label{sigma ab results}

Although we do not venture\footnote{In principle, this can, of course, be done, and a formula generalizing that for $d^2(n)$, that is, \eqref{vsf dn squared eqn} can be derived; however, it is quite complicated and hence we refrain from going in that direction.} into deriving the Vorono\"{\dotlessi} summation formula for $\sigma_a(n)\sigma_b(n)$, we do obtain Cohen-type identity as well the analogue of Ramanujan-Guinand identity for this arithmetic function. 

Before stating the Cohen-type identity for $\sigma_a(n)\sigma_b(n)$, we need to define an arithmetic function $C_{a, b}(n)$ occurring in our formulas below. It is defined by means of the Dirichlet series
\begin{align}\label{cab dirichlet}
\sum_{n=1}^{\infty}\frac{C_{a, b}(n)}{n^s}:=\frac{\zeta(s)\zeta(s-a)\zeta(s-b)\zeta(s-a-b)}{\zeta(2s-a-b-1)}\hspace{8mm}(\textup{Re}(s)>\eta),
\end{align}
where 
\begin{equation}\label{eta}
	\eta:=\max{\{1,1+\Re(a),1+\Re(b),1+\Re(a+b),1+\Re(a+b)/2\}}.
\end{equation}
We now give an expression for $C_{a, b}(n)$ as the Dirichlet convolution of two familiar functions. Write the right-hand side of \eqref{cab dirichlet} as
\begin{align}\label{cab dirichlet start}
\frac{\zeta(s)\zeta(s-a)\zeta(s-b)\zeta(s-a-b)}{\zeta(2s-a-b-1)}=\frac{\zeta(s)\zeta(s-a)\zeta(s-b)\zeta(s-a-b)}{\zeta(2s-a-b)}\cdot\frac{\zeta(2s-a-b)}{\zeta(2s-a-b-1)}.
\end{align}
Now for $\textup{Re}(s)>\eta$, we have \cite[p.~8, Equation (1.3.3)]{titch},
\begin{align}\label{sigma ab dirichlet}
	\frac{\zeta(s)\zeta(s-a)\zeta(s-b)\zeta(s-a-b)}{\zeta(2s-a-b)}=\sum_{n=1}^{\infty}\frac{\sigma_a(n)\sigma_b(n)}{n^s}.
\end{align}
Morever, if $\phi^{-1}$ denotes the Dirichlet inverse of the Euler totient function $\phi$, then, for Re$(s)>2$, we have
$\sum_{n=1}^{\infty}\phi^{-1}(n)n^{-s}=\zeta(s)/\zeta(s-1)$,
whence, for Re$(s)>\eta$,
\begin{align*}
\frac{\zeta(2s-a-b)}{\zeta(2s-a-b-1)}=\sum_{n=1}^{\infty}\frac{\kappa_{a,b}(n)}{n^s},
\end{align*}
where
\begin{equation}\label{cab dirichlet end}
\kappa_{a,b}(n)=\begin{cases}
	m^{a+b}\phi^{-1}(m),\text{if}\hspace{1mm}n=m^2,\\
	0,\hspace{19mm}\text{else}.
\end{cases}
\end{equation}
Thus, from \eqref{cab dirichlet start}-\eqref{cab dirichlet end}, we see that
\begin{align}\label{cab}
	C_{a,b}(n)=\left(\sigma_a\sigma_b*\kappa_{a,b}\right)(n).
\end{align}
\begin{theorem}\label{cohen sigma ab}
Let $-1<\textup{Re}(a), \textup{Re}(b), \textup{Re}(a-b), \textup{Re}(a+b)<1$ and $x>0$. Let $C_{a,b}(n)$ be given in \eqref{cab}. Let $f(s):=\Gamma(s)\zeta(s)$. For $m\in\mathbb{Z}$, let $\rho_m:=\beta_m+i\gamma_m$ denote the $m^{\textup{th}}$ non-trivial zero of $\zeta(s)$, where $\rho_{-m}=\beta_m-i\gamma_m$. Assume that the non-trivial zeros of $\zeta(s)$ are simple. Then there exists a sequence of numbers $\{T_n\}_{n=1}^\infty$ with $T_n\to\infty$ such that
\begin{align}\label{sigma ab final}
&\sum_{n=1}^{\infty}{\sigma_{a}(n)\sigma_{b}(n)}\frac{x\big(x^{-a}-n^{-a}\big)\big(x^{-b}-n^{-b}\big)}{x^2-n^2}\nonumber\\
&=32\pi x^{\frac{1-a-b}{2}} \sin\left(\frac{\pi a}{2}\right)\sin\left(\frac{\pi b}{2}\right)\sum_{n=1}^{\infty}\frac{C_{a, b}(n) }{n^{\frac{a+b-1}{2}}}\bigg(K_{a-1}\Big(4\pi\sqrt{nx}\Big)K_{b}\Big(4\pi\sqrt{nx}\Big)\nonumber\\
&\quad+K_{b-1}\Big(4\pi\sqrt{nx}\Big)K_{a}\Big(4\pi\sqrt{nx}\Big)+\frac{(a+b-1)}{4\pi\sqrt{nx}}K_{b}\Big(4\pi\sqrt{nx}\Big)K_{a}\Big(4\pi\sqrt{nx}\Big)\bigg)-2(2\pi)^{a+b}\nonumber\\
&\quad\times\sin\left(\frac{\pi a}{2}\right)\sin\left(\frac{\pi b}{2}\right)\left(\sum_{k=0}^{1}\left(R_k(x)+ R_{k+a}(x)+R_{k+b}(x)+R_{k+a+b}(x)\right)+\lim_{T_n\to\infty}\sum_{|\gamma_m|<T_n}R_{\rho_m,a,b}(x)\right),
\end{align}
where
\begin{align}\label{residues 1}
	R_0(x)&=\frac{-f(-a)f(-b)f(-a-b)}{2f(-a-b-1)},\hspace{15mm}
	R_a(x)=\frac{-f(a)f(a-b)f(-b)}{2f(a-b-1)}(4\pi^2x)^{-a},\nonumber\\
	R_b(x)&=\frac{-f(b)f(b-a)f(-a)}{2f(b-a-1)}(4\pi^2x)^{-b},\hspace{6mm}
	R_{a+b}(x)=\frac{-f(a+b)f(a)f(b)}{2f(a+b-1)}(4\pi^2x)^{-b-a},\nonumber\\
R_1(x)&={f(1-a)f(1-b)}(4\pi^2x)^{-1},\hspace{12mm}
R_{1+a}(x)=f(1+a)f(1-b)(4\pi^2x)^{-1-a},\nonumber\\
R_{1+b}(x)&=f(1+b)f(1-a)(4\pi^2x)^{-1-b},\hspace{8mm}
R_{1+a+b}(x)=f(1+a)f(1+b)(4\pi^2x)^{-1-a-b},
\end{align}
and
\begin{align}\label{residues 2}
R_{\rho_m,a,b}(x)=\frac{f((1+\rho_m+a+b)/2)f((1+\rho_m-a+b)/2)f((1+\rho_m+a-b)/2)f((1+\rho_m-a-b)/2)}{2\zeta'(\rho_m)\Gamma(\rho_m)(4\pi^2x)^{\frac{\rho_m+1+a+b}{2}}}.
\end{align}
\end{theorem}
As a corollary, we get the Cohen-type identity for $d^2(n)$.
\begin{corollary}\label{cohen d(n) squared}
	Let $x>0$. Recall the notation for the non-trivial zeros of $\zeta(s)$ from Theorem \ref{cohen sigma ab} and assume that they are simple. Let $C_{a, b}(n)$ be defined in \eqref{cab dirichlet}, or equivalently, by \eqref{cab}. Then there exists a sequence of numbers $\{T_n\}_{n=1}^\infty$ with $T_n\to\infty$ such that
\begin{align}\label{cohen square of d(n)}
\sum_{n=1}^{\infty}d^2(n)\frac{x\log^2(x/n)}{x^2-n^2}&= 8\pi^3\sqrt{x}\sum_{n=1}^{\infty}\sqrt{n}C_{0,0}(n)\bigg(2K_0(4\pi\sqrt{nx})K_1(4\pi\sqrt{nx})-\frac{K_0^2(4\pi\sqrt{nx})}{4 \pi\sqrt{nx}}\bigg)\nonumber\\
&\quad-\frac{\pi^2}{2}\left(4R_0(x)+4R_1(x)+\lim_{T_n\to\infty}\sum_{|\gamma_m|<T_n}R_{\rho_m,0,0}(x)\right),
\end{align}
where
\begin{align*}
R_{\rho_m,0,0}(x)=\frac{\Gamma^{4}\left(\frac{1+\rho_m}{2}\right)\zeta^{4}\left(\frac{1+\rho_m}{2}\right)}{2\zeta'(\rho_m)\Gamma(\rho_m)(4\pi^2x)^{\frac{\rho_m+1}{2}}},\hspace{12mm}	
R_{1}(x)=\frac{\log^{2}(4\pi^2x)}{4\pi^2x},
\end{align*}
and
\begin{align*}
R_{0}(x)&=\frac{\pi^2}{4}+3\log^{2}(2)-72\log(A)+6(\gamma+12\log(A))(12\log(A)-\log(2\pi))+\log(\pi)\log(64\pi^3)\nonumber\\
&\quad+\frac{3}{4}\log(x)\left(4\gamma+48\log(A)-4\log(2\pi)+\log(x)\right)-6\gamma_1+36\zeta''(-1),
\end{align*}
with $A$ being the Glaisher-Kinkelin constant.
\end{corollary}

\begin{remark}
The above identity should be compared with \eqref{voronoi dn}. Also, for $a=b=0$, the arithmetic function $C_{a,b}(n)$ can be written in an alternative form, namely, $C_{0,0}(n)=b(n)$, where $b(1)=1$, and for primes p,  $b(p^k)$ is defined in \eqref{bn exp}. 
\end{remark}

We now state the Ramanujan-Guinand identity for $\sigma_a(n)\sigma_b(n)$.

\begin{theorem}\label{rg sigma ab}
Let $a,b\in\mathbb{C}$ and $x>0$. Let $C_{a, b}(n)$ be defined by \eqref{cab dirichlet}. Let $g(s):=\Gamma(s/2)\zeta(s)$. For $m\in\mathbb{Z}$, let $\rho_m:=\beta_m+i\gamma_m$ denote the $m^{\textup{th}}$ non-trivial zero of $\zeta(s)$, where $\rho_{-m}=\beta_m-i\gamma_m$. Assume that the non-trivial zeros of $\zeta(s)$ are simple. Then there exists a sequence of numbers $\{T_n\}_{n=1}^\infty$ with $T_n\to\infty$ such that
\begin{align}\label{jhep}
\frac{8}{x^{\frac{a+b}{2}}}\sum_{n=1}^{\infty}\frac{\sigma_{a}(n)\sigma_{b}(n)}{n^{(a+b)/2}}K_{\frac{a}{2}}(2nx)K_{\frac{b}{2}}(2nx)
&=\frac{2^{\frac{3-a-b}{2}}}{\pi^{a+b+1}}\sum_{n=1}^{\infty}\frac{C_{-a,-b}(n)}{n}G_{4, \, \, 2}^{0, \, \, 4} \left(\begin{matrix}
	\frac{1}{2}, \frac{1-a}{2}, \frac{1-b}{2}, \frac{1-a-b}{2}\\
	\frac{1-a-b}{4}, \frac{3-a-b}{4}
\end{matrix} \Bigg| \, \frac{x^2}{4n^2\pi^4} \right)\nonumber\\
&\quad+\sum_{k=0}^{1}\left(R_k(x)+ R_{k+a}(x)+R_{k+b}(x)+R_{k+a+b}(x)\right) \nonumber\\
&\quad+\lim_{T_n\to\infty}\sum_{|\gamma_m|<T_n}R_{\frac{\rho_m+a+b}{2}}(x),
\end{align}
where $G_{4, \, \, 2}^{0, \, \, 4} $ is the Meijer $G$-function defined in \eqref{MeijerG}, and
\begin{align}\label{residues 3}
	R_0(x)&=-g(-a)g(-b),\hspace{42mm}
	R_a(x)=-g(a)g(-b)x^{-a},\nonumber\\
	R_b(x)&=-g(b)g(-a)x^{-b}\hspace{40mm}
	R_{a+b}(x)=-g(a)g(b)x^{-a-b},\nonumber\\
	R_1(x)&=\frac{\sqrt{\pi}g(1-a)g(1-b)g(1-a-b)}{g(2-a-b)x},\hspace{10mm}
	R_{1+a}(x)=\frac{\sqrt{\pi}g(1+a)g(1+a-b)g(1-b)}{g(2+a-b)x^{1+a}},\nonumber\\
	R_{1+b}(x)&=\frac{\sqrt{\pi}g(1+b)g(1+b-a)g(1-a)}{g(2+b-a)x^{1+b}},\hspace{8mm}
	R_{1+a+b}(x)=\frac{\sqrt{\pi}g(1+a+b)g(1+a)g(1+b)}{g(2+a+b)x^{1+a+b}},
\end{align}
and
\begin{align}\label{residues 4}
	R_{\frac{\rho_m+a+b}{2}}(x)=\frac{1}{2\Gamma\left(\frac{\rho_m}{2}\right)\zeta'(\rho_m)x^{\frac{\rho_m+a+b}{2}}}g\left(\frac{\rho_m+a+b}{2}\right)g\left(\frac{\rho_m-a+b}{2}\right)g\left(\frac{\rho_m+a-b}{2}\right)g\left(\frac{\rho_m-a-b}{2}\right).
\end{align}
\end{theorem}
\begin{remark}
As mentioned in the introduction, the series on the left-hand side of \eqref{jhep} has recently turned up in the work of Dorigoni and Treilis \cite{dorigoni-treilis} in string theory.  
\end{remark}

The following corollary gives the corresponding Ramanujan-Guinand identity for $d^{2}(n)$.

\begin{corollary}\label{rg d(n) squared cor}
Under the same hypotheses in Theorem \ref{rg sigma ab}, we have
\begin{align*}
&4\left(\gamma-\log\left(\tfrac{4\pi^2}{x}\right)\right)^2+8\sum_{n=1}^{\infty}d^{2}(n)K_{0}^{2}(2nx)\nonumber\\
&=\frac{2^{3/2}}{\pi}\sum_{n=1}^{\infty}\frac{C_{0,0}(n)}{n}G_{4, \, \, 2}^{0, \, \, 4} \left(\begin{matrix}
	\frac{1}{2}, \frac{1}{2}, \frac{1}{2}, \frac{1}{2}\\
	\frac{1}{4}, \frac{3}{4}
\end{matrix} \Bigg| \, \frac{x^2}{4n^2\pi^4} \right)+4\tilde{R}_1(x)+\lim_{T_n\to\infty}\sum_{|\gamma_m|<T_n}\frac{\Gamma^{4}(\rho_m/4)\zeta^4(\rho_m/2)}{2\Gamma\left(\frac{\rho_m}{2}\right)\zeta'(\rho_m)}x^{-\rho_m/2},
\end{align*}
where
$x\tilde{R}_1(x)$ is a complicated cubic polynomial in $\log(x)$. 
\end{corollary}
\begin{remark}
The left-hand side of Corollary \ref{rg d(n) squared cor} should be compared with that of \eqref{koshfor}, the corresponding formula for $d(n)$.
\end{remark}
\begin{remark}
While the Meijer $G$-function occurring in Theorem \ref{rg sigma ab} does not admit a representation in terms of well-known functions, it can be represented as a double integral involving the modified Bessel function of the second kind. Indeed, using the basic properties of the Meijer $G$-function, namely,
\begin{equation*}
\int_{0}^{1}x^{\alpha-1}(1-x)^{\beta-1}	G_{p, \, \, q}^{m, \, \, n}\left(\begin{matrix}a_1, a_2, \cdots, a_p\\b_1, b_2, \cdots, b_q
\end{matrix} \Bigg| \, zx \right)\dd x=\Gamma(\beta)G_{p+1, \, \, q+1}^{m, \, \, n+1}\left(\begin{matrix}1-\alpha, a_1, a_2, \cdots, a_p\\b_1, b_2, \cdots, b_q, 1-\alpha-\beta
\end{matrix} \Bigg| \, z \right),
\end{equation*}
and \cite[p.~621, Formula (39)]{prud3}, namely, for $m\leq q-1$,
\begin{align*}
	\frac{d}{dz}\left[z^{-b_q}G_{p, \, \, q}^{m, \, \, n} \left(\begin{matrix}
		a_1,\cdots,a_n, a_{n+1},\cdots, a_p\\
		b_1,\cdots,b_m, b_{m+1},\cdots, b_q
	\end{matrix} \Bigg| \, z \right)\right]=z^{-b_q-1}G_{p, \, \, q}^{m, \, \, n} \left(\begin{matrix}
		a_1,\cdots,a_n, a_{n+1},\cdots, a_p\\
		b_1,\cdots,b_m, b_{m+1},\cdots, b_{q-1}, b_q+1
	\end{matrix} \Bigg| \, z \right),
\end{align*}
 it is not difficult to see that
\begin{align*}
&G_{4, \, \, 2}^{0, \, \, 4} \left(\begin{matrix}
	\frac{1}{2}, \frac{1-a}{2}, \frac{1-b}{2}, \frac{1-a-b}{2}\\
	\frac{1-a-b}{4}, \frac{3-a-b}{4}
\end{matrix} \Bigg| \, y \right)=\sqrt{y}G_{4, \, \, 2}^{0, \, \, 4} \left(\begin{matrix}
0, \frac{-a}{2}, \frac{-b}{2}, \frac{-a-b}{2}\\
\frac{-a-b-1}{4}, \frac{1-a-b}{4}
\end{matrix} \Bigg| \, y \right)\nonumber\\
&=\frac{2\sin\left(\pi\left(\frac{1-a-b}{4}\right)\right)}{\pi y^{(a+b+1)/2}}\int_{0}^{1}\int_{0}^{1}\frac{(1-x)^{\frac{-3-a-b}{4}}(1-t)^{\frac{a+b-1}{4}}}{xt^{1+\frac{a+b}{2}}}\left\{\tfrac{(2a-1)}{4}K_{\frac{a-b}{2}}\left(\tfrac{2}{\sqrt{txy}}\right)+\tfrac{1}{\sqrt{txy}}K_{1+\frac{a-b}{2}}\left(\tfrac{2}{\sqrt{txy}}\right)\right\}\dd x\dd t.
\end{align*}
\end{remark}
\section{Preliminaries}\label{prelim}
Stirling's formula for the Gamma function in the vertical strip $p\leq\sigma\leq q$ is given by \cite[p.~224]{cop}
\begin{equation}\label{strivert}
	|\Gamma(\sigma+it)|=\sqrt{2\pi}|t|^{\sigma-\frac{1}{2}}e^{-\frac{1}{2}\pi |t|}\left(1+O\left(\frac{1}{|t|}\right)\right)\hspace{8mm}(|t|\to \infty)
\end{equation}
\begin{theorem}[Parseval's formula] \cite[p.~83, Equation (3.1.13)]{kp}\label{Parseval}
	Let $F(s)$ and $G(s)$ be the Mellin transforms of $f(x)$ and $g(x)$ respectively. If $F(1-s)$ and $G(s)$ have a common strip of analyticity, then for any vertical line $\textup{Re}(s)=c$ in the common strip, we have
	\begin{align}\label{Parseval1}
		\frac{1}{2 \pi i} \int_{(c)} G(s) F(1-s) ds = \int_{0}^\infty f(t) g(t) dt,
	\end{align}
	under the assumption that the integral on the right-hand side exists and the conditions 
	\begin{equation*}
		t^{c-1} g(t) \in L[0, \infty)\hspace{5mm}\text{and}\hspace{5mm}F(1-c-it) \in L(-\infty,  \infty)
	\end{equation*}
	hold. 
	\end{theorem}
	We now define the Meijer $G$-function \cite[p.~617, Definition \textbf{8.2.1}]{prud3}.
	Let $m,n,p,q$ be integers such that $0\leq m \leq q$, $0\leq n \leq p$. Let $a_1, \cdots, a_p, b_1, \cdots, b_q\in\mathbb{C}$ such that $a_i - b_j \not\in \mathbb{N}$ for $1 \leq i \leq n$ and $1 \leq j \leq m$.   Then the Meijer $G$-function is defined by 
	\begin{align}\label{MeijerG}
		G_{p,q}^{\,m,n} \!\left(  \,\begin{matrix} a_1,\cdots , a_p \\ b_1, \cdots b_m, b_{m+1}, \cdots, b_q \end{matrix} \; \Big| z   \right) := \frac{1}{2 \pi i} \int_L \frac{\prod_{j=1}^m \Gamma(b_j 
			+ w) \prod_{j=1}^n \Gamma(1 - a_j -w) z^{-w}  } {\prod_{j=m+1}^q \Gamma(1 - b_j - w) \prod_{j=n+1}^p \Gamma(a_j + w)}\, dw,
	\end{align}
	where $L$  goes from $-i \infty$ to $+i \infty$ and separates the poles of $\Gamma(1-a_j-w)$  from the poles of $\Gamma(b_j+w)$.  
	The integral converges  absolutely if $p+q  < 2(m+n)$ and $|\arg(z)| < \left(m+n - \frac{p+q}{2}\right) \pi$.  In the case $p+q  = 2(m+n)$ and $\arg(z)=0$,  the integral converges absolutely if $ \left(  \textup{Re}(w) + \frac{1}{2} \right) (q-p) > \textup{Re}(\psi ) +1$,  where $\psi = \sum_{j=1}^q b_j - \sum_{j=1}^p a_j$.

	\section{Proofs of the results on  $\lambda(n)$}\label{lambda}

Before proving Theorem \ref{vsf lambda(n)}, we derive some crucial lemmas in the following subsection. These lemmas are interesting in themselves. 
	
	\subsection{Infinite series evaluation using the Vinogradov-Korobov zero-free region}\label{c(n)}
	Recall the definition of $c(n)$ from \eqref{c(n) ds}. 
	In the Vorono\"{\dotlessi} summation formula for $\lambda(n)$ given in Theorem \ref{vsf lambda(n)}, we crucially require the exact evaluation of the Dirichlet series of $c(n)$ at $s=1$. If we let $s\to 1$ on both sides of \eqref{c(n) ds} and \textit{formally} interchange the order of limit and summation on the left-hand side, we do get a correct evaluation. However, the rigorous justification of this result is delicate and requires the use of the Vinogradov-Korobov zero-free region. This is exactly what is done next.
	\begin{theorem}\label{c(n) over n thm}
		Let $c(n)$ be given by \eqref{c(n) defn}. We have
		\begin{equation}\label{c(n) over n}
			\sum_{n=1}^{\infty}\frac{c(n)}{n}=\frac{1}{2}.
		\end{equation}
	\end{theorem}
	\begin{proof}
		From \eqref{c(n) ds}, we have
		$$\sum_{n=1}^{\infty}\frac{c(n)}{n^{1+w}}=\frac{\zeta(1+2w)}{\zeta(1+w)} $$
		for $\Re w>0.$
		From \cite[p.~140, Corollary 5.3]{montgomery-vaughan}, for $T$ large enough, $1<x<T$, and $\sigma_0>0$ (to be chosen later),
		\begin{equation}\label{ps}
			\sum_{n\leq x}\frac{c(n)}{n}=\frac{1}{2\pi \di}\int_{\sigma_0-\di T}^{\sigma_0+\di T}\frac{\zeta(1+2w)}{\zeta(1+w)}\frac{x^w}{w}\dd w+R   
		\end{equation}
		where $$R\ll\sum_{x/2<n<2x\atop n\neq x}\frac{|c(n)|}{n}\min\bigg( 1,\frac{x}{T|x-n|}\bigg)+\frac{4^{\sigma_0}+x^{\sigma_0}}{T}\sum_{n=1}^{\infty}\frac{|c(n)|}{n^{1+\sigma_0}}.$$
		Now we first approximate the error term $R$.
		Using \eqref{abs c(n) ds}, we see that $$\sum_{n=1}^{\infty}\frac{|c(n)|}{n^{1+\sigma_0}}= \frac{\zeta(1+\sigma_0)\zeta(1+2\sigma_0)}{\zeta(2+2\sigma_0)}\ll\frac{1}{\sigma^2_0}$$
		as $\sigma_0\to 0^+.$
		As $c(n)\ll \sqrt{n}$,
		\begin{equation} \label{R}
			\sum_{x/2<n<2x\atop n\neq x}\frac{|c(n)|}{n}\min\bigg( 1,\frac{x}{T|x-n|}\bigg)\ll \frac{\sqrt{x}\log x}{T}.
		\end{equation}
		Thus,
		$$
		R\ll \frac{x^{\sigma_0}}{T \sigma^2_0}+\frac{\sqrt{x}\log x}{T}.
		$$
		We take $\sigma_0=\frac{1}{\log x}$ to get
		$$R\ll \frac{\sqrt{x}\log x}{T}. $$
		Now we will compute $$ I=\int_{\sigma_0-\di T}^{\sigma_0+\di T}\frac{\zeta(1+2w)}{\zeta(1+w)}\frac{x^w}{w}\dd w.$$
		From Vinogradov-Korobov zero-free region \cite{korobov}, \cite{vinogradov}, we know that $\zeta(1+w)$ has no zeros when  $\sigma\geq-\frac{c}{\log^{2/3}(t)(\log\log(t))^{1/3}}$ for some $c>0$. (For state-of-the-art results on how large $c$ can be, the reader is referred to a recent article by Mossinghoff, Trudgian and Yang \cite{trudgian}.)
		
		We construct the contour $[\sigma_0-iT, \sigma_0+iT, \sigma_1+iT, \sigma_1-iT]$, where  $\sigma_1=-\frac{c}{\log^{2/3}(T)(\log\log(T))^{1/3}}$ so that the rectangle formed by the contour completely stays in the Vinogradov-Korobov zero-free region. Invoking the residue theorem and taking into account the simple pole of the integrand at $w=0$, we see that
		\begin{equation}\label{I}
			I=\text{Res}_{w=0}\left(\frac{\zeta(1+2w)}{\zeta(1+w)}\frac{x^w}{w}\right)+\int_{\sigma_0+\di T}^{\sigma_1+\di T}\frac{\zeta(1+2w)}{\zeta(1+w)}\frac{x^w}{w}\dd w+\int_{\sigma_1+\di T}^{\sigma_1-\di T}\frac{\zeta(1+2w)}{\zeta(1+w)}\frac{x^w}{w}\dd w+\int_{\sigma_1-\di T}^{\sigma_0-\di T}\frac{\zeta(1+2w)}{\zeta(1+w)}\frac{x^w}{w}\dd w.
		\end{equation}
		Using the Laurent series expansions of $\zeta(1+2w)$ and $\zeta(1+w)$ around $w=0$, it is easy to see that
		\begin{align}\label{res}
			\text{Res}_{w=0}\left(\frac{\zeta(1+2w)}{\zeta(1+w)}\frac{x^w}{w}\right)=\frac{1}{2}.
		\end{align}
		In the given region of integration, we have \cite[p.~136]{titch}
		\begin{align}\label{vk}
			\frac{1}{\zeta(1+\sigma+it)}\ll\log^{\frac{2}{3}}(t)(\log\log(t))^{\frac{1}{3}}
		\end{align}
		Also, from \cite[p.~98]{richert}, for $-1/2\leq\sigma\leq0, t\geq2$,
		\begin{align}\label{richert result}
			|\zeta(1+\sigma+it)|<c_1t^{100(-\sigma)^{\frac{3}{2}}}
			\log^{\frac{2}{3}}(t)
		\end{align}
		for some $c_1>0$. 
		Since $-\sigma<\frac{c}{\log^{2/3}(t)(\log\log(t))^{1/3}}$, we see that $|\zeta(1+\sigma+it)|\ll\log^{\eta}(t)$ for some $\eta>0$. Thus, combining \eqref{vk} and \eqref{richert result}, we get
		\begin{align}\label{bound}
			\frac{\zeta(1+2w)}{\zeta(1+w)}\ll \log^{\kappa} t
		\end{align}	 
		for some $\kappa>0$. 
		Hence
		$$
		\int_{\sigma_0+\di T}^{\sigma_1+\di T}\frac{\zeta(1+2w)}{\zeta(1+w)}\frac{x^w}{w}\dd w\ll \frac{\log^{\kappa} T}{T}x^{\sigma_0}(\sigma_0-\sigma_1)\ll \frac{\log^{\kappa} T}{T}.
		$$
		Similarly,
		$$\int_{\sigma_1-\di T}^{\sigma_0-\di T}\frac{\zeta(1+2w)}{\zeta(1+w)}\frac{x^w}{w}\dd w\ll \frac{\log^{\kappa} T}{T}.$$
		Now we compute
		$$
		\int_{\sigma_1-\di T}^{\sigma_1+\di T}\frac{\zeta(1+2w)}{\zeta(1+w)}\frac{x^w}{w}\dd w.
		$$
		Using \eqref{bound}, we  obtain
		\begin{equation}\label{vertical}
			\begin{split}
				\int_{\sigma_1-\di T}^{\sigma_1+\di T}\frac{\zeta(1+2w)}{\zeta(1+w)}\frac{x^w}{w}\dd w&\ll x^{\frac{-c_1}{\log^{2/3}(T)(\log\log(T))^{1/3}}}\int_{1}^{T}\log^{\kappa}t\frac{\dd t}{t}\\& \ll x^{\frac{-c_1}{\log^{2/3}(T)(\log\log(T))^{1/3}}}\log^{\kappa+1}(T)
			\end{split}
		\end{equation}
		We now choose $x=\exp(\log^{4/5}T)$.
		Then from \eqref{ps}, \eqref{I} and \eqref{res},
		\begin{equation*}
			\sum_{n\leq x}\frac{c(n)}{n}=\frac{1}{2}+O\bigg( \frac{2\log^{\kappa} T}{T}+\frac{\log^{\kappa+1}T}{\exp\left(c_1\log^{2/15}(T)(\log\log(T))^{-1/3}\right)}+\frac{\exp(0.5\log^{4/5}T)\log T}{T} \bigg).
		\end{equation*}
		Thus taking the limit $T\to \infty$, we arrive at \eqref{c(n) over n}.
	\end{proof}
	\begin{remark}
		Since $c(n)=\Omega(\sqrt{n})$, neither Ingham's theorem \cite[p.~63]{luca} nor its extended version for bounded arithmetic functions is applicable here. Hence we use the above method of evaluating the infinite series using contour integration. As can be checked, the standard zero-free region does not suffice to have the right-hand sides of  \eqref{R} and \eqref{vertical} go to zero as $T\to\infty$, and hence this is one of those rare instances where one has to resort to the Vinogradov-Korobov zero-free region. 
	\end{remark}
	\begin{theorem}\label{c(n) log(n) over n thm}
		With $c(n)$ given by \eqref{c(n) defn}, we have $$\sum_{n=1}^{\infty}\frac{c(n)\log n}{n}=-\frac{\gamma}{2}.$$   
	\end{theorem}
	\begin{proof}
		Let Re$(s)>1$. Differentiating both sides of \eqref{c(n) ds} with respect to $s$, we see that
		$$
		-\sum_{n=1}^{\infty}\frac{c(n)\log n}{n^s}=-\frac{\zeta'(s)\zeta(2s-1)}{\zeta^2(s)}+\frac{2\zeta'(2s-1)}{\zeta(s)}.
		$$
		Again, from \cite[p.~140, Corollary 5.3]{montgomery-vaughan},
		\begin{equation*}
			-\sum_{n\leq x}\frac{c(n)\log n}{n}=\frac{1}{2\pi \di}\int_{\sigma_0-\di T}^{\sigma_0+\di T}\bigg(-\frac{\zeta'(1+w)\zeta(1+2w)}{\zeta^2(1+w)}+\frac{2\zeta'(1+2w)}{\zeta(1+w)}\bigg)\frac{x^w}{w}\dd w+R   
		\end{equation*}
		where $$R\ll\sum_{x/2<n<2x\atop n\neq x}\frac{|c(n)|\log n}{n}\min\bigg( 1,\frac{x}{T|x-n|}\bigg)+\frac{4^{\sigma_0}+x^{\sigma_0}}{T}\sum_{n=1}^{\infty}\frac{|c(n)|\log n}{n^{1+\sigma_0}}.$$
		From \eqref{abs c(n) ds}, we have $-\sum\limits_{n=1}^{\infty}\displaystyle\frac{|c(n)|\log n}{n^{s}}=\frac{d}{ds}\left(\frac{\zeta(s)\zeta(2s-1)}{\zeta(2s)}\right)$.
		Hence
		$-\sum\limits_{n=1}^{\infty}\displaystyle\frac{|c(n)|\log n}{n^{1+\sigma_0}}\ll\frac{1}{\sigma^3_0}$.
		
		As $c(n)\ll \sqrt{n}$,
		$$ 
		\sum_{x/2<n<2x\atop n\neq x}\frac{|c(n)|\log n}{n}\min\bigg( 1,\frac{x}{T|x-n|}\bigg)\ll \frac{\sqrt{x}\log^2 x}{T}.
		$$
		So with $\sigma_0=\frac{1}{\log x}$, we obtain
		
		$$
		R\ll \frac{x^{\sigma_0}}{T \sigma^3_0}+\frac{\sqrt{x}\log^2 x}{T}\ll\frac{\sqrt{x}\log^2 x}{T}.
		$$
		In the Vinogradov-Korobov zero free region, using \eqref{vk}, \eqref{richert result}, and a bound on $\zeta'(1+w)/\zeta(1+w)$ \cite[p.~136]{titch}, we have
		$$
		-\frac{\zeta'(1+w)\zeta(1+2w)}{\zeta^2(1+w)}+\frac{2\zeta'(1+2w)}{\zeta(1+w)}\ll \log^\kappa t.
		$$
		for some $\kappa>0$.
		Since the rest of the proof follows the same technique as that of Theorem \ref{c(n) over n thm}, we arrive at
		$$
		-\sum_{n=1}^{\infty}\frac{c(n)\log n}{n}=\lim_{w\to 0}w\bigg(-\frac{\zeta'(1+w)\zeta(1+2w)}{\zeta^2(1+w)}+\frac{2\zeta'(1+2w)}{\zeta(1+w)}\bigg)\frac{x^w}{w}=\frac{\gamma}{2}.
		$$
	\end{proof}
	
		\subsection{The Voronoi summation formula for $\lambda(n)$}\label{4.2}
Theorem \ref{vsf lambda(n)} is proved here. Let $\tau$ be a number satisfying $0<\tau<\min{(\delta, 1/8)}$, where $\delta>0$ is the number occurring in \eqref{delta defn}. 
	
	Define the integral $I$ by
	\begin{align}\label{I defn}
		I:=\frac{1}{2\pi i}\int_{(1+\tau)}\frac{\zeta(2s)}{\zeta(s)}\Phi(s)\dd s.
	\end{align}
	The integral is well-defined since for large enough $|t|$, $\Phi(\sigma+it)\ll t^{-1-\delta}$ and for a fixed $\sigma>1$, we have $\zeta(2\sigma+2it)/\zeta(\sigma+it)=O_{\sigma}(1)$.
	Therefore,
	\begin{equation}\label{first expression for I}
		\begin{split}
			I=\frac{1}{2\pi i}\int_{(1+\tau)}\frac{\zeta(2s)}{\zeta(s)}\Phi(s)\dd s=\sum_{n=1}^{\infty}\frac{\lambda(n)}{2\pi i}\int_{(1+\tau)}\Phi(s)n^{-s}\dd s
			=\sum_{n=1}^{\infty}\lambda(n)\phi(n),
		\end{split}
	\end{equation}
	where the interchange in the order of summation and integration can be justified using the dominated convergence theorem.
	
	We would like to shift the line of integration to $\Re s=-\tau$, use the functional equation of $\zeta(s)$, and then employ the change of variable $s\to1-s$ so as to be able to write the resulting shifted integral as an infinite sum. 
	
	Let $T$ be a sufficiently large number. Form the sequence $\{\tilde{T}_n\}_{n=0}^{\infty}$ defined by $\tilde{T}_0=T$ and $\tilde{T}_n=\tilde{T}_{n-1}+\tilde{T}_{n-1}^{1/3}$ for $n\geq1$. It is not difficult to see that $\{\tilde{T}_n\}_{n=1}^{\infty}\to\infty$. From \cite[Lemma 1]{inoue} (see also \cite[Theorem 2]{ramachandra-sankaranarayanan jims} or \cite[Lemma 2.8]{basak-robles-zaharescu}), there exists a sequence $ \{T_n\}_{n=1}^{\infty}$ with $T_n\in[\tilde{T}_{n-1}, \tilde{T}_n]$ such that 
	\begin{equation}\label{bound for reciprocal of zeta}
	\frac{1}{\zeta(\sigma+i T_n)}\ll T_n^\epsilon
	\end{equation}
	 for any $\epsilon>0$ and $1/2\leq \sigma \leq 2$. From \eqref{zetafe sym}, 
	\begin{align}\label{fe zeta}
		\zeta(s)=\chi(s)\zeta(1-s),
	\end{align}
	where $\chi(s):=\pi^{s-\frac{1}{2}}\Gamma\left(\frac{1-s}{2}\right)/\Gamma\left(\frac{s}{2}\right)$. Now \eqref{strivert} implies
	\begin{equation}\label{chi}
		\chi(\sigma+it)\sim t^{1/2-\sigma},
	\end{equation}
	This, along with \eqref{fe zeta}, implies that for $-1\leq\sigma\leq1/2$,
	$1/\zeta(\sigma+i T_n)\ll T_n^{\sigma-\frac{1}{2}+\epsilon}\ll T_n^{\epsilon}$. In conclusion, 
	\begin{align}\label{one over zeta}
	\frac{1}{\zeta(\sigma+i T_n)}\ll T_{n}^\epsilon,
	\end{align}
	 for any $\epsilon>0$ and $-1\leq \sigma \leq 2$.
	
	Construct the rectangular contour $[1+\tau-iT_n, 1+\tau+iT_n, -\tau+iT_n, -\tau-iT_n]$, where $T_n$ is a number belonging to the sequence $\{T_m\}_{m=1}^{\infty}$ constructed above. Inside the contour lie the pole of the integrand at $0$ (due to $\Phi(s)$), the simple pole at $1/2$ (due to $\zeta(2s)$) as well as the simple poles due to the non-trivial zeros of $\zeta(s)$ (assuming the simplicity of the zeros). Let $\rho_m$ denote the $m^{\textup{th}}$ non-trivial zero of $\zeta(s)$. Then, by the residue theorem,
	\begin{align}\label{contour 1}
		\left[	  \int_{1+\tau-i T_n}^{1+\tau+i T_n}+ \int_{1+\tau+i T_n}^{-\tau+i T_n}+ \int_{-\tau+i T_n}^{-\tau-i T_n}+\int_{-\tau-i T_n}^{1+\tau-i T_n}\right]\frac{\zeta(2s)}{\zeta(s)}\Phi(s)\dd s = 2\pi i\left(R_{0}+R_{1/2}+\sum_{|\gamma_m|<T_n}R_{\rho_m}\right),
	\end{align}
	where $R_a$ denotes the residue of the integrand $\frac{\zeta(2s)}{\zeta(s)}\Phi(s)$ at the pole $a$. 
	
	Let us analyze the integrals along the horizontal segments as $T_n\to\infty$. For $1/2\leq\sigma<3$, we have
	\begin{equation} \label{zeta-bound}
		\zeta(\sigma+\di T)\ll T^{1/4},
	\end{equation}
		 as $T\to \infty$, which follows from the bound  \cite[p.~96, Equation (5.1.8)]{titch} $\zeta\left(\frac{1}{2}+iT\right)=O\left(T^{1/4}\right)$ and the Phragm\'{e}n-Lindel\"{o}f principle.  Moreover, for $0\leq\sigma<1/2$, we use \eqref{fe zeta}, \eqref{chi}, and \eqref{zeta-bound} to get $\zeta(\sigma+\di T)\ll T^{3/4-\sigma}\ll T^{3/4}$. Now for $-1/4<\sigma<0$, we have \cite[p.~95, Equation (5.1.3)]{titch}, 
	$\zeta(\sigma+\di T)\ll T^{1/2-\sigma}\ll T^{3/4}$. Combining the last two bounds, for $-1/4<\sigma<1/2$, we have
\begin{align}\label{zeta-bound1}
\zeta(\sigma+\di T)\ll T^{3/4},
\end{align}
as $T\to\infty$. Thus, from \eqref{zeta-bound} and \eqref{zeta-bound1},	                                                              
%
	\begin{equation*}
		\begin{split}
			\left|\int_{1+\tau+\di T_n}^{-\tau+\di T_n}\frac{\zeta(2s)}{\zeta(s)}\Phi(s)\dd s\right|
			&=\left|\int_{-\tau}^{1+\tau}\frac{\zeta(2u+2\di T_n)}{\zeta(u+\di T_n)}\Phi(u+\di T_n)\dd u\right|\\
			&\leq\left|\int_{-\tau}^{1/4}\frac{\zeta(2u+2\di T_n)}{\zeta(u+\di T_n)}\Phi(u+\di T_n)\dd u\right|+\left|\int_{1/4}^{1+\tau}\frac{\zeta(2u+2\di T_n)}{\zeta(u+\di T_n)}\Phi(u+\di T_n)\dd u\right|\\
			&\ll\int_{-\tau}^{1/4}T_n^{3/4}T_n^{\epsilon}T_n^{-1-\delta}\dd u +\int_{1/4}^{1+\tau}T_n^{1/4}T_n^{\epsilon}T_{n}^{-1-\delta}\dd u.
		\end{split}
	\end{equation*}
	Now choose $\epsilon=\delta$ so that
	\begin{align*}
		\lim_{T_n\to\infty}\int_{1/4}^{1+\tau}T_n^{1/4}T_n^{\epsilon}T_{n}^{-1-\delta}\dd u= \lim_{T_n\to\infty}\left(\frac{3}{4}+\tau\right)T_n^{-3/4}=0,
	\end{align*}
	and $\int_{-\tau}^{1/4}T_n^{3/4}T_n^{\epsilon}T_n^{-1-\delta}\dd u\ll T_n^{-1/4}$
	so that
	\begin{align*}
		\lim_{T_n\to\infty}	\int_{-\tau}^{1/4}T_n^{1/2-2u}T_n^{\epsilon}T_n^{-1-\delta}\dd u=0.
	\end{align*}
	Therefore, 
	\begin{align*}
		\lim_{T_n\to\infty}\int_{1+\tau+\di T_n}^{-\tau+\di T_n}\frac{\zeta(2s)}{\zeta(s)}\Phi(s)\dd s=0.
	\end{align*}
	Similarly, the integral along $[-\tau-iT_n,  1+\tau-\di T_n]$ can be shown to approach zero as $T_n\to\infty$. Moreover, observe that using the functional equation \eqref{fe zeta}, we have, on the line Re$(s)=-\tau$,
	$\zeta(2s)/\zeta(s)\ll_{\tau}t^{\tau}$, which, along with the hypothesis $\Phi(s)\ll t^{-1-\delta}$ and the fact that $\tau<\delta$, implies
	\begin{align}\label{left line 1}
		\int_{(-\tau)}\frac{\zeta(2s)}{\zeta(s)}\Phi(s)\dd s\ll 1.
	\end{align}
	Thus, from \eqref{I defn}, \eqref{contour 1}, \eqref{left line 1} and the functional equation \eqref{fe zeta}, we see that
	\begin{align}
			I&=\frac{1}{2\pi i}\int_{(-\tau)}\frac{\zeta(2s)}{\zeta(s)}\Phi(s)\dd s+R_{0}+R_{1/2}+\lim_{T_n\to\infty}\sum_{|\gamma_m|<T_n}R_{\rho_m}\label{added}\\
		&=\frac{1}{2\pi i}\int_{(-\tau)}\frac{\zeta(1-2s)\Gamma(\frac{1}{2}-s)\Gamma(\frac{s}{2})}{\zeta(1-s)\Gamma(s)\Gamma(\frac{1-s}{2})}\Phi(s)\pi^{s}\dd s+R_{0}+R_{1/2}+\lim_{T_n\to\infty}\sum_{|\gamma_m|<T_n}R_{\rho_m}\nonumber\\
		&=\frac{1}{2\pi i}\int_{(1+\tau)}\frac{\zeta(2s-1)\Gamma(s-\frac{1}{2})\Gamma(\frac{1-s}{2})}{\zeta(s)\Gamma(1-s)\Gamma(\frac{s}{2})}\Phi(1-s)\pi^{1-s}\dd s+R_{0}+R_{1/2}+\lim_{T_n\to\infty}\sum_{|\gamma_m|<T_n}R_{\rho_m},\label{after residue theorem}
	\end{align}
	where we employed the change of variable $s\to 1-s$. 
	It is easy to see that
	\begin{align}
		R_{\frac{1}{2}}&=\lim_{s\to \frac{1}{2}}\bigg(s-\frac{1}{2}\bigg)\frac{\zeta(2s)}{\zeta(s)}\Phi(s)=\lim_{s\to \frac{1}{2}}\bigg(s-\frac{1}{2}\bigg)\frac{\zeta(2s)}{\zeta(s)}\int_{0}^{\infty}\phi(x)x^{s-1}\dd x =\frac{1}{2\zeta(\frac{1}{2})}\int_{0}^{\infty}\frac{\phi(x)}{\sqrt{x}}\dd x,\nonumber\\
		R_{\rho_m}&= \lim_{s\to \rho_m}(s-\rho_m)\frac{\zeta(2s)}{\zeta(s)}\Phi(s)=\lim_{s\to \rho_m}(s-\rho_m)\frac{\zeta(2s)}{\zeta(s)}\int_{0}^{\infty}\phi(x)x^{s-1}\dd x=\frac{\zeta(2\rho_m)}{\zeta'(\rho_m)}\int_{0}^{\infty}\phi(x)x^{\rho_m-1}\dd x,\nonumber\\
		R_0&=\textup{Res}_{s=0}\frac{\zeta(2s)}{\zeta(s)}\Phi(s),\label{r0}
	\end{align}
	where, in the first two residue calculations, we invoked the Mellin inversion theorem \cite[p.~341, Theorem 1]{Mcla}. We choose to simplify the expression for $R_0$ at the end.
	
	It remains to analyze the integral 	on the right-hand side of \eqref{after residue theorem}, which we denote by $I_1$. Observe that using \eqref{c(n) ds}, we have
	\begin{align}\label{I_1}
		I_1
		&=\frac{\pi}{2\pi i}\sum_{n=1}^{\infty}c(n)\int_{(1+\tau)}\frac{\Gamma(s-\frac{1}{2})\Gamma(\frac{1-s}{2})}{\Gamma(1-s)\Gamma(\frac{s}{2})}\Phi(1-s)(\pi n)^{-s}\dd s\nonumber\\
		&=\frac{\sqrt{\pi}}{2\pi i}\sum_{n=1}^{\infty}c(n)\int_{(1+\tau)}\Gamma\bigg(s-\frac{1}{2}\bigg)\sin\bigg(\frac{\pi s}{2}\bigg)\Phi(1-s)2^s(\pi n)^{-s}\dd s\nonumber\\		
		&=\sqrt{\pi} \sum_{n=1}^{\infty}c(n)J\left(\frac{\pi n}{2}\right),
	\end{align}
	where
	\begin{equation*}
		J(x):=\frac{1}{2\pi i}\int_{(1+\tau)}\Gamma\bigg(s-\frac{1}{2}\bigg)\sin\bigg(\frac{\pi s}{2}\bigg)\Phi(1-s)x^{-s}\dd s.
	\end{equation*}
The interchange of the order of summation and integration can be justified using the dominated convergence theorem. We would like to express $J(x)$ as a real integral using Parseval's formula \eqref{Parseval1}. However, from the conditions on $\Phi$, it can be seen that $\Phi(1-s)$ is analytic in $-1<\textup{Re}(s)<1$ whereas $\Gamma\left(s-\frac{1}{2}\right)\sin\left(\frac{\pi s}{2}\right)x^{-s}$ is analytic in $-1/2<\textup{Re}(s)<1/2$. Hence the Parseval formula is inapplicable as Re$(s)>1$ in the definition of $J(x)$. 
To address this issue, we shift the line of integration to from  Re$(s)=1+\tau$ to Re$(s)=-\tau$, and then employ change of variable $s$ to $1-s$ as will be seen below.
	
	Using the reflection formula for Gamma function and \eqref{strivert}, it is easy to see that as $t\to\pm\infty$,
	\begin{equation}\label{gamma sine bound}
		\Gamma\bigg(s-\frac{1}{2}\bigg)\sin\bigg(\frac{\pi s}{2}\bigg)\ll t^{\sigma-1}.
	\end{equation}
	Also, $\Phi(1-s)\ll t^{-1-\delta}$. This implies that the integrals along the horizontal segments tend to zero as $t\to\pm\infty$. Therefore, by Cauchy's residue theorem,
	\begin{align}\label{J(x) after residue theorem}
		J(x)&=\frac{1}{2\pi i}\int_{(-\tau)}\Gamma\bigg(s-\frac{1}{2}\bigg)\sin\bigg(\frac{\pi s}{2}\bigg)\Phi(1-s)x^{-s}\dd s+\tilde{R}_{\frac{1}{2}}(x)+\tilde{R}_1(x)\nonumber\\
		&=\frac{1}{2\pi i}\int_{(1+\tau)}\Gamma\bigg(\frac{1}{2}-s\bigg)\cos\bigg(\frac{\pi s}{2}\bigg)\Phi(s)x^{s-1}\dd s+\tilde{R}_{\frac{1}{2}}(x)+\tilde{R}_1(x),
	\end{align}
	where $\tilde{R}_a(x)$ denotes the residue of the integrand $\Gamma\left(s-\frac{1}{2}\right)\sin\left(\frac{\pi s}{2}\right)\Phi(1-s)x^{-s}$ at the pole $a$, and where, in the last step, we made the change of variable $s\to1-s$.
	
	We now justify the applicability of Parseval's formula to the integral on the last line of \eqref{J(x) after residue theorem}. Let $F(s)=\Gamma\left(s-\frac{1}{2}\right)\sin\left(\frac{\pi s}{2}\right)x^{-s}$ and $G(s)=\Phi(s)$. Firstly, the line of integration of the integral on the right-hand side of \eqref{J(x) after residue theorem} lies in the common strip of analyticity of $F(1-s)$ and $G(s)$, that is,  $1/2<\textup{Re}(s)<3/2$. Moreover, it is clear from \eqref{gamma sine bound} that $F(-\tau-it)\in L(-\infty,\infty)$, and also, the fact that $\Phi(s)$ is holomorphic in $0<\textup{Re}(s)<2$ implies $t^{\tau}\phi(t)\in L(0, \infty)$. Hence by Parseval's formula \ref{Parseval1},
	\begin{align}\label{bef bef sincos}
		\frac{1}{2\pi i}\int_{(1+\tau)}\Gamma\bigg(\frac{1}{2}-s\bigg)\cos\bigg(\frac{\pi s}{2}\bigg)\Phi(s)x^{s-1}\dd s&=\int_{0}^{\infty}\phi(t)\Bigg(\sin\bigg(t x+\frac{\pi}{4}\bigg){(tx)}^{-1/2}-\frac{(tx)^{-1/2}}{\sqrt{2}}\Bigg)\dd t,
	\end{align}
	since
	\begin{align}\label{bef sincos}
		\frac{1}{2\pi i}\int_{(-\tau)}F(s)t^{-s}\dd s=\sin\bigg(tx +\frac{\pi}{4}\bigg)(xt)^{-1/2}-\frac{(xt)^{-1/2}}{\sqrt{2}},
	\end{align}
	which is what we set to prove next. Employing the change of variable $s=w+1/2$, we have
	\begin{align}\label{sincos}
		&\frac{1}{2\pi i}\int_{(-\tau)}\Gamma\bigg(s-\frac{1}{2}\bigg)\sin\bigg(\frac{\pi s}{2}\bigg)t^{-s}\dd s\nonumber\\
		&=\frac{t^{-1/2}}{\sqrt{2}}\left\{\frac{1}{2\pi i}\int_{(-\tau-\frac{1}{2})}\Gamma(w)\sin\bigg(\frac{\pi w}{2}\bigg)t^{-w}\dd w+\frac{1}{2\pi i}\int_{(-\tau-\frac{1}{2})}\Gamma(w)\cos\bigg(\frac{\pi w}{2}\bigg)t^{-w}\dd w\right\}.
	\end{align}
	Now we know that for $-1/2<\textup{Re}(\xi)<1/2$,
	\begin{align}\label{sin}
		\frac{1}{2\pi i}\int_{(\textup{Re}(\xi))}\Gamma(w)\sin\bigg(\frac{\pi w}{2}\bigg)t^{-w}\dd w=\sin(t),
	\end{align}
	and  for $0<\textup{Re}(\xi)<1/2$,
	\begin{align}\label{cos}
		\frac{1}{2\pi i}\int_{(\textup{Re}(\xi))}\Gamma(w)\cos\bigg(\frac{\pi w}{2}\bigg)t^{-w}\dd w=\cos(t).
	\end{align}
	Thus, in order to use both \eqref{sin} and \eqref{cos} in \eqref{sincos}, we need to shift the line of integration from Re$(w)=-\tau-1/2$ to $0<\textup{Re}(w)=\lambda<1/2$ in each of the integrals on the right-hand side of \eqref{sincos}. While this produces no poles for the first integral, we do have to consider the contribution of pole at $w=0$ for the second integral at which the residue is $1$. Since the integrals along the horizontal segments in each of the integrals go to zero as $T\to\infty$, combining all of this and finally replacing $t$ by $xt$, we obtain \eqref{bef sincos}. Further, 
	\begin{align}\label{residues tilde half}
		\tilde{R}_{\frac{1}{2}}(x)&=\frac{x^{-1/2}}{\sqrt{2}}\Phi\bigg(\frac{1}{2}\bigg).
	\end{align}
	Hence from \eqref{J(x) after residue theorem}, \eqref{bef bef sincos} and \eqref{residues tilde half}, we have
	\begin{align*}
		J(x)&=\int_{0}^{\infty}\phi(t)\Bigg(\sin\bigg(t x+\frac{\pi}{4}\bigg){(tx)}^{-1/2}-\frac{(tx)^{-1/2}}{\sqrt{2}}\Bigg)\dd t+\frac{x^{-1/2}}{\sqrt{2}}\Phi\bigg(\frac{1}{2}\bigg)+\tilde{R}_1(x)\nonumber\\
		&=x^{-1/2}\int_{0}^{\infty}\frac{\phi(t)}{\sqrt{t}}\sin\bigg(t x+\frac{\pi}{4}\bigg)\dd t+\tilde{R}_1(x),
	\end{align*}
	since $\Phi(1/2)=\int_{0}^{\infty}\phi(t)/\sqrt{t}\hspace{0.5mm}\dd t$. Along with the definition of $\tilde{R}_1(x)$, this gives
	\begin{equation}\label{J pi n over 2}
		\begin{split}
			J\bigg(\frac{\pi n}{2}\bigg)&=\sqrt{\frac{2}{\pi n}}\int_{0}^{\infty}\frac{\phi(t)}{\sqrt{t}}\sin\bigg(\frac{\pi  nt}{2} +\frac{\pi}{4}\bigg)\dd t+\Res_{s=1}\Gamma\bigg(s-\frac{1}{2}\bigg)\sin\bigg(\frac{\pi s}{2}\bigg)\Phi(1-s)\bigg(\frac{\pi n}{2}\bigg)^{-s} \\
			&=\sqrt{\frac{2}{\pi n}}\int_{0}^{\infty}\frac{\phi(t)}{\sqrt{t}}\sin\bigg(\frac{\pi  nt}{2} +\frac{\pi}{4}\bigg)\dd t-\frac{2}{\pi n}\Res_{s=0}\Gamma\bigg(\frac{1}{2}-s\bigg)\cos\bigg(\frac{\pi s}{2}\bigg)\Phi(s)\bigg(\frac{\pi n}{2}\bigg)^{s}.
		\end{split}
	\end{equation}
	We now evaluate the residue on the extreme right-hand side of the above equation. As $s\to0$, the standard power series expansions of the Gamma function, cosine and the power function give
	\begin{equation*}
		\Gamma\bigg(\frac{1}{2}-s\bigg)\cos\bigg(\frac{\pi s}{2}\bigg)\bigg(\frac{\pi n}{2}\bigg)^{s}=\sqrt{\pi}\bigg\{1+\bigg(\log n+\log\left(\frac{\pi}{2}\right)-\psi\bigg(\frac{1}  {2}\bigg)\bigg)s+O(|s|^2)\bigg\},  
	\end{equation*}
	where $\psi(s)$ is the digamma function. 
	Since we have assumed $\Phi(s)$ has at most a double pole at $s=0$, let
	\begin{equation*}
		\Phi(s)=\sum_{m=-2}^{\infty}a_m s^m.
	\end{equation*}
	Then it can be seen that
	\begin{equation}\label{resdone}
		\textup{Res}_{s=0}\Gamma\bigg(\frac{1}{2}-s\bigg)\cos\bigg(\frac{\pi s}{2}\bigg)\Phi(s)\bigg(\frac{\pi n}{2}\bigg)^{s}=\sqrt{\pi}\bigg(a_{-2}\bigg(\log n+\log\left(\frac{\pi}{2}\right)-\psi\bigg(\frac{1}  {2}\bigg)\bigg)+a_{-1}\bigg).
	\end{equation}
	Therefore, from \eqref{I_1}, \eqref{J pi n over 2} and \eqref{resdone},
	\begin{align}\label{I_1 simplification}
		I_1&=\sqrt{2}\sum_{n=1}^{\infty}\frac{c(n)}{\sqrt{n}}\int_{0}^{\infty}\frac{\phi(t)}{\sqrt{t}}\sin\bigg(\frac{\pi  nt}{2} +\frac{\pi}{4}\bigg)\dd t-2\sum_{n=1}^{\infty}\frac{c(n)}{n}\bigg(a_{-2}\bigg(\log n+\log\left(\frac{\pi}{2}\right)-\psi\bigg(\frac{1}  {2}\bigg)\bigg)+a_{-1}\bigg)\nonumber\\
		&=\sqrt{2}\sum_{n=1}^{\infty}\frac{c(n)}{\sqrt{n}}\int_{0}^{\infty}\frac{\phi(t)}{\sqrt{t}}\sin\bigg(\frac{\pi  nt}{2} +\frac{\pi}{4}\bigg)\dd t-\log(2\pi)\cdot a_{-2}-a_{-1},
	\end{align}
	where, in the last step, we used Theorems \ref{c(n) over n thm}, \ref{c(n) log(n) over n thm} and the fact \cite[p.~905, formula \textbf{8.366.2}]{gr} that
	$\psi(1/2)=-\gamma-2\log(2)$.
	Thus from \eqref{after residue theorem}-\eqref{I_1} and \eqref{I_1 simplification},
	\begin{align}\label{lambda vsf almost there}
		I&=\sqrt{2}\sum_{n=1}^{\infty}\frac{c(n)}{\sqrt{n}}\int_{0}^{\infty}\frac{\phi(t)}{\sqrt{t}}\sin\bigg(\frac{\pi  nt}{2} +\frac{\pi}{4}\bigg)\dd t-\log(2\pi)\cdot a_{-2}-a_{-1}+\frac{1}{2\zeta(\frac{1}{2})}\int_{0}^{\infty}\frac{\phi(x)}{\sqrt{x}}\dd x\nonumber\\
		&\quad+\lim_{T_n\to\infty}\sum_{|\gamma_m|<T_n}\frac{\zeta(2\rho_m)}{\zeta'(\rho_m)}\int_{0}^{\infty}\phi(x)x^{\rho_m-1}\dd x+R_0.
	\end{align}
	The only thing remaining now is to evaluate the expression for $R_0$ given in \eqref{r0}. To that end, as $s\to0$, we have from \cite[pp. 19-20, Equations (2.4.3), (2.4.5)]{titch}, $\zeta(s)=-\frac{1}{2}-\frac{1}{2}\log(2\pi)s+O(|s|^2)$, 
	which implies
	$\frac{\zeta(2s)}{\zeta(s)}=1+\log(2\pi)s+O(|s|^{2})$
	so that $R_0$ is the coefficient of $s^{-1}$ in 
	\begin{align*}
		\left(1+\log(2\pi)s+O(|s|^{2})\right)\left(\frac{a_{-2}}{s^2}+\frac{a_{-1}}{s}+O(1)\right).
	\end{align*}
	This implies $R_0=\log(2\pi)\cdot a_{-2}+a_{-1}$. Substituting this expression in \eqref{lambda vsf almost there} , simplifying, and equating the resulting expression for $I$ with the right-hand side of \eqref{first expression for I}, we arrive at \eqref{vsf lambda eqn}. This completes the proof.
	
\subsection{Corollaries of the Voronoi summation formula}	
	
\begin{proof}[Corollary \textup{\ref{lambda simple pole}}][]
	
Let Re$(s)>0$ and $y>0$.  Let $\Phi(s)=\Gamma(s)y^{-s}$ in Theorem \ref{vsf lambda(n)}.	Firstly, $\Phi(s)$ satisfies the conditions given before the statement of Theorem \ref{vsf lambda(n)} and has a simple pole at $s=0$. 
Secondly, $\phi(x)=e^{-xy}$. Also, the following is easily verified:
\begin{align*}
\int_{0}^{\infty}\frac{\phi(x)}{\sqrt{x}}\dd x=\frac{\sqrt{\pi}}{\sqrt{y}},\hspace{8mm}\int_{0}^{\infty}\phi(x)x^{\rho_m-1}\dd x=\frac{\Gamma(\rho_m)}{y^{\rho_m}}.
\end{align*}
Next,
\begin{align*}
\int_{0}^{\infty}\frac{\phi(x)}{\sqrt{x}}\sin\bigg(\frac{\pi nx}{2}+\frac{\pi}{4}\bigg)\dd x&=\frac{1}{\sqrt{2}}\left(\int_{0}^{\infty}\frac{e^{-xy}}{\sqrt{x}}\sin\bigg(\frac{\pi nx}{2}\bigg)\dd x+\int_{0}^{\infty}\frac{e^{-xy}}{\sqrt{x}}\cos\bigg(\frac{\pi nx}{2}\bigg)\dd x\right)\nonumber\\
&=\frac{1}{\sqrt{2}}\left\{\sqrt{\frac{\pi}{2}}\frac{\sqrt{\sqrt{y^2+\frac{\pi^2n^2}{4}}-y}}{\sqrt{y^2+\frac{\pi^2n^2}{4}}}+\sqrt{\frac{\pi}{2}}\frac{\sqrt{\sqrt{y^2+\frac{\pi^2n^2}{4}}+y}}{\sqrt{y^2+\frac{\pi^2n^2}{4}}}\right\}\nonumber\\
&=\frac{\sqrt{\pi}}{2}\frac{\left(\sqrt{\sqrt{y^2+\frac{\pi^2n^2}{4}}-y}+\sqrt{\sqrt{y^2+\frac{\pi^2n^2}{4}}+y}\right)}{\sqrt{y^2+\frac{\pi^2n^2}{4}}},
\end{align*}
where, in the second step, we used the integral evaluations \cite[p.~499, Equations \textbf{3.944.13, 3.944.14}]{gr} with $\mu=1/2, n=0, \beta=y$ and $b=\pi n/2$. Using all of these evaluations, we arrive at \eqref{lambda simple pole eqn}.

To prove \eqref{lambda simple pole bound}, we first observe that as $n\to\infty$,
\begin{align*}
\frac{\sqrt{\sqrt{4y^2+\pi^2n^2}-2y}+\sqrt{\sqrt{4y^2+\pi^2n^2}+2y}}{\sqrt{4y^2+\pi^2n^2}}=\frac{2}{\sqrt{\pi n}}+O_y\left(\frac{1}{n^{5/2}}\right),
\end{align*}
where the big-O involves terms with positive powers of $y$. Therefore, using \eqref{c(n) over n}, we see that as $y\to0^{+}$,
\begin{align*}
\sqrt{\pi}\sum_{n=1}^{\infty}\frac{c(n)}{\sqrt{n}}\frac{\sqrt{\sqrt{4y^2+\pi^2n^2}-2y}+\sqrt{\sqrt{4y^2+\pi^2n^2}+2y}}{\sqrt{4y^2+\pi^2n^2}}=O(1).
\end{align*}
Since we assume the RH and the absolute convergence of $\lim_{T_n\to\infty}\sum_{|\gamma_m|<T_n}\frac{\zeta(2\rho_m)\Gamma(\rho_m)}{\zeta'(\rho_m)}y^{-\rho_m}$, we see that
\begin{align*}
\lim_{T_n\to\infty}\sum_{|\gamma_m|<T_n}\frac{\zeta(2\rho_m)\Gamma(\rho_m)}{\zeta'(\rho_m)}y^{-\rho_m}\ll y^{-1/2}\sum_{\rho_m}\left|\frac{\zeta(2\rho_m)\Gamma(\rho_m)}{\zeta'(\rho_m)}\right|=O(y^{-1/2}).
\end{align*}
Along with \eqref{lambda simple pole eqn}, the above two equations imply \eqref{lambda simple pole bound}. 
\end{proof}
	
\begin{proof}[Corollary \textup{\ref{lambda gaussian}}][]
Let Re$(s)>0$ and $y>0$. Let $\Phi(s)=\frac{1}{2}\Gamma\left(\frac{s}{2}\right)y^{-\frac{s}{2}}$ in Theorem \ref{vsf lambda(n)} so that $\phi(x)=e^{-x^2y}$. We refrain from giving the details since they  are similar to that of Corollary \ref{lambda simple pole}; instead, we only show below how one obtains the evaluation
\begin{align}\label{sin cos both}
\int_{0}^{\infty}\frac{e^{-x^2y}}{\sqrt{x}}	\sin\bigg(\frac{\pi nx}{2}+\frac{\pi}{4}\bigg)\dd x=\frac{\pi^{3/2}\sqrt{n}}{4\sqrt{2y}}e^{-\frac{\pi^2n^2}{32y}}\left(I_{-\frac{1}{4}}\left(\frac{\pi^2n^2}{32y}\right)+I_{\frac{1}{4}}\left(\frac{\pi^2n^2}{32y}\right)\right).
\end{align}
From \cite[p.~503, formulas \textbf{3.952.7, 3.952.8}]{gr},
\begin{align}
\int_{0}^{\infty}x^{\mu-1}e^{-\beta x^2}\sin(cx)\dd x&=\frac{ce^{-\frac{c^2}{4\beta}}}{2\beta^{\frac{\mu+1}{2}}}\Gamma\left(\frac{1+\mu}{2}\right){}_1F_{1}\left(1-\frac{\mu}{2};\frac{3}{2};\frac{c^2}{4\beta}\right)\hspace{6mm}(\textup{Re}(\beta)>0, \textup{Re}(\mu)>-1),\label{sin wala}\\
\int_{0}^{\infty}x^{\mu-1}e^{-\beta x^2}\cos(cx)\dd  x&=\frac{e^{-\frac{c^2}{4\beta}}}{2\beta^{\frac{\mu}{2}}}\Gamma\left(\frac{\mu}{2}\right){}_1F_{1}\left(\frac{1-\mu}{2};\frac{1}{2};\frac{c^2}{4\beta}\right)\hspace{6mm}(\textup{Re}(\beta)>0, \textup{Re}(\mu)>0).\label{cos wala}
\end{align}
Let $\mu=1/2, \beta=y$ and $c=\pi n/2$ in \eqref{sin wala} to get
\begin{align}\label{sin wale se}
\int_{0}^{\infty}\frac{e^{-x^2y}}{\sqrt{x}}\sin\left(\frac{\pi n x}{2}\right)\, dx=\frac{\pi n}{4y^{3/4}}\Gamma\left(\frac{3}{4}\right)e^{-\frac{\pi^2n^2}{16y}}{}_1F_{1}\left(\frac{3}{4};\frac{3}{2};\frac{\pi^2n^2}{16y}\right).
\end{align}
Now from \cite[p.~126, Theorem 43]{rainville}, for $2a$ not equal to a negative odd integer, 
\begin{align*}
e^{-z}{}_1F_{1}(a;2a;2z)&={}_0F_{1}\left(-;a+\tfrac{1}{2};\tfrac{z^2}{4}\right)
=\Gamma\left(a+\frac{1}{2}\right)\left(\frac{z}{2}\right)^{\frac{1}{2}-a}I_{a-\frac{1}{2}}(z),
\end{align*}
where the second equality follows from \eqref{sumbesselj} and \eqref{besseli}. Using the above equation with $a=3/4$ and $z=\pi^2n^2/(32y)$ and using the resultant in \eqref{sin wale se}, we arrive at
\begin{align}\label{sin wale se aa gaya}
\int_{0}^{\infty}\frac{e^{-x^2y}}{\sqrt{x}}\sin\left(\frac{\pi n x}{2}\right)\, dx
=\frac{\pi^{3/2}\sqrt{n}}{4\sqrt{y}}e^{-\frac{\pi^2n^2}{32y}}I_{\frac{1}{4}}\left(\frac{\pi^2n^2}{32y}\right).
\end{align}
Similarly letting $\mu=1/2, \beta=y$ and $c=\pi n/2$ in \eqref{cos wala}, we get
\begin{align}\label{cos wale se aa gaya}
	\int_{0}^{\infty}\frac{e^{-x^2y}}{\sqrt{x}}\cos\left(\frac{\pi n x}{2}\right)\, dx
	=\frac{\pi^{3/2}\sqrt{n}}{4\sqrt{y}}e^{-\frac{\pi^2n^2}{32y}}I_{-\frac{1}{4}}\left(\frac{\pi^2n^2}{32y}\right).
\end{align}
Therefore, \eqref{sin wale se aa gaya} and \eqref{cos wale se aa gaya} result in \eqref{sin cos both}, which, in turn, leads to \eqref{lambda gaussian eqn}.

Using the asymptotic expansion of the modified Bessel function of the second kind \cite[p.~203, Equation (7.23.3)]{watson-1966a}, it can be checked that as $n\to\infty$,
\begin{align*}
e^{-\frac{\pi^2n^2}{32y}}\left(I_{-\frac{1}{4}}\left(\frac{\pi^2n^2}{32y}\right)+I_{\frac{1}{4}}\left(\frac{\pi^2n^2}{32y}\right)\right)=\frac{8\sqrt{y}}{\pi^{3/2}n}\left(1+O_{y}\left(\frac{1}{n^2}\right)\right),
\end{align*}
where the big-O term involves positive powers of $y$. Along with \eqref{c(n) over n}, this implies that
\begin{align*}
\frac{\pi^{3/2}}{4\sqrt{y}}\sum_{n=1}^{\infty}c(n)e^{-\frac{\pi^2n^2}{32y}}\left(I_{-\frac{1}{4}}\left(\frac{\pi^2n^2}{32y}\right)+I_{\frac{1}{4}}\left(\frac{\pi^2n^2}{32y}\right)\right)=O(1),
\end{align*}
which leads to \eqref{gaussian bound} following the analysis similar to that done for obtaining \eqref{lambda simple pole bound}.
\end{proof}

\begin{proof}[Corollary \textup{\ref{lambda hypergeometric}}][]
Let Re$(s)>0$ and $y>0$, and let $\Phi(s):=2^{s-2}y^{-s}\Gamma^{2}\left(\frac{s}{2}\right)$ in Theorem \ref{vsf lambda(n)}. Then $\Phi(s)$ has a second order pole at $s=0$, and it satisfies the conditions given before the statement of Theorem \ref{vsf lambda(n)}. Also, from \cite[p.~196, formula 5.39]{ober}, we have $\phi(x)=K_{0}(xy)$.
Next,
\begin{align*}
\int_{0}^{\infty}\frac{\phi(x)}{\sqrt{x}}\dd x=\frac{4\sqrt{2}}{\sqrt{y}}\Gamma^{2}\left(\frac{5}{4}\right),\hspace{5mm}
\int_{0}^{\infty}\phi(x)x^{\rho_m-1}\dd x=2^{\rho_m-2}y^{-\rho_m}\Gamma^{2}\left(\frac{\rho_m}{2}\right).
\end{align*}
Thus we will be done with proving \eqref{lambda hypergeometric eqn} provided we show that
\begin{align}\label{hyp complete}
\int_{0}^{\infty}\frac{\phi(x)}{\sqrt{x}}\sin\bigg(\frac{\pi nx}{2}+\frac{\pi}{4}\bigg)\dd x
=\frac{\pi^{\frac{3}{2}}}{\sqrt{2}(\pi^2n^2+4y^2)^{\frac{1}{4}}}{}_2F_{1}\left(\frac{1}{2},\frac{1}{2};1;\frac{1}{2}+\frac{\pi n}{2\sqrt{\pi^2n^2+4y^2}}\right).
\end{align}
To that end, using \cite[p.~731, formulas \textbf{6.699.3, 6.699.4}]{gr}, we have
\begin{align}
\int_{0}^{\infty}\frac{\phi(x)}{\sqrt{x}}\sin\bigg(\frac{\pi nx}{2}\bigg)\dd x&=\frac{\pi n}{(2y)^{3/2}}\Gamma^{2}\left(\frac{3}{4}\right){}_2F_{1}\left(\frac{3}{4},\frac{3}{4};\frac{3}{2};-\frac{\pi^2n^2}{4y^2}\right),\label{1hyp}\\
\int_{0}^{\infty}\frac{\phi(x)}{\sqrt{x}}\cos\bigg(\frac{\pi nx}{2}\bigg)\dd x&=\frac{1 }{2^{3/2}\sqrt{y}}\Gamma^{2}\left(\frac{1}{4}\right){}_2F_{1}\left(\frac{1}{4},\frac{1}{4};\frac{1}{2};-\frac{\pi^2n^2}{4y^2}\right),\label{2hyp}
\end{align} 
Using \cite[p.~176, Exercise 1(e)]{aar}, with $a=b=1/4$, we obtain
\begin{align}\label{hyp1}
{}_2F_{1}\left(\frac{1}{4}, \frac{1}{4};\frac{1}{2};-x\right)=\frac{(1+x)^{-1/4}}{2\sqrt{\pi}}\Gamma^{2}\left(\frac{3}{4}\right)\bigg\{{}_2F_{1}\left(\frac{1}{2}, \frac{1}{2};1;\frac{1}{2}+\frac{1}{2}\sqrt{\frac{x}{1+x}}\right)+{}_2F_{1}\left(\frac{1}{2}, \frac{1}{2};1;\frac{1}{2}-\frac{1}{2}\sqrt{\frac{x}{1+x}}\right)\bigg\}.
\end{align}
Similarly, one can derive (see, for example, \cite{wolfram}\footnote{It is to be noted that the notation for the complete elliptic integral used here is such that  $K(k)=\frac{\pi}{2}{}_2F_{1}\left(\frac{1}{2},\frac{1}{2};1;k\right)$, which is different from the conventional notation whereby we have $K(k)=\frac{\pi}{2}{}_2F_{1}\left(\frac{1}{2},\frac{1}{2};1;k^2\right)$.}
\begin{align}\label{hyp2}
	{}_2F_{1}\left(\frac{3}{4}, \frac{3}{4};\frac{3}{2};-x\right)=\frac{(1+x)^{-1/4}}{4\sqrt{\pi x}}\Gamma^{2}\left(\frac{1}{4}\right)\bigg\{{}_2F_{1}\left(\frac{1}{2}, \frac{1}{2};1;\frac{1}{2}+\frac{1}{2}\sqrt{\frac{x}{1+x}}\right)-{}_2F_{1}\left(\frac{1}{2}, \frac{1}{2};1;\frac{1}{2}-\frac{1}{2}\sqrt{\frac{x}{1+x}}\right)\bigg\}.
\end{align}
From \eqref{hyp1} and \eqref{hyp2}, it is easy to derive
\begin{align}\label{add hyp}
2\sqrt{x}\Gamma^{2}\left(\frac{3}{4}\right){}_2F_{1}\left(\frac{3}{4}, \frac{3}{4};\frac{3}{2};-x\right)+\Gamma^{2}\left(\frac{1}{4}\right){}_2F_{1}\left(\frac{1}{4}, \frac{1}{4};\frac{1}{2};-x\right)=\frac{2\pi^{3/2}}{(1+x)^{1/4}}{}_2F_{1}\left(\frac{1}{2}, \frac{1}{2};1;\frac{1}{2}+\frac{1}{2}\sqrt{\frac{x}{1+x}}\right).
\end{align}
Now let $x=\pi^2n^2/(4y^2)$ in \eqref{add hyp}, divide both sides by $2^{3/2}\sqrt{y}$, and then compare the left-hand side of the resulting equation with the right-hand side of the equation obtained by adding the corresponding sides of \eqref{1hyp} and \eqref{2hyp}. This establishes \eqref{hyp complete} which completes the proof of \eqref{lambda hypergeometric eqn}.

To prove \eqref{lambda K big-O}, observe that as $n\to\infty$,
\begin{align*}
\frac{1}{(\pi^2n^2+4y^2)^{\frac{1}{4}}}{}_2F_{1}\left(\frac{1}{2},\frac{1}{2};1;\frac{1}{2}+\frac{\pi n}{2\sqrt{\pi^2n^2+4y^2}}\right)=\frac{2\log\left(4\pi n/y\right)}{\pi^{3/2}\sqrt{n}}+O_y\left(n^{-5/2}\right),
\end{align*}
where the big-O term involves positive powers of $y$. Thus, using Theorems \ref{c(n) over n thm} and \ref{c(n) log(n) over n thm}, we see that as $y\to0$,
 \begin{align*}
 \pi^{\frac{3}{2}}\sum_{n=1}^{\infty}\frac{c(n)}{\sqrt{n}(\pi^2n^2+4y^2)^{\frac{1}{4}}}{}_2F_{1}\left(\frac{1}{2},\frac{1}{2};1;\frac{1}{2}+\frac{\pi n}{2\sqrt{\pi^2n^2+4y^2}}\right)=O\left(\frac{1}{\sqrt{y}}\right).
 \end{align*}
 Now RH and the assumption on the absolute convergence of the series in \eqref{lambda hypergeometric eqn} involving the non-trivial zeros of $\zeta(s)$ imply \eqref{lambda K big-O}.
\end{proof}

\begin{proof}[Corollary \textup{\ref{lambda riesz}}][]
Let Re$(s)>0$, $y>0$, and let $\Phi(s):=\frac{\sqrt{\pi}y^{s+\frac{1}{2}}\Gamma(s)}{2\Gamma\left(s+\frac{3}{2}\right)}$. Then $\Phi(s)$ has a simple pole at $s=0$.  Also from \cite[p.~195, formula 5.35]{ober}, $\phi(x)=(y-x)^{1/2}$ if $x<y$, and $0$ else. Then
\begin{align}\label{before riesz complete}
	\int_{0}^{\infty}\frac{\phi(x)}{\sqrt{x}}\dd x=\frac{\pi y}{2},\hspace{5mm}
	\int_{0}^{\infty}\phi(x)x^{\rho_m-1}\dd x=\frac{\sqrt{\pi}y^{\rho_m+\frac{1}{2}}\Gamma(\rho_m)}{2\Gamma\left(\rho_m+\frac{3}{2}\right)}.
\end{align}
We next show
\begin{align}\label{riesz complete}
	\int_{0}^{\infty}\frac{\phi(x)}{\sqrt{x}}\sin\bigg(\frac{\pi nx}{2}+\frac{\pi}{4}\bigg)\dd x&=\frac{\pi y}{2\sqrt{2}}\bigg\{J_{0}\left(\frac{\pi n y}{4}\right)\left(\sin\left(\frac{\pi n y}{4}\right)+\cos\left(\frac{\pi n y}{4}\right)\right)\nonumber\\
	&\qquad\quad\quad+J_{1}\left(\frac{\pi n y}{4}\right)\left(\sin\left(\frac{\pi n y}{4}\right)-\cos\left(\frac{\pi n y}{4}\right)\right)\bigg\}.
	\end{align}
To that end, letting $\alpha=1/2, \beta=3/2, b=\pi n/2, \delta=1$ in \cite[p.~391, Formula \textbf{2.5.7.1}]{prud1}, one finds that
\begin{align}\label{diff hyp 1}
\int_{0}^{y}x^{-1/2}(y-x)^{1/2}\sin\left(\frac{\pi nx}{2}\right)\, dx=-\frac{i\pi y}{4}\left\{{}_1F_{1}\left(\frac{1}{2};2;\frac{i\pi n y}{2}\right)-{}_1F_{1}\left(\frac{1}{2};2;-\frac{i\pi n y}{2}\right)\right\}.
\end{align}	
Now from \cite[p.~580, Formula \textbf{7.11.2.12}]{prud3},
\begin{align*}
{}_1F_{1}\left(\frac{1}{2};2;z\right)=e^{z/2}\left(I_0\left(\frac{z}{2}\right)-I_1\left(\frac{z}{2}\right)\right), 
\end{align*}
so that
\begin{align*}
	{}_1F_{1}\left(\frac{1}{2};2;iz\right)=e^{iz/2}\left(J_0\left(\frac{z}{2}\right)-iJ_1\left(\frac{z}{2}\right)\right),
\end{align*}
and thus
\begin{align}\label{diff hyp 2}
	{}_1F_{1}\left(\frac{1}{2};2;iz\right)-{}_1F_{1}\left(\frac{1}{2};2;-iz\right)=2i\left(J_0\left(\frac{z}{2}\right)\sin\left(\frac{z}{2}\right)-J_1\left(\frac{z}{2}\right)\cos\left(\frac{z}{2}\right)\right).
	\end{align}
Therefore, from \eqref{diff hyp 1} and \eqref{diff hyp 2},
	\begin{align*}
		\int_{0}^{y}x^{-1/2}(y-x)^{1/2}\sin\left(\frac{\pi nx}{2}\right)\, dx=
		\frac{\pi y}{2}\left(J_0\left(\frac{\pi ny}{4}\right)\sin\left(\frac{\pi n y}{4}\right)-J_1\left(\frac{\pi n y}{4}\right)\cos\left(\frac{\pi n y}{4}\right)\right).
		\end{align*}
Similarly using  \cite[p.~391, Formula \textbf{2.5.7.1}]{prud1} with $\alpha=1/2, \beta=3/2, b=\pi n/2, \delta=0$, it can be seen that
 	\begin{align*}
 	\int_{0}^{y}x^{-1/2}(y-x)^{1/2}\cos\left(\frac{\pi nx}{2}\right)\, dx=
 	\frac{\pi y}{2}\left(J_0\left(\frac{\pi ny}{4}\right)\cos\left(\frac{\pi n y}{4}\right)+J_1\left(\frac{\pi n y}{4}\right)\sin\left(\frac{\pi n y}{4}\right)\right).
 \end{align*}
 The above two equations establish \eqref{riesz complete}. Now use \eqref{before riesz complete} and \eqref{riesz complete} in Theorem \ref{vsf lambda(n)} and divide both sides of the resulting equation by $\sqrt{y}$ to arrive at \eqref{lambda riesz eqn}.
 
 To prove \eqref{riesz bound}, we first note the well-known asymptotic expansion of $J_{\nu}(z)$ \cite[p.~199]{watson-1966a} as $|z|\to\infty$:
\begin{align*}
	J_{\nu}(z)&\sim\sqrt{\frac{2}{\pi z}}\left(\cos\left(w\right)\sum_{n=0}^{\infty}\frac{(-1)^n(\nu,2n)}{(2z)^{2n}}-\sin\left(w\right)\sum_{n=0}^{\infty}\frac{(-1)^n(\nu,2n+1)}{(2z)^{2n+1}}\right), \quad (|\arg(z)|<\pi),
\end{align*} 
where $w=z-\frac{1}{2}\pi\nu-\frac{1}{4}\pi$ and $(\nu, n)=\frac{\Gamma(\nu+n+1/2)}{\Gamma(n+1)\Gamma(\nu-n+1/2)}$. Then a simple calculation leads to the fact that as $n\to\infty$,
\begin{align*}
J_{0}\left(\frac{\pi n y}{4}\right)\left(\sin\left(\frac{\pi n y}{4}\right)+\cos\left(\frac{\pi n y}{4}\right)\right)+J_{1}\left(\frac{\pi n y}{4}\right)\left(\sin\left(\frac{\pi n y}{4}\right)-\cos\left(\frac{\pi n y}{4}\right)\right)=\frac{4}{\pi\sqrt{ny}}\left(1+O_{y}\left(n^{-\frac{3}{2}}\right)\right),	
\end{align*}
where the big-O term involves terms with negative powers of $y$. Thus, using Theorem \ref{c(n) over n thm}, as $y\to\infty$,
\begin{align*}
&\frac{\pi \sqrt{y}}{2}\sum_{n=1}^{\infty}\frac{c(n)}{\sqrt{n}}\bigg\{J_{0}\left(\frac{\pi n y}{4}\right)\left(\sin\left(\frac{\pi n y}{4}\right)+\cos\left(\frac{\pi n y}{4}\right)\right)+J_{1}\left(\frac{\pi n y}{4}\right)\left(\sin\left(\frac{\pi n y}{4}\right)-\cos\left(\frac{\pi n y}{4}\right)\right)\bigg\}\nonumber\\
&=1+O\left(\frac{1}{y}\right).
\end{align*}
Now RH and the assumption on the absolute convergence of the series in \eqref{lambda riesz eqn} involving the non-trivial zeros of $\zeta(s)$ imply \eqref{riesz bound}.
\end{proof}

	\subsection{Cohen type identity for $\lambda(n)$}
	Theorem \ref{cohen lambda thm} is proved here. Let $c=\textup{Re}(s)>1$. We do this by evaluating the integral 
	\begin{equation}\label{I one}
		I:=\frac{1}{2\pi i}\int_{(c)}\frac{\Gamma(2s-1)\zeta(2s-1)}{\Gamma(s)\zeta(s)}(2\pi x)^{-s}\dd s
	\end{equation}
	in two different ways. On one hand, using \eqref{c(n) ds} in the first step and duplication formula in the second, we have
	\begin{align}\label{one hand}
	I&=\sum_{n=1}^{\infty}c(n)\frac{1}{2\pi i}\int_{(c)}\frac{\Gamma(2s-1)}{\Gamma(s)}(2 \pi nx)^{-s}\dd s\nonumber\\
	&=\frac{1}{4\sqrt{\pi}}\sum_{n=1}^{\infty}c(n)\frac{1}{2\pi i}\int_{(c)}\Gamma\left(s-\frac{1}{2}\right)\left(\frac{\pi n x}{2}\right)^{-s}\dd s\nonumber\\
	&=\frac{1}{2\pi\sqrt{2x}}\sum_{n=1}^{\infty}\frac{c(n)}{\sqrt{n}}e^{-\pi n x/2}.
	\end{align}
	On the other hand, using \eqref{zetafe asym} twice, we see that
	\begin{align*}
		I&=\frac{1}{4\pi^2i}\int_{(c)}\frac{\zeta(2-2s)x^{-s}}{2\sin\left(\frac{\pi s}{2}\right)\zeta(1-s)}\dd s.
	\end{align*}
	Next, in order to use the fact that for Re$(s)<0$, $\sum_{n=1}^{\infty}\lambda(n)n^{s-1}=\zeta(2-2s)/\zeta(1-s)$, we shift the line of integration from $c=\textup{Re}(s)>1$ to $-2<c'=\textup{Re}(s)<0$ by constructing a rectangular contour $[c-iT_n, c+iT_n, c'+iT_n, c'-iT_n, c-iT_n]$, where $\{T_n\}_{n=1}^{\infty}$ is the sequence constructed in the proof of Theorem \ref{vsf lambda(n)} which gives \eqref{bound for reciprocal of zeta}. Along with \eqref{strivert}, this shows that the integrals along the horizontal segments tend to zero as $n\to\infty$.  We need to consider the contribution of the poles of the integrand in \eqref{I one}at $1/2$ (due to $\Gamma(2s-1)$) and at the non-trivial zeros of $\zeta(s)$ at $s=\rho_m, m\in\mathbb{Z}$. The residues at these poles are given by
	\begin{align}\label{residues cohen lambda}
	R_{1/2}=\frac{-1}{4\pi\sqrt{2x}\zeta(1/2)},\hspace{8mm}
	R_{\rho_m}=\frac{\zeta(2\rho_m-1)\Gamma(2\rho_m-1)}{\zeta'(\rho_m)\Gamma(\rho_m)}(2\pi x)^{-\rho_m}.
	\end{align}
	Therefore, by the residue theorem,
	\begin{align}\label{I two}
	I&=\frac{1}{4\pi^2i}\int_{(c')}\frac{x^{-s}}{2\sin\left(\frac{\pi s}{2}\right)}\sum_{n=1}^{\infty}\frac{\lambda(n)}{n^{1-s}}\dd s+R_{1/2}+\lim_{T_n\to\infty}\sum_{|\gamma_m|<T_n}R_{\rho_m}\nonumber\\
	&=\frac{1}{2\pi}\sum_{n=1}^{\infty}\frac{\lambda(n)}{n}\frac{1}{2\pi i}\int_{(c')}\frac{(x/n)^{-s}\dd s}{2\sin\left(\frac{\pi s}{2}\right)}+R_{1/2}+\lim_{T_n\to\infty}\sum_{|\gamma_m|<T_n}R_{\rho_m}.
\end{align}
In order to use the well-known formula,
\begin{align*}
\frac{1}{2\pi i}\int_{(\textup{Re}(s))}\frac{x^{-s}\dd s}{2\sin\left(\frac{\pi s}{2}\right)}=\frac{1}{\pi(1+x^2)},
	\end{align*}
which is valid for $0<$Re$(s)<2$, we again shift the line of integration from Re$(s)=c'$ to $0<$Re$(s)<2$, consider the contribution of the pole at $s=0$, thereby obtaining
\begin{align}\label{csc shifted}
\frac{1}{2\pi i}\int_{(c')}\frac{x^{-s}\dd s}{2\sin\left(\frac{\pi s}{2}\right)}=\frac{1}{\pi(1+x^2)}-\frac{1}{\pi}.
\end{align}
Hence, from \eqref{I two} and \eqref{csc shifted}, we obtain
	\begin{align}\label{other hand}
	I=\frac{-x^2}{2\pi^2}\sum_{n=1}^{\infty}\frac{\lambda(n)}{n(x^2+n^2)}+R_{1/2}+\lim_{T_n\to\infty}\sum_{|\gamma_m|<T_n}R_{\rho_m}.
	\end{align}
Now \eqref{cohen lambda eqn} follows from equating the expressions for $I$ in \eqref{one hand} and \eqref{other hand} (while using \eqref{residues cohen lambda}), and simplifying.
	\subsection{Ramanujan-Guinand type identity for $\lambda(n)$} Before proving Theorem \ref{rg liouville}, we begin with a lemma.
	
	\begin{lemma}\label{sum k bessel}
		For $c=\textup{Re}(s)<1/2$, we have
	\begin{align*}
	\frac{1}{2\pi i}\int_{(c)}\frac{\Gamma(\frac{1}{2}-s)}{\Gamma(\frac{1-s}{2})}t^{-s}\dd s=\frac{\exp{\left(-\tfrac{1}{8t^2}\right)}}{4\sqrt{2}\pi t^2}\left(K_{\frac{1}{4}}\left(\frac{1}{8t^2}\right)+K_{\frac{3}{4}}\left(\frac{1}{8t^2}\right)\right).
	\end{align*}
\end{lemma}
\begin{proof}
	From \cite[p.~666, formula, \textbf{8.4.23.4}]{prud3}, for Re$(s)<-|\textup{Re}(\nu)|$,
	\begin{align*}
		\int_{0}^{\infty}x^{s-1}e^{-1/(2x)}K_{\nu}\left(\frac{1}{2x}\right)\dd x=\frac{\sqrt{\pi}\Gamma(\nu-s)\Gamma(-\nu-s)}{\Gamma\left(\frac{1}{2}-s\right)}.
	\end{align*}
Employ the change of variable $x=4t^2$ followed by the replacement of $s$ by $s/2$ to obtain, for Re$(s)<-2|\textup{Re}(\nu)|$,
	\begin{align*}
	\int_{0}^{\infty}t^{s-1}\exp{\left(-\frac{1}{8t^2}\right)}K_{\nu}\left(\frac{1}{8t^2}\right)\dd t=\frac{\sqrt{\pi}\Gamma\left(\nu-\frac{s}{2}\right)\Gamma\left(-\nu-\frac{s}{2}\right)}{2^{s+1}\Gamma\left(\frac{1-s}{2}\right)}.
	\end{align*}
	Now use the above equation, once with $\nu=1/4$, and then with $\nu=3/4$, and add the resulting equations to see that for Re$(s)<-3/2$,
	\begin{align*}
		\int_{0}^{\infty}t^{s-1}\exp{\left(-\frac{1}{8t^2}\right)}\left(K_{\frac{1}{4}}\left(\frac{1}{8t^2}\right)+K_{\frac{3}{4}}\left(\frac{1}{8t^2}\right)\right)\dd t=\frac{\sqrt{\pi}}{2^{s+1}}\frac{\Gamma\left(\frac{1}{4}-\frac{s}{2}\right)\Gamma\left(-\frac{1}{4}-\frac{s}{2}\right)+\Gamma\left(\frac{3}{4}-\frac{s}{2}\right)\Gamma\left(-\frac{3}{4}-\frac{s}{2}\right)}{\Gamma\left(\frac{1-s}{2}\right)}.
		\end{align*}
		Now replace $s$ by $s-2$, divide both sides by $4\sqrt{2}\pi$, and then use the Mellin inversion (which is valid since the integral converges absolutely for Re$(s)<1/2$), to obtain the result.
\end{proof}

\begin{proof}[Theorem \textup{\ref{rg liouville}}][]
	We evaluate 
	\begin{equation*}
		\begin{split}
			I(x):=\frac{1}{2\pi i}\int_{(3/2)}\frac{\pi^{-s}\zeta(2s)\Gamma(s)}{\pi^{-s/2}\zeta(s)\Gamma(s/2)}x^{-s}ds
		\end{split}
	\end{equation*}
	in two different ways. On one hand, use \eqref{lambda ds} to get
	\begin{align}\label{4.93}
	I(x)=\sum_{n=1}^{\infty}\frac{\lambda(n)}{2\pi i}\int_{(3/2)}\frac{\Gamma(s)}{\Gamma(s/2)}(\sqrt{\pi}nx)^{-s}ds=\frac{x}{2}\sum_{n=1}^{\infty}n\lambda(n){e^{-\pi(xn)^2/4}},
	\end{align}
	using the duplication formula for the gamma function. On the other hand, we can shift the line of integration to Re$(s)=-1/2$, and evaluate $I(x)$ using the residue theorem. As done in the proof of Theorem \ref{vsf lambda(n)}, there exists an infinite sequence $\{T_n\}_{n=0}^{\infty}$ such that
		$$\lim_{T_n\to\infty}\int_{-1/2\pm\di T_n}^{3/2\pm\di T_n}\frac{\pi^{-(1-2s)/2}\zeta(1-2s)\Gamma(\frac{1-2s}{2})}{\pi^{-(1-s)/2}\zeta(1-s)\Gamma(\frac{1-s}{2})}x^{-s}ds=0.$$
	Hence invoking the residue theorem, using the symmetric form of the functional equation of $\zeta(s)$, considering the contributions of the poles at $1/2$ and at the non-trivial zeros $\rho_m$ of $\zeta(s)$, and using \eqref{c(n) ds} in the second step below, we find that
	\begin{equation}\label{4.94}
		\begin{split}
			I(x)&=\frac{1}{2\pi i}\int_{(-1/2)}\frac{\pi^{-(1-2s)/2}\zeta(1-2s)\Gamma(\frac{1-2s}{2})}{\pi^{-(1-s)/2}\zeta(1-s)\Gamma(\frac{1-s}{2})}x^{-s}ds+R(x)\\
			&=\sum_{n=1}^{\infty}\frac{c(n)}{n}\frac{1}{2\pi i}\int_{(-1/2)}\frac{\Gamma(\frac{1-2s}{2})}{\Gamma(\frac{1-s}{2})}\bigg(\frac{x}{\sqrt{\pi}n}\bigg)^{-s}ds+R(x),
		\end{split}
	\end{equation}
	where
	\begin{equation*}
		R(x)=\frac{\pi^{1/4}}{2\sqrt{x}\Gamma\left(\frac{1}{4}\right)\zeta\left(\frac{1}{2}\right)}+\lim_{T_n\to\infty}\sum_{|\gamma_m|<T_n}\frac{\zeta(2\rho_m)\Gamma(\rho_m)}{\zeta'(\rho_m)\Gamma(\frac{\rho_m}{2})}(\sqrt{\pi}x)^{-\rho_m}.
	\end{equation*}
	The result now follows from invoking Lemma \ref{sum k bessel} in \eqref{4.94} and then equating the resulting right-hand side with that of \eqref{4.93}.
	\end{proof}
	
\begin{remark}
Theorem \ref{rg liouville} can also be proved from the Vorono\"{\dotlessi} summation formula for $\lambda(n)$, that is, Theorem \ref{vsf lambda(n)}, by choosing $\phi(y)=\frac{1}{2}xye^{-\pi x^2y^2/4}$.
\end{remark}

\section{Proofs of the Voronoi summation formula for $d^{2}(n)$}\label{d(n) squared}

We begin with some lemmas.

\begin{lemma}\label{bn lemma}
	For $\Re(s)>1$, let $$\frac{\zeta^4(s)}{\zeta(2s-1)}=\sum_{n=1}^{\infty}\frac{b(n)}{n^{s}}.$$
	Then  $b(1)=1$, and at prime powers $p^k, k\geq0$, $b(p^k)$ is given by \eqref{bn exp}.
\end{lemma}
\begin{proof}
For Re$(s)>1$, $\zeta^4(s)=\sum_{n=1}^{\infty}d_4(n)n^{-s},$  where $d_4(n)$ is the number of ways in which $n$ can be written as a product of $4$ positive integers. 
	Also, $1/\zeta(2s-1)=\sum_{n=1}^{\infty}a(n)n^{-s},$ where
	$$
	a(n)=
	\begin{cases}
		m\mu(m),&\text {if $n=m^2$},\\
		0, &\text{otherwise.}
	\end{cases}
	$$ and Re$(s)>1$. 
	Since these two Dirichlet series converge absolutely in $\Re s>1$, so does that of $b(n)$, and we have
	$b(n)=(a*d_4)(n)$. Since both $a$ and $d_4$ are multiplicative, so is $b(n)$. Hence, $b(1)=1$, and $b$ is completely determined from its values at prime powers. We now evaluate  $d_4(p^k), k\geq0$. Let $p^k=a_1a_2a_3a_4$. Then $a_i=p^{x_i}, x_i\geq0$, for $1\leq i\leq 4$, so that $k=x_1+x_2+x_3+x_4$. Since the number of weak compositions\footnote{A weak composition of $m$ is a composition of $m$ in which $0$ is allowed to be a part.} of $k$ into exactly $r$ parts is given by $\binom{k+r-1}{r-1}$, it follows that $d_4(p^k)=\binom{k+3}{3}$. Thus, 
	$$b(p^k)=\sum_{d|p^k}a(d)d_4(p^k/d)=d_4(p^k)+p\mu(p)d_4(p^{k-2})=\binom{k+3}{3}-p\binom{k+1}{3}.
	$$
\end{proof}

\begin{lemma}\label{lemma K-Y}
	For $0<\Re(w)<1/8$,
	\begin{equation}\label{K-Y}
		\int_{0}^{\infty}y^{w-1}\Bigg(\frac{2}{\pi}K_0\big(4y^{1/4}\big)-Y_0\big(4y^{1/4}\big)\Bigg)\dd y=\frac{\G^2(w)}{\G^2(\frac{1}{2}-w)},
	\end{equation} 
	and the integral converges absolutely in this range.
\end{lemma}
\begin{proof}
This follows from the standard evaluations
 \cite[p.~115, formula 11.1; p.~93, formula 10.2]{ober}, namely, for Re$(s)>\pm$ Re$(z)$ and $a>0$,
\begin{equation*}
	\int_{0}^{\infty}x^{s-1}K_{z}(ax)\dd x=2^{s-2}a^{-s}\Gamma\left(\frac{s-z}{2}\right)\Gamma\left(\frac{s+z}{2}\right),
\end{equation*}
and,  for $\pm$Re$(z)<$Re$(s)<\frac{3}{2}$,
\begin{equation*}
	\int_{0}^{\infty}x^{s-1}Y_{z}(ax)\dd x=-\frac{1}{\pi}2^{s-1}a^{-s}\cos\left(\frac{1}{2}\pi(s-z)\right)\Gamma\left(\frac{s-z}{2}\right)\Gamma\left(\frac{s+z}{2}\right).
\end{equation*}
Now $K_0(x)\ll x^{-1/2}$ and $Y_0(x)\ll x^{-1/2}$ as $x\to\infty$ and for any $\epsilon>0$ we have $K_0(x)\ll x^{-\epsilon}$ and $Y_0(x)\ll x^{-\epsilon}$ as $x\to 0^{+}$.
Thus, the integral in \eqref{K-Y} converges absolutely for $0<\Re(w)<1/8$.
\end{proof}

\begin{lemma}\label{lemma g(x)}
Define $g(x)$ to be
\begin{equation}\label{eq815}
	g(x):=\frac{1}{2\pi i}\int_{(c)}\frac{\Gamma(\frac{1}{2}-w)\Gamma(1-w)}{\Gamma(w-\frac{1}{4})\Gamma(w+\frac{1}{4})}\Phi(1-2w)x^{-w}\dd w,
\end{equation}
where $\Phi(s)$ is a function satisfying \eqref{delta defn2} and $-\frac{1}{8}-\frac{\delta}{4}<c=\textup{Re}(s)<\frac{1}{8}$. Then, with $\phi$ defined in the statement of Theorem \ref{vsf dn squared}, we have
\begin{equation}\label{g(x) real}
g(x)=\int_{0}^{\infty}\frac{\phi(t)}{t\sqrt{2\pi x}}\cos\left(\frac{4}{t^{1/2}x^{1/4}} \right)\dd t.
\end{equation}
\end{lemma}
\begin{proof}
	
With $w=\sigma+it$, Stirling's formula \eqref{strivert} implies
\begin{equation}\label{gamma bound}
\frac{\Gamma(\frac{1}{2}-w)\Gamma(1-w)}{\Gamma(w-\frac{1}{4})\Gamma(w+\frac{1}{4})}\ll t^{\frac{3}{2}-4\sigma},
\end{equation}
and hence, with the help of \eqref{delta defn2},
$$
\frac{\Gamma(\frac{1}{2}-w)\Gamma(1-w)}{\Gamma(w-\frac{1}{4})\Gamma(w+\frac{1}{4})}\Phi(1-2w)\ll t^{-3/2-\delta-4\sigma}
$$
as $t \to \infty.$
Thus $g(x)$
converges when 
$c>-3/8-\delta/4$, and converges absolutely when
$c>-1/8-\delta/4.$
Let $c$ be such that\footnote{We enforce the restriction $c<1/8$ since this condition appears in the evaluation of the integral in \eqref{triple integral MT} while using Hardy's result for simplifying the kernel in the proof of Theorem \ref{vsf dn squared}.} $-1/8-\delta/4<c<1/8$.	

	From \eqref{gamma bound}, it is seen that the integral $$\frac{1}{2\pi i}\int_{(c)}\frac{\Gamma(\frac{1}{2}-w)\Gamma(1-w)}{\Gamma(w-\frac{1}{4})\Gamma(w+\frac{1}{4})}x^{-w}\dd w$$ is absolutely convergent for $c>5/8.$ 
But the line of integration in \eqref{eq815} is such that $c=\textup{Re}(w)<1/8$.
	The condition Re$(w)>5/8$ is necessary to satisfy the hypotheses of Parseval's formula, that is, Theorem \ref{Parseval} which is what we intend to apply to the integral in \eqref{eq815}. Hence we first shift the line of integration to $\textup{Re}(w)=5/7(>5/8)$.
	Showing that  the integrals along the horizontal segments approach zero as the height of the rectangular contour tends to $\pm \infty$ and considering the contribution of the pole of the integrand at $w=1/2$, by Cauchy's residue theorem, we get
	\begin{equation}\label{g(x)}
		\begin{split}
			g(x)&=\frac{1}{2\pi i}\int_{(5/7)}\frac{\Gamma(\frac{1}{2}-w)\Gamma(1-w)}{\Gamma(w-\frac{1}{4})\Gamma(w+\frac{1}{4})}\Phi(1-2w)x^{-w}\dd w+x^{-1/2}\frac{\sqrt{\pi}\Phi(0)}{\Gamma(1/4)\Gamma(3/4)}\\
			&=\frac{1}{2\pi i}\int_{(-3/7)}\frac{\Gamma(\frac{s}{2})\Gamma\left(\frac{1+s}{2}\right)x^{-\frac{(1-s)}{2}}}{2\Gamma\left(\frac{1}{4}-\frac{s}{2}\right)\Gamma\left(\frac{3}{4}-\frac{s}{2}\right)}\Phi(s)\dd s+x^{-1/2}\frac{\sqrt{\pi}\Phi(0)}{\Gamma(1/4)\Gamma(3/4)},
		\end{split}
	\end{equation}
	where, in the last step, we made a change of variable $w=(1-s)/2$.
%
%
To apply Parseval's formula, we now show that the hypotheses of Theorem \ref{Parseval} are satisfied. Let $F(s)=\frac{\Gamma(\frac{1-s}{2})\Gamma\left(1-\frac{s}{2}\right)x^{-s/2}}{2\Gamma\left(\frac{s}{2}-\frac{1}{4}\right)\Gamma\left(\frac{s}{2}+\frac{1}{4}\right)}$ and $G(s)=\Phi(s)$. First of all, the line of integration $\textup{Re}(s)=-3/7$ lies in the common strip of analyticity of $F(1-s)$ and $G(s)$. Next, \eqref{gamma bound} implies $F(10/7-it)\in L(-\infty, \infty)$. Moreover, 
 $\int_{0}^{\infty}t^{-10/7}|\phi(t)|\dd t<\infty$
since $\phi(t)=O(t^{1-\epsilon})$ as $t\to0^{+}$ and $\phi(t)=O(t^{-2+\epsilon})$ as $t\to\infty$. The latter, in turn, follows from the fact that $\Phi(s)$ is holomorphic in $-1<\textup{Re}(s)<2$.
Thus, invoking Theorem \ref{Parseval}, we see that
\begin{align}\label{parseval use}
 \frac{x^{-1/2}}{4\pi i}\int_{(-3/7)}\frac{\Gamma(\frac{s}{2})\Gamma\left(\frac{1+s}{2}\right)}{\Gamma\left(\frac{1}{4}-\frac{s}{2}\right)\Gamma\left(\frac{3}{4}-\frac{s}{2}\right)}\Phi(s)x^{s/2}\dd s&=\int_{0}^{\infty}P_x(t)\phi(t)\dd t,
\end{align}
where
\begin{align*}
	P_x(t):= \frac{1}{2\pi i}\int_{(10/7)}F(s)t^{-s}\dd s
	=
	\frac{1}{2\pi i}\int_{(5/7)}\frac{\Gamma(\frac{1}{2}-w)\Gamma(1-w)}{\Gamma(w-\frac{1}{4})\Gamma(w+\frac{1}{4})}(t^2x)^{-w}\dd w,
\end{align*}
where we made the change of variable $s=2w$.

 We now evaluate $P_x(t)$. To express it in terms of well-known functions, we need to shift the line of integration to $3/8<\lambda=\textup{Re}(w)<1/2$.  The residue theorem then gives
\begin{align}\label{eq818}
	P_x(t)
	=\frac{1}{2\pi i}\int_{(\lambda)}\frac{\Gamma(\frac{1}{2}-w)\Gamma(1-w)}{\Gamma(w-\frac{1}{4})\Gamma(w+\frac{1}{4})}(t^2x)^{-w}\dd w-\frac{\sqrt{\pi}x^{-1/2}}{t\Gamma(1/4)\Gamma(3/4)},
\end{align}
Next, we show
\begin{equation}\label{cos eqn}
	\frac{1}{2\pi i}\int_{(\lambda)}\frac{\Gamma(\frac{1}{2}-w)\Gamma(1-w)}{\Gamma(w-\frac{1}{4})\Gamma(w+\frac{1}{4})}x^{-w}\dd w=\frac{\cos\left(4x^{-1/4}\right)}{\sqrt{2\pi x}}.
\end{equation}
Replacing $w$ and $t$ in \eqref{cos} by $2-4w$ and $4x^{-1/4}$ respectively, we see that for $3/8<\lambda=\textup{Re}(w)<1/2$,
\begin{equation}\label{bef gmf}
\frac{1}{2\pi i}\int_{(\lambda)}-\Gamma(2-4w)\cos(2\pi w)2^{8w-2}x^{-w}\dd w=\frac{\cos\left(4x^{-1/4}\right)}{\sqrt{x}}.	
\end{equation}
Using the Gauss multiplication formula \cite[p.~52]{temme}
\begin{equation*}
	\prod_{k=1}^{m}\Gamma\left(z+\frac{k-1}{m}\right)=(2\pi)^{\frac{1}{2}(m-1)}m^{\frac{1}{2}-mz}\Gamma(mz)
\end{equation*}
 with $m=4$ and $z=1/2-w$, and the reflection formula for the Gamma function, the integrand in \eqref{bef gmf} simplifies to  $\frac{\sqrt{2\pi}\Gamma(\frac{1}{2}-w)\Gamma(1-w)}{\Gamma(w-\frac{1}{4})\Gamma(w+\frac{1}{4})}$, whence \eqref{cos eqn} follows. Hence from \eqref{eq818} and \eqref{cos eqn},
 \begin{align*}
 P_x(t)=\frac{\cos\left(4t^{-1/2}x^{-1/4}\right)}{t\sqrt{2\pi x}}-\frac{\sqrt{\pi}x^{-1/2}}{t\Gamma(1/4)\Gamma(3/4)},
 \end{align*}
which, along with \eqref{g(x)} and \eqref{parseval use}, leads to
\begin{align*}
g(x)&=\int_{0}^{\infty}\phi(t)\Bigg(\frac{\cos(4 t^{-1/2}x^{-1/4})}{t\sqrt{2\pi x}}-\frac{x^{-1/2}}{t}\frac{\sqrt{\pi}}{\Gamma(1/4)\Gamma(3/4)} \Bigg)\dd t+\frac{\Phi(0)\sqrt{\pi}x^{-1/2}}{\Gamma(1/4)\Gamma(3/4)}\nonumber\\
&=\int_{0}^{\infty}\phi(t)\Bigg(\frac{\cos(4 t^{-1/2}x^{-1/4})}{t\sqrt{2\pi x}}-\frac{x^{-1/2}}{t}\frac{\sqrt{\pi}}{\Gamma(1/4)\Gamma(3/4)} \Bigg)\dd t+\frac{\sqrt{\pi}x^{-1/2}}{\Gamma(1/4)\Gamma(3/4)}\int_{0}^{\infty}\frac{\phi(t)}{t}\dd t\nonumber\\
&=\int_{0}^{\infty}\phi(t)\frac{\cos(4 t^{-1/2}x^{-1/4})}{t\sqrt{2\pi x}}\dd t,
\end{align*}
where in the second step we used the fact that $\Phi(s)=\int_{0}^{\infty}t^{s-1}\phi(t)\dd t$ for any $s$ such that $-1<\textup{Re}(s)<2$. This completes the proof.

\end{proof}
We are now ready to prove the Vorono\"{\dotlessi} summation formula for $d^{2}(n)$.

\begin{proof}[Theorem \textup{\ref{vsf dn squared}}][]
Let $\tau$ be a number satisfying $0<\tau<\delta/2$, where $\delta>0$ is the number occurring in \eqref{delta defn2}. Define 
\begin{equation*}
I:=\frac{1}{2\pi i}\int_{(1+\tau)}\frac{\zeta^4(s)}{\zeta(2s)}\Phi(s)\dd s.
\end{equation*}
Using the fact that for Re$(s)>1$, $\zeta^4(s)/\zeta(2s)=\sum_{n=1}^{\infty}d^2(n)n^{-s}$ (which follows by letting $a=b=0$ in \eqref{sigma ab dirichlet}), and proceeding along the similar lines as \eqref{first expression for I}, we obtain
\begin{equation}\label{I-series}
I=\sum_{n=1}^{\infty}d^2(n)\phi(n).   
\end{equation}
We now to shift the line of integration to $-\tau=\Re s<0$ by constructing the contour $[1+\tau-iT_n, 1+\tau+iT_n, -\tau+iT_n, -\tau-iT_n]$, where $T_n$ is a number which belongs to the sequence $\{T_m\}_{m=1}^{\infty}$ constructed in the proof of Theorem \ref{vsf lambda(n)}. We first study the behavior of the  integrals along the horizontal segments $-\tau+\di T_n$ to $1+\tau+\di T_n$ for large values of $T_n$.
From \eqref{delta defn2}, \eqref{one over zeta},  \eqref{zeta-bound} and \eqref{zeta-bound1},
\begin{equation*}
	\begin{split}
	\left|	\int_{1+\tau+\di T_n }^{-\tau+\di T_n}\frac{\zeta^4(s)}{\zeta(2s)}\Phi(s)\dd s\right|
		&=\left|\int_{-\tau}^{1+\tau}\frac{\zeta^4(u+\di T_n)}{\zeta(2u+2\di T_n)}\Phi(u+\di T_n)\dd u\right|\\
		&\leq\left|\int_{-\tau}^{1/2}\frac{\zeta^4(u+\di T_n)}{\zeta(2u+2\di T_n)}\Phi(u+\di T_n)\dd u\right|+\left|\int_{1/2}^{1+\tau}\frac{\zeta^4(u+\di T_n)}{\zeta(2u+2\di T_n)}\Phi(u+\di T_n)\dd u\right|\\
		&\leq\int_{-\tau}^{1/2}T_n^{3}T_n^{\epsilon}T_n^{-3-\delta}\dd u +\int_{1/2}^{1+\tau}T_n^{1}T_n^{\epsilon}T_{n}^{-3-\delta}du.
	\end{split}
\end{equation*}
Now choose $\epsilon=\delta/2$ so that $\lim_{T_n\to\infty}\int_{-\tau}^{1/2}T_n^{3}T_n^{\epsilon}T_n^{-3-\delta}\dd u=0$. Also, $\lim_{T_n\to\infty}\int_{1/2}^{1+\tau}T_n^{1}T_n^{\epsilon}T_{n}^{-3-\delta}du=0$. In conclusion, $\lim_{T_n\to\infty}\int_{1+\tau+\di T_n }^{-\tau+\di T_n}\frac{\zeta^4(s)}{\zeta(2s)}\Phi(s)\dd s=0$. Similarly one can show that the integral along the horizontal segment $[-\tau-\di T_n, 1+\tau-\di T_n]$ approaches zero as $T_n\to\infty.$ Therefore, by the residue theorem,
\begin{equation}\label{I-shift}
	\begin{split}
		I=\frac{1}{2\pi i}\int_{(-\tau)}\frac{\zeta^4(s)}{\zeta(2s)}\Phi(s)\dd s+R_1+\lim_{T_n\to\infty}\sum_{|\gamma_m|<T_n}R_{\frac{\rho_m}{2}},
	\end{split}
\end{equation}
where $R_a$ denotes the residue of the integrand $\zeta^4(s)\Phi(s)/\zeta(2s)$ at $s=a$.
\begin{align}\label{1R}
		R_1&=\frac{1}{6}\lim_{s\to 1}\frac{d^3}{ds^3}(s-1)^4\frac{\zeta^4(s)}{\zeta(2s)}\Phi(s)\nonumber\\
		&=\frac{1}{\pi^8}\bigg[\Phi(1)\big(24 \gamma^3 \pi^6 -  72 \gamma\pi^6 \gamma_1 + 12 \pi^6 \gamma_2- 432 \gamma^2 \pi^4 \zeta'(2)+288 \pi^4 \gamma_1 \zeta'(2)+ 
		3456 \gamma\pi^2 (\zeta'(2))^2- 
		10368 (\zeta'(2))^3\nonumber\\ 
		&\quad- 
		288 \gamma \pi^4 \zeta''(2)+ 
		1728 \pi^2 \zeta'(2) \zeta''(2)-48\pi^4\zeta'''(2) \big)\nonumber\\
		&\quad+\Phi'(1)\left(36 \gamma^2 \pi^6- 24 \pi^6 \gamma_1 - 
		288 \gamma\pi^4\zeta'(2)+ 
		864 \pi^2  (\zeta'(2))^2 - 
		72 \pi^4\zeta''(2) \right)\nonumber\\
		&\quad+\Phi''(1)\left( 
		12 \gamma \pi^6 -
		36 \pi^4 \zeta'(2)\right)+\pi^6\Phi'''(1)\bigg]\nonumber\\
		&=\int_{0}^{\infty}(A_0+A_1\log x+A_2\log^2 x+A_3\log^3 x)\phi(x)\dd x,
\end{align}
where $A_0, A_1, A_2$ and $A_3$ are defined in \eqref{A0-A3}. The last step  is now justified. The Mellin inversion theorem \cite[p.~341, Theorem 1]{Mcla} implies $\Phi(s)=\int_{0}^{\infty}x^{s-1}\phi(x)\dd x$ in $-1<\textup{Re}(s)<2$. Owing to the fact that this integral is analytic in the given vertifcal strip, we can also differentiate this equation under the integral sign with respect to $s$. In particular, this gives
$$\Phi(1)=\int_{0}^{\infty}\phi(x)\dd x,\hspace{8mm}\Phi'(1)=\int_{0}^{\infty}\phi(x)\log(x)\dd x,$$

$$\Phi''(1)=\int_{0}^{\infty}\phi(x)\log^2(x)\dd x, \hspace{8mm}\Phi'''(1)=\int_{0}^{\infty}\phi(x)\log^3(x)\dd x,$$ which justifies the last step.
Also, if $\rho_m$ is the $m$th non-trivial zero of $\zeta(s)$, then
\begin{equation}\label{zeroR}
	\begin{split}
		R_{\frac{\rho_m}{2}}=\lim_{s\to\frac{\rho_m}{2}}\frac{\zeta^4(s)}{\zeta(2s)}\int_{0}^{\infty}\phi(x)x^{s-1}dx=\frac{\zeta^4(\frac{\rho_m}{2})}{2\zeta'(\rho_m)}\int_{0}^{\infty}\phi(x)x^{\frac{\rho_m}{2}-1}dx.
	\end{split}
\end{equation}
It remains to evaluate the integral  $\frac{1}{2\pi i}\int_{(-\tau)}\zeta^4(s)\Phi(s)/\zeta(2s)\dd s$. To that end, using \eqref{fe zeta}, we have
\begin{align}\label{integral rhs}
\frac{1}{2\pi i}\int_{(-\tau)}\frac{\zeta^4(s)}{\zeta(2s)}\Phi(s)\dd s&=\frac{\pi^{-3/2}}{2\pi i}\int_{(-\tau)}\frac{\zeta^4(1-s)}{\zeta(1-2s)}\frac{\Gamma^4(\frac{1-s}{2})}{\Gamma^4(\frac{s}{2})}\frac{\Gamma(s)}{\Gamma(\frac{1}{2}-s)}\Phi(s)\pi^{2s}\dd s\nonumber\\
&=\frac{\pi^{-3/2}}{2\pi i}\int_{(1+\tau)}\frac{\zeta^4(s)}{\zeta(2s-1)}\frac{\Gamma^4(\frac{s}{2})}{\Gamma^4(\frac{1-s}{2})}\frac{\Gamma(1-s)}{\Gamma(s-\frac{1}{2})}\Phi(1-s)\pi^{2(1-s)}\dd s\nonumber\\
&=\frac{\pi^{1/2}}{2\pi i}\sum_{n=1}^{\infty}b(n)\int_{(1+\tau)}\frac{\Gamma^4(\frac{s}{2})}{\Gamma^4(\frac{1-s}{2})}\frac{\Gamma(1-s)}{\Gamma(s-\frac{1}{2})}\Phi(1-s)(\pi^2 n)^{-s}\dd s,
\end{align}
where, in the last step, we invoked Lemma \eqref{bn lemma} and interchanged the order of summation and integration using absolute and uniform convergence. Indeed, this follows from the fact that Re$(s)>1$ and from \eqref{strivert}, \eqref{delta defn2}, and the choice of $\tau$, namely, $\tau<\delta/2$. 
Now let 
\begin{equation}\label{M(x)}
M(t):=\frac{1}{2\pi i}\int_{(1+\tau)}\frac{\Gamma^4(\frac{s}{2})}{\Gamma^4(\frac{1-s}{2})}\frac{\Gamma(1-s)}{\Gamma(s-\frac{1}{2})}\Phi(1-s)t^{-s}\dd s.
\end{equation}
By the change of variable $s=2w$ and an application of the duplication formula, we have
\begin{equation}\label{triple integral MT}
	\begin{split}
		M(t)&=\frac{2^{5/2}}{2\pi i}\int_{(\frac{1+\tau}{2})}\frac{\Gamma^4(w)}{\Gamma^4\left(\frac{1}{2}-w\right)}\frac{\Gamma\left(\frac{1}{2}-w\right)}{\Gamma\left(w-\frac{1}{4}\right)}\frac{\Gamma(1-w)}{\Gamma\left(w+\frac{1}{4}\right)}\Phi(1-2w)(4t)^{-2w}\dd w.\\
	\end{split}
\end{equation}
The main task now is to express $M(t)$ as a (multiple) integral whose integrand consists of $\phi$ and well-known functions. Even though it looks like the extension of the usual version of Parseval's formula for three functions would be applicable in this situation, one can see that it is, in fact, difficult to apply. This is where a variant of such a formula due to G. H. Hardy comes in very handy.

By extending Theorems A and C in \cite{hardy-mellin} to $p$ functions, Hardy derived a result \cite[p.~91]{hardy-mellin}  in the same paper whose special case for $p=3$ is
\begin{align*}
\frac{1}{2\pi i}\int_{(c)}f_1(w)f_2(w)f_3(w)t^{-w}\dd w=\int_{0}^{\infty}\int_{0}^{\infty}\phi_1(x)\phi_2(y)\phi_3\left(\frac{t}{xy}\right)\frac{\dd x\dd y}{xy}.
\end{align*}
The conditions on $f_j, \phi_j, 1\leq j\leq 3$, which make the above result valid are extensions of the conditions given for the two functions in Theorem A of \cite{hardy-mellin}.
We use this result with $\phi_1(x)=\phi_2(x)=\frac{2}{\pi}K_0\big(4x^{1/4}\big)-Y_0\big(4x^{1/4}\big)$ and $\phi_3(x)=g(x)$, where $g$ is the function in \eqref{g(x) real},  $\alpha=\epsilon, \beta=-1/8-\delta/4, \gamma=1/2$ and\footnote{We changed Hardy's notation $\delta$ to $\delta'$ so as to not get confused with our delta stemming from \eqref{delta defn2}.} $\delta'=1/8$, and then employ Lemmas \ref{lemma K-Y} and \ref{lemma g(x)}. Upon simplification, this gives
\begin{align}\label{M(x) evaluation}
M(t)&=\frac{1}{\sqrt{\pi}}\int_{0}^{\infty}\int_{0}^{\infty}\bigg(\frac{2}{\pi}K_0\big(4x^{1/4}\big)-Y_0\big(4x^{1/4}\big)\bigg)\bigg(\frac{2}{\pi}K_0\big(4y^{1/4}\big)-Y_0\big(4y^{1/4}\big)\bigg)\nonumber\\ &\quad\times\int_{0}^{\infty}\frac{\phi(z)}{tz}\cos\left(\frac{2(xy)^{1/4}}{\sqrt{tz}}\right)
\frac{\dd z \dd x \dd y}{\sqrt{xy}}\nonumber\\
&=\frac{4}{\sqrt{\pi}}\int_{0}^{\infty}\int_{0}^{\infty}\int_{0}^{\infty}\bigg(\frac{2}{\pi}K_0\big(4\sqrt{x}\big)-Y_0\big(4\sqrt{x}\big)\bigg)\bigg(\frac{2}{\pi}K_0\big(4\sqrt{y}\big)-Y_0\big(4\sqrt{y}\big)\bigg)\frac{\phi(z)}{tz}\cos\bigg(\frac{2\sqrt{xy}}{\sqrt{tz}}\bigg)
{\dd z \dd x \dd y},
\end{align}
where, in the last step, we replaced $x$ and $y$ by $x^2$ and $y^2$.
Therefore, from \eqref{integral rhs}, \eqref{M(x)} and \eqref{M(x) evaluation},
\begin{align}\label{final M}
\frac{1}{2\pi i}\int_{(-\tau)}\frac{\zeta^4(s)}{\zeta(2s)}\Phi(s)\dd s&=\frac{4}{\pi^2 }\sum_{n=1}^{\infty}\frac{b(n)}{n}\int_{0}^{\infty}\int_{0}^{\infty}\int_{0}^{\infty}\bigg(\frac{2}{\pi}K_0\big(4\sqrt{x}\big)-Y_0\big(4\sqrt{x}\big)\bigg)\nonumber\\
&\quad\times\bigg(\frac{2}{\pi}K_0\big(4\sqrt{y}\big)-Y_0\big(4\sqrt{y}\big)\bigg)\frac{\phi(z)}{z}\cos\bigg(\frac{2\sqrt{xy}}{\pi\sqrt{nz}}\bigg)
{\dd z \dd x \dd y}.
\end{align}
Finally, from \eqref{I-series}, \eqref{I-shift}, \eqref{1R}, \eqref{zeroR} and \eqref{final M}, we arrive at \eqref{vsf dn squared eqn}.

\end{proof}



\section{Proofs of the results on $\sigma_a(n)\sigma_b(n)$}\label{sigma ab}

\subsection{Cohen type identity for $\sigma_a(n)\sigma_b(n)$}

Before we embark upon the proof of Theorem \ref{cohen sigma ab}, we need the following two lemmas. 

\begin{lemma}\label{I evaluation}
	Let $x\in\mathbb{C}, |x|\neq1$, and $-1<\textup{Re}(a), \textup{Re}(b), \textup{Re}(a-b), \textup{Re}(a+b)<1$. Define $I(x)$ by
	\begin{equation*}
	I(x)=\frac{1}{2\pi i}\int_{(c)}\frac{\sec\big(\frac{\pi}{2}s\big)\sec(\frac{\pi}{2}(s-a))\sec(\frac{\pi}{2}(s-b))\sec(\frac{\pi}{2}(s-a-b))}{\sec(\frac{\pi}{2}(2s-a-b-1))}x^{-s} \dd s,
	\end{equation*}
	where \scriptsize$\max{\{-1, -1+\textup{Re}(a), -1+\textup{Re}(b), -1+\textup{Re}(a+b)\}}<c=\textup{Re}(s)<\min{\{1, 1+\textup{Re}(a), 1+\textup{Re}(b), 1+\textup{Re}(a+b)\}}$. \normalsize Then
	\begin{align*}
	I(x)=\frac{2}{\pi}\textup{cosec}\left(\frac{\pi a}{2}\right)\textup{cosec}\left(\frac{\pi b}{2}\right)\frac{x(x^{-a}-1)(x^{-b}-1)}{x^2-1}.
	\end{align*}
\end{lemma}
\begin{proof}
Assume $|x|<1$. Construct a rectangular contour $[c-iT, c+iT, -N+iT, -N-iT, c-iT]$, where 
\scriptsize\begin{equation*}
\max{\left\{-2m-3, -2m-3+a, -2m-3+b, -2m-3+a+b\right\}}<N<\min{\left\{-2m-1, -2m-1+a, -2m-1+b, -2m-1+a+b\right\}}
\end{equation*} 
\normalsize
for $m\geq0$. The integrals over the horizontal segments tend to zero as $T\to\infty$ as can be seen by the exponential decay of the integrand, which, in turn, can be inferred by first rewriting the integrand using $\Gamma\left(\frac{1}{2}+w\right)\Gamma\left(\frac{1}{2}-w\right)=\pi/\cos(\pi w)$ and then using Stirling's formula \eqref{strivert}. Hence by the residue theorem,
\begin{align*}
I(x)&=\frac{1}{2\pi i}\int_{-N-i\infty}^{-N+i\infty}\frac{\sec\big(\frac{\pi}{2}s\big)\sec(\frac{\pi}{2}(s-a))\sec(\frac{\pi}{2}(s-b))\sec(\frac{\pi}{2}(s-a-b))}{\sec(\frac{\pi}{2}(2s-a-b-1))}x^{-s} \dd s\nonumber\\
&\quad+\sum_{n=0}^{m}\left(R_{-2n-1}+R_{-2n-1+a}+R_{-2n-1+b}+R_{-2n-1+a+b}\right),
\end{align*}
where $R_{-2n-1+u}=\pm2x^{1-u}\textup{cosec}\left(\frac{\pi a}{2}\right)\textup{cosec}\left(\frac{\pi b}{2}\right)x^{2n}/\pi$, where $u$ takes the values $0, a, b$ or $a+b$, and we take minus sign when $u=0$ or $a+b$ and plus sign when $u=a$ or $b$. 

We now let $N\to\infty$ in the above equation. It is not difficult to see that the integral on the right-hand side tends to zero. Indeed, \eqref{strivert} and the fact that $|x|<1$ shows that the integral approaches zero in the limit. This leads to 
\begin{align}\label{6.5}
	I(x)&=\sum_{n=0}^{\infty}\left(R_{-2n-1}+R_{-2n-1+a}+R_{-2n-1+b}+R_{-2n-1+a+b}\right)\nonumber\\
	&=\frac{2x}{\pi}\textup{cosec}\left(\frac{\pi a}{2}\right)\textup{cosec}\left(\frac{\pi b}{2}\right)\left(1-x^{-a}-x^{-b}+x^{-a-b}\right)\sum_{n=0}^{\infty}x^{2n}\nonumber\\
	&=\frac{2}{\pi}\textup{cosec}\left(\frac{\pi a}{2}\right)\textup{cosec}\left(\frac{\pi b}{2}\right)\frac{x(x^{-a}-1)(x^{-b}-1)}{x^2-1}.
	\end{align}
This proves the result for $|x|<1$. Now if $|x|>1$, we shift the line of integration to $+\infty$, and proceed along the similar lines as above. This leads to the same evaluation of $I(x)$ as in \eqref{6.5}.
\end{proof}

To the best of our knowledge, the following evaluation of Meijer $G$-function (defined in \eqref{MeijerG}) seems to be new.
\begin{lemma}\label{Meijer G evaluation}
For $c=\textup{Re}(s)>\max{\{0, \textup{Re}(a), \textup{Re}(b), \textup{Re}(a+b)\}}$,
\begin{align}\label{new Meijer G evaluation}
&\frac{1}{2\pi i}\int_{(c)}\frac{\Gamma(s)\Gamma(s-a)\Gamma(s-b)\Gamma(s-a-b)}{\Gamma\left(s-\frac{a+b+1}{2}\right)\Gamma\left(s-\frac{a+b}{2}\right)}z^{-s}\dd s\nonumber\\&=G_{2, \, \, 4}^{4, \, \, 0} \left(\begin{matrix}
\frac{-a-b}{2}, \frac{-a-b-1}{2}\\
	0, -a, -b, -a-b
\end{matrix} \Bigg| \, z \right)\nonumber\\
&=\frac{1}{\sqrt{\pi}}z^{ \frac{1-a-b}{2}}\left(K_{a-1}(\sqrt{z})K_{b}(\sqrt{z})+K_{b-1}(\sqrt{z})K_{a}(\sqrt{z})+\frac{(a+b-1)}{\sqrt{z}}K_{a}(\sqrt{z})K_{b}(\sqrt{z})\right).
\end{align}
\end{lemma}
\begin{proof}
	The first equality follows from the definition of Meijer $G$-function. 
	Now it is known \cite{wolfram another} that
	\begin{align}\label{known meijer g}
		G_{2, \, \, 4}^{4, \, \, 0} \left(\begin{matrix}
			A, A+\frac{1}{2}\\
			B, C, 2A-C, 2A-B
		\end{matrix} \Bigg| \, z \right)=\frac{2}{\sqrt{\pi}}z^{A}K_{B-C}(\sqrt{z})K_{B+C-2A}(\sqrt{z}).
	\end{align}
	Let $A=\frac{-a-b}{2}, B=0$ and $C=-a$. Then the resulting Meijer $G$-function differs from the one in \eqref{new Meijer G evaluation} only in one of the top parameters with the difference between them being $1$. To address this issue, we use the identity \cite[p.~621, Formula (37)]{prud3}, namely, for $n\leq p-1$,
	\begin{align*}
		\frac{d}{dz}\left[z^{1-a_p}G_{p, \, \, q}^{m, \, \, n} \left(\begin{matrix}
			a_1,\cdots,a_n, a_{n+1},\cdots, a_p\\
			b_1,\cdots,b_m, b_{m+1},\cdots, b_q
		\end{matrix} \Bigg| \, z \right)\right]=-z^{-a_p}G_{p, \, \, q}^{m, \, \, n} \left(\begin{matrix}
		a_1,\cdots,a_n, a_{n+1},\cdots, a_p-1\\
		b_1,\cdots,b_m, b_{m+1},\cdots, b_q
		\end{matrix} \Bigg| \, z \right)
	\end{align*}
	Therefore letting $m=4, n=0, p=2, q=4$ and $a_p=(1-a-b)/2$ in the above identity, and using \eqref{known meijer g} in the second step, we see that
	\begin{align*}
	&G_{2, \, \, 4}^{4, \, \, 0} \left(\begin{matrix}
		\frac{-a-b}{2}, \frac{-a-b-1}{2}\\
		0, -a, -b, -a-b
	\end{matrix} \Bigg| \, z \right)\nonumber\\
	&=-z^{\frac{1-a-b}{2}}\frac{d}{dz}\left(z^{\frac{1+a+b}{2}}G_{2, \, \, 4}^{4, \, \, 0} \left(\begin{matrix}
	\frac{-a-b}{2}, \frac{-a-b+1}{2}\\
	0, -a, -b, -a-b
	\end{matrix} \Bigg| \, z \right)\right)\nonumber\\
	&=-\frac{2}{\sqrt{\pi}}z^{\frac{1-a-b}{2}}\frac{d}{dz}\left(\sqrt{z}K_a(\sqrt{z})K_b(\sqrt{z})\right)\nonumber\\
	&=\frac{1}{\sqrt{\pi}}z^{ \frac{1-a-b}{2}}\left(K_{a-1}(\sqrt{z})K_{b}(\sqrt{z})+K_{b-1}(\sqrt{z})K_{a}(\sqrt{z})+\frac{(a+b-1)}{\sqrt{z}}K_{a}(\sqrt{z})K_{b}(\sqrt{z})\right),
	\end{align*}
	where, in the last step, we used the standard formula for the differentiation of the $K$-Bessel function \cite[p.~929, \textbf{8.486.11}]{gr}, namely,
 $\frac{d}{dw}K_{\nu}(w)=-\frac{1}{2}\left(K_{\nu-1}(w)+K_{\nu+1}(w)\right)$. This completes the proof.
\end{proof}

\begin{remark}
We note a related formula which appears in \cite[p.~647, Formula (55)]{brychkov}:
\begin{align*}
G_{2, \, \, 4}^{4, \, \, 0} \left(\begin{matrix}
	\frac{1}{2}, 1\\
	a+b, a-b, b-a, -a-b
\end{matrix} \Bigg| \, z \right)&=\frac{2}{(a+b)\sqrt{\pi}}K_{2a}(\sqrt{z})K_{2b}(\sqrt{z})\nonumber\\
&\quad+\frac{\sqrt{z}}{(a^2-b^2)\sqrt{\pi}}\left(K_{2a-1}(\sqrt{z})K_{2b}(\sqrt{z})-K_{2a}(\sqrt{z})K_{2b-1}(\sqrt{z})\right).
\end{align*}
\end{remark}
Armed with these results, we now ready to prove the Cohen-type identity for $\sigma_a(n)\sigma_b(n)$.

\begin{proof}[Theorem \textup{\ref{cohen sigma ab}}][]
	We assume
	\small\begin{align*}
	\max{\{-1, -1+\textup{Re}(a), -1+\textup{Re}(b), -1+\textup{Re}(a+b)\}}<c<\min{\left\{0, \textup{Re}(a), \textup{Re}(b), \textup{Re}(a+b), \textup{Re}\left(\tfrac{a+b+1}{2}\right)\right\}}. 
	\end{align*}
	\normalsize
	The above strip is of positive width because of the conditions on $a$ and $b$ given in the hypotheses. 
Let
\begin{equation*}
S(a, b, x):=\sum_{n=1}^{\infty}{\sigma_{a}(n)\sigma_{b}(n)}\frac{x\big(x^{-a}-n^{-a}\big)\big(x^{-b}-n^{-b}\big)}{x^2-n^2}.
\end{equation*}	
Then an application of $\sigma_{s}(n)=\sigma_{-s}(n)n^s$ in the first step and Lemma \ref{I evaluation} in the second leads to
\begin{align*}
S(a, b, x)&=\sum_{n=1}^{\infty}\frac{\sigma_{-a}(n)\sigma_{-b}(n)}{n}\frac{\frac{x}{n}\left(\left(\frac{x}{n}\right)^{-a}-1\right)\left(\left(\frac{x}{n}\right)^{-b}-1\right)}{\frac{x^2}{n^2}-1}\nonumber\\
&=\frac{\pi}{2}\sin\left(\frac{\pi a}{2}\right)\sin\left(\frac{\pi b}{2}\right)\sum_{n=1}^{\infty}\frac{\sigma_{-a}(n)\sigma_{-b}(n)}{n}\nonumber\\
&\quad\times\frac{1}{2\pi i}\int_{(c)}\frac{\sec\big(\frac{\pi}{2}s\big)\sec(\frac{\pi}{2}(s-a))\sec(\frac{\pi}{2}(s-b))\sec(\frac{\pi}{2}(s-a-b))}{\sec(\frac{\pi}{2}(2s-a-b-1))}\left(\frac{x}{n}\right)^{-s} \dd s\nonumber\\
&=\frac{\pi}{2}\sin\left(\frac{\pi a}{2}\right)\sin\left(\frac{\pi b}{2}\right)\nonumber\\
&\quad\times\frac{1}{2\pi i}\int_{(c)}\sum_{n=1}^{\infty}\frac{\sigma_{-a}(n)\sigma_{-b}(n)}{n^{1-s}}\frac{\sec\big(\frac{\pi}{2}s\big)\sec(\frac{\pi}{2}(s-a))\sec(\frac{\pi}{2}(s-b))\sec(\frac{\pi}{2}(s-a-b))}{\sec(\frac{\pi}{2}(2s-a-b-1))}x^{-s} \dd s\nonumber\\
&=\frac{\pi}{2}\sin\left(\frac{\pi a}{2}\right)\sin\left(\frac{\pi b}{2}\right)\frac{1}{2\pi i}\int_{(c)}\frac{\zeta(1-s)\zeta(1-s+a)\zeta(1-s+b)\zeta(1-s+a+b)}{\zeta(2-2s+a+b)}\nonumber\\
&\quad\times\frac{\sec\big(\frac{\pi}{2}s\big)\sec(\frac{\pi}{2}(s-a))\sec(\frac{\pi}{2}(s-b))\sec(\frac{\pi}{2}(s-a-b))}{\sec(\frac{\pi}{2}(2s-a-b-1))}x^{-s} \dd s,
\end{align*}
where we interchanged the order of summation and integration because of the uniform convergence of the associated series in the above strip, and where we used \eqref{sigma ab dirichlet} in the last step.

Next, use \eqref{zetafe asym} for each of the zeta functions occurring in the integrand of the above equation and simplify to obtain
\begin{align*}
S(a, b, x)&=2(2\pi)^{a+b}\sin\left(\frac{\pi a}{2}\right)\sin\left(\frac{\pi b}{2}\right)\frac{1}{2\pi i}\int_{(c)}\frac{\Gamma(s)\Gamma(s-a)\Gamma(s-b)\Gamma(s-a-b)}{\Gamma(2s-a-b-1)}\nonumber\\
&\quad\times\frac{\zeta(s)\zeta(s-a)\zeta(s-b)\zeta(s-a-b)}{\zeta(2s-a-b-1)}(4\pi^2x)^{-s}\dd s.
\end{align*} 
Now we wish to invoke \eqref{cab dirichlet}. It necessitates shifting the line of integration from Re$(s)=c$ to Re$(s)=\eta$, where $\eta$ is defined in \eqref{eta}. Clearly, Stirling's formula \eqref{strivert} implies that the integrals along the horizontal segments tend to zero as the height of the rectangular contour approaches $\infty$. The integrand has simple poles at $s=0, a, b, a+b$ (due to the gamma functions), at $s=1, 1+a, 1+b, 1+a+b$ (due to the zeta functions), and at the non-trivial zeros of $\zeta(2s-a-b-1)$ at $s=(1+\rho_m+a+b)/2$, where $\rho_m$ are the non-trivial zeros of $\zeta(s)$. (Here we have assumed the simplicity of the zeros.)

The residues at these poles can be easily calculated to be the ones given in \eqref{residues 1} and \eqref{residues 2}. Therefore invoking the residue theorem, expressing the quotient of zeta functions by the left-hand side of \eqref{cab dirichlet} and interchanging the order of summation and integration (valid due to uniform convergence), we find that
\begin{align*}
S(a, b, x)&=	2(2\pi)^{a+b}\sin\left(\tfrac{\pi a}{2}\right)\sin\left(\tfrac{\pi b}{2}\right)\Bigg\{\sum_{n=1}^{\infty}C_{a, b}(n)\frac{1}{2\pi i}\int_{(\eta)}\frac{\Gamma(s)\Gamma(s-a)\Gamma(s-b)\Gamma(s-a-b)}{\Gamma(2s-a-b-1)}\left(4\pi^2nx\right)^{-s}\dd s\nonumber\\
&\quad-\Bigg(\sum_{k=0}^{1}R_k(x)+ R_{k+a}(x)+R_{k+b}(x)+R_{k+a+b}(x)+\lim_{T_n\to\infty}\sum_{|\gamma_m|<T_n}R_{\rho_m,a,b}(x)\Bigg)\Bigg\}.
\end{align*}
This results in \eqref{sigma ab final} upon invoking Lemma \ref{Meijer G evaluation}.

\end{proof}

\begin{proof}[Corollary \textup{\ref{cohen d(n) squared}}][]
Divide both sides of \eqref{sigma ab final} by $ab$ and then let $a\to0$ and $b\to0$. Using the fact that $\sigma_s(n)n^{-\frac{s}{2}}=O\left(n^{\frac{1}{2}|\textup{Re}(s)|+\epsilon}\right)$, it is easily seen that the series on the left-hand side of \eqref{sigma ab final} is uniformly convergent in any compact interval of $(-1, 1)$, viewed as a function of the complex variable $a$ or $b$. Hence we can interchange the order of limits and summation. The same can be done on the right-hand side as well. This leads to \eqref{cohen d(n) squared} upon simplification\footnote{The residue $R_0(x)$ in this case was calculated using \emph{Mathematica}.}.
\end{proof}

\subsection{Ramanujan-Guinand type identity for $\sigma_a(n)\sigma_b(n)$}

We prove Theorem \ref{rg sigma ab} in this subsection. Let $b=c=2$ and then $\a=s-a/2-b/2, \mu=a/2, \nu=b/3$ in \cite[p.~384, formula \textbf{2.16.33.1}]{prud2} so that for $c=\textup{Re}(s)>\max{\{0, \textup{Re}(a), \textup{Re}(b), \textup{Re}(a+b)\}}$,
\begin{equation*}
	\begin{split}
	\frac{1}{2\pi i}	\int_{(c)}\frac{\Gamma(\frac{s}{2})\Gamma(\frac{s-a}{2})\Gamma(\frac{s-b}{2})\Gamma(\frac{s-a-b}{2})}{\Gamma(\frac{2s-a-b}{2})}x^{-s}\dd s=8x^{(-a-b)/2}K_{a/2}(2x)K_{b/2}(2x).
	\end{split}
\end{equation*}
Now let $c=\textup{Re}(s)>1+\max{(0, \textup{Re}(a), \textup{Re}(b), \textup{Re}(a+b))}$ and define
\begin{equation}
I(x):=	\frac{1}{2\pi i}\int_{(c)}\frac{\Gamma(\frac{s}{2})\zeta(s)\Gamma(\frac{s-a}{2})\zeta(s-a)\Gamma(\frac{s-b}{2})\zeta(s-b)\Gamma(\frac{s-a-b}{2})\zeta(s-a-b)}{\Gamma(\frac{2s-a-b}{2})\zeta(2s-a-b)}x^{-s}\dd s.
\end{equation}
Then, using \eqref{sigma ab dirichlet}, we have
\begin{align*}
I(x)=8x^{(-a-b)/2}\sum_{n=1}^{\infty}\frac{\sigma_{a}(n)\sigma_{b}(n)}{n^{(a+b)/2}}K_{a/2}(2nx)K_{b/2}(2nx).
\end{align*}
On the other hand, invoking \eqref{zetafe sym} in the first step, we get
\begin{align*}
	I(x)=\hspace{-1mm}\frac{\pi^{-a-b-\frac{3}{2}}}{2\pi i}\hspace{-2mm}\int_{(c)}\hspace{-2mm}\frac{\Gamma\left(\frac{1-s}{2}\right)\zeta(1-s)\Gamma\left(\frac{1-s+a}{2}\right)\zeta(1-s+a)\Gamma\left(\frac{1-s+b}{2}\right)\zeta(1-s+b)\Gamma\left(\frac{1-s+a+b}{2}\right)\zeta(1-s+a+b)}{\Gamma\left(\frac{1-2s+a+b}{2}\right)\zeta(1-2s+a+b)\pi^{-2s}x^{s}}\dd s.
\end{align*}
Now shift the line of integration to $c'=\textup{Re}(s)<\min{(0, \textup{Re}(a), \textup{Re}(b), \textup{Re}((a+b)/2), \textup{Re}(a+b))}$ so as to be able to use \eqref{cab dirichlet}, consider the contributions of the poles at  $(\rho_m+a+b)/2$, and at $k, k+a, k+b$ and $k+a+b$, where $k=0$ or $1$, and invoke the residue theorem to get
\begin{align*}
I(x)&=\pi^{-a-b-\frac{3}{2}}\sum_{n=1}^{\infty}\frac{C_{-a, -b}(n)}{n}\frac{1}{2\pi i}\int_{(c')}\frac{\Gamma\left(\frac{1-s}{2}\right)\Gamma\left(\frac{1-s+a}{2}\right)\Gamma\left(\frac{1-s+b}{2}\right)\Gamma\left(\frac{1-s+a+b}{2}\right)}{\Gamma\left(\frac{1-2s+a+b}{2}\right)}\left(\frac{x}{\pi^2n}\right)^{-s}\dd s +R(x)\nonumber\\
&=\frac{\pi^{-a-b-1}}{2^{\frac{a+b-3}{2}}}\sum_{n=1}^{\infty}\frac{C_{-a, -b}(n)}{n}\frac{1}{2\pi i}\int_{(\frac{c'}{2})}\frac{\Gamma\left(\frac{1}{2}-w\right)\Gamma\left(\frac{1+a}{2}-w\right)\Gamma\left(\frac{1+b}{2}-w\right)\Gamma\left(\frac{1+a+b}{2}-w\right)}{\Gamma\left(\frac{1+a+b}{4}-w\right)\Gamma\left(\frac{3+a+b}{4}-w\right)}\left(\frac{x^2}{4\pi^4n^2}\right)^{-w}\dd w\nonumber\\
&\quad+R(x),
\end{align*}
where $R(x)$ is the sum of the residues in \eqref{residues 3} and \eqref{residues 4}. In the last step, we used the duplication formula and then employed the change of variable $s=2w$. The integral in the last step can be easily seen to be the Meijer $G$-function claimed in \eqref{jhep}, which completes the proof.

\begin{proof}[Corollary \textup{\ref{rg d(n) squared cor}}][]
Let $a=b=0$ in Theorem \ref{rg sigma ab}. We found using \emph{Mathematica} that $x$ times the residue $\tilde{R}_1(x)$ is a polynomial in $\log(x)$. Althought it can be explicitly written down, we avoid giving it here as it is quite complicated.
	\end{proof}
	
\section{Proofs of Theorem \ref{voronoi mu(n)} and Ramanujan's identity \eqref{mr gen}}\label{mr gen section}

\begin{proof}[Theorem \textup{\ref{voronoi mu(n)}}][]
Using the Mellin inversion formula, we get
$\Phi(s)=\int_{0}^{\infty}\phi(x)x^{s-1}\dd x $ for $\Re s>0$ and $\Phi(s)=\int_{0}^{\infty}(\phi(x)-k)x^{s-1}\dd x $ for $-1<\Re s<0$, where $k$ is the residue of $\Phi(s)$ at $s=0$.

Choose $\tau$ such that $0<\tau<\delta$, where $\delta$ is the number defined in  \eqref{delta mu(n)}, and define
\begin{align}\label{before intermediate}
I:=\frac{1}{2\pi i}\int_{(1+\tau)}\frac{\Phi(s)}{\zeta(s)}\dd s.
\end{align}
Then, one one hand, 
\begin{equation*}
	\begin{split}
		I=\frac{1}{2\pi i}\int_{(1+\tau)}\frac{\Phi(s)}{\zeta(s)}\dd s=\sum_{n=1}^{\infty}\mu(n)\frac{1}{2\pi i}\int_{(1+\tau)}\Phi(s)n^{-s}\dd s
		=\sum_{n=1}^{\infty}\mu(n)\phi(n).
	\end{split}
\end{equation*}
If we now shift the line of integration to Re$(s)=-\tau$, and choose the sequence $\{T_n\}$ as in Theorem \ref{vsf lambda(n)} so that the integrals along the horizontal segments go to zero. Then by Cauchy's residue theorem, we have
\begin{equation}\label{intermediate step for I}
I=\frac{1}{2\pi i}\int_{(-\tau)}\frac{\Phi(s)}{\zeta(s)}\dd s+\Res_{s=0}\frac{\Phi(s)}{\zeta(s)}+\lim_{T_n\to \infty}\sum_{|\gamma_m|<T_n}\frac{1}{\zeta'(\rho_m)}\int_{0}^{\infty}\phi(x)x^{\rho_m-1}\dd x. 
\end{equation}
From \eqref{before intermediate} and \eqref{intermediate step for I}, we have
\begin{align}\label{eq96}
\sum_{n=1}^{\infty}\mu(n)\phi(n)=\frac{1}{2\pi i}\int_{(-\tau)}\frac{\Phi(s)}{\zeta(s)}\dd s-2k+\lim_{T_n\to \infty}\sum_{|\gamma_m|<T_n}\frac{1}{\zeta'(\rho_m)}\int_{0}^{\infty}\phi(x)x^{\rho_m-1}\dd x.
\end{align}
With a fair amount of calculation using Parseval's formula \eqref{Parseval1}, one can see that
\begin{align}
\frac{1}{2\pi i}\int_{(-\tau)}\frac{\Phi(s)}{\zeta(s)}\dd s=4\sum_{n=1}^{\infty}\frac{\mu(n)}{n}\int_{0}^{\infty}(\phi(t)-k)\bigg(\frac{\cos^2(\pi/nt)-1}{t}\bigg)\dd t, 
\end{align}
which, along with \eqref{eq96}, completes the proof.
\end{proof}

We prove Corollary \ref{mr gen corollary} here.
Let $\phi(x)=\frac{\sqrt{a}}{x}\tilde\phi(\frac{a}{x})$ in Theorem \ref{voronoi mu(n)}, where $\tilde{\phi}$ is chosen so that the hypotheses of Theorem \ref{voronoi mu(n)} are satisfied, and, in addition, there is no residue of $\Phi(s)$ at $s=0$, that is, $k$ in Theorem \ref{voronoi mu(n)} is equal to $0$. Then
\begin{align*}
		&\sqrt{a}\sum_{n=1}^{\infty}\frac{\mu(n)}{n}\tilde\phi\left(\frac{a}{n}\right)+4\sum_{n=1}^{\infty}\frac{\mu(n)}{n}\int_{0}^{\infty}\frac{\sqrt{a}}{t^2}\tilde\phi\left(\frac{a}{t}\right)\sin^2\left(\frac{\pi}{nt}\right)\dd t\nonumber\\
		&= \lim_{T_n\to \infty}\sum_{|\gamma_m|<T_n}\frac{1}{\zeta'(\rho_m)}\int_{0}^{\infty}\frac{\sqrt{a}}{x}\tilde\phi\left(\frac{a}{x}\right)x^{\rho_m-1}\dd x.
\end{align*}
Employ change of variables $t\to a/t$ and $x\to a/x$ in the two integrals and use the elementary identity $2\sin^{2}(\theta)=1-\cos(2\theta)$ so that
\begin{equation*}
	\begin{split}
		&\sqrt{a}\sum_{n=1}^{\infty}\frac{\mu(n)}{n}\tilde\phi\left(\frac{a}{n}\right)+\frac{2}{\sqrt{a}}\sum_{n=1}^{\infty}\frac{\mu(n)}{n}\int_{0}^{\infty}\tilde\phi(t)\left(1-\cos\left(\frac{2\pi t}{a n}\right)\right)\dd t\\&= \lim_{T_n\to \infty}\sum_{|\gamma_m|<T_n}\frac{a^{\rho_m-1/2}}{\zeta'(\rho_m)}\int_{0}^{\infty}\tilde\phi\left(x\right)x^{-\rho_m}\dd x.
	\end{split}
\end{equation*}
From the hypotheses given at the beginning of \eqref{mr gen}, the integral $
\frac{2}{\sqrt{\pi}}\int_{0}^{\infty}\tilde\phi(u)\cos(2 u x)\dd u$ exists; hence defining it to be $\tilde{\psi}(x)$ and then employing it and the prime number theorem in the form $\sum_{n=1}^{\infty}\mu(n)/n=0$, we are led to
\begin{equation*}
\sqrt{\a}\sum_{n=1}^{\infty}\frac{\mu(n)}{n}\tilde\phi\left(\frac{\a}{n}\right)-\frac{\sqrt{\pi}}{\sqrt{\a}}\sum_{n=1}^{\infty}\frac{\mu(n)}{n}\tilde{\psi}\left(\frac{\pi}{\a n}\right)= \lim_{T_n\to \infty}\sum_{|\gamma_m|<T_n}\frac{\a^{\rho_m-1/2}}{\zeta'(\rho_m)}\Gamma\left(1-\rho_m\right)Z_1\left(1-\rho_m\right).
\end{equation*}
Now let $\b=\pi/\a$ and use \eqref{nmt} to arrive at the first equality in \eqref{mr gen}. The second follows by swapping $\a$ and $\b$ in the first and using the second integral in \eqref{bef nmt}.

\section{Oscillations of Riesz sums}\label{sign changes}

Throughout this section we assume RH and the simplicity of the zeros of $\zeta(s)$. The Gonek-Hejhal conjecture \cite{gonek-gh}, \cite{hejhal-gh} (see also \cite[Equation (3)]{ng} states that for $k\in\mathbb{R}$,
\begin{equation}\label{gh conjecture}
\sum_{|\gamma_m|\leq T}\frac{1}{|\zeta'(\rho_m)|^{2k}}\asymp T(\log T)^{(k-1)^2}.
\end{equation}
In particular, we have
$$\sum_{|\gamma_m|\leq T}\frac{1}{|\zeta'(\rho_m)|}\asymp T(\log T)^{1/4}.$$
The conjecture is still open but 
assuming RH, Heap, Li and Zhao \cite{heap-li-zhao} have proved the conjectured lower bound in \eqref{gh conjecture} for all fractional $k\geq0$, assuming RH and  the simplicity of the zeros; in particular, for $k=1/2$, we have
\begin{equation}\label{hlz}
\sum_{|\gamma_m|\leq T}\frac{1}{|\zeta'(\rho_m)|}\gg T(\log T)^{1/4}.
\end{equation}
The special case $k=1/2$ of a recent result of Bui, Florea and Milinovich \cite[Theorem 1.1]{bui-florea-milinovich} implies that for any $\delta>0$,
$$
\sum_{\gamma_m\in \mathcal{F_T}}\frac{1}{|\zeta'(\rho_m)|}\ll T^{1+\delta},
$$
where $$
\mathcal{F}_T=\bigg\{\gamma\in (T,2 T]:|\gamma-\gamma'|\gg \frac{1}{\log T}\hspace{1mm}\text{for any other ordinate}\hspace{1mm}\gamma'\bigg\}.
$$
Since there is no non-trivial zero of $\zeta(s)$ with $|\gamma_m|<14$ , if $\mathcal{F}=\sqcup_{k=0}^{\lfloor\log_2(T)\rfloor}\mathcal{F}_{T/2^k}$, then for any $\delta>0$,
$$\sum_{|\gamma_m|\leq T\atop \gamma_m\in \mathcal{F}}\frac{1}{|\zeta'(\rho_m)|}\ll T^{1+\delta}\log T.$$
The set of excluded zeros whose ordinates do not belong to the family $\mathcal{F}$ conjecturally has arbitrarily small proportion \cite[p.~2682]{bui-florea-milinovich}.
Thus, it is reasonable to assume that for any $\epsilon>0$,
\begin{equation}\label{conj1}
	\sum_{|\gamma_m|\leq T}\frac{1}{|\zeta'(\rho_m)|}\ll T^{1+\epsilon}.   
\end{equation}
We note that the above conjecture is weaker than the Gonek-Hejhal conjecture.

To prepare ourselves for stating and proving the main results, we begin with some lemmas.
\begin{lemma}\label{beta}
	For $\Re(s)>0$, $\Re(k)>-1 $ and $y>0$, \begin{equation*}
		\int_{0}^{y}\bigg(1-\frac{x}{y}\bigg)^kx^{s-1}\dd x=y^s\frac{\Gamma(s)\Gamma(k+1)}{\Gamma(s+k+1)}.
	\end{equation*}    
\end{lemma}
\begin{proof}
	Substitute $x=y t$ and use the evaluation of Euler's beta integral.
\end{proof}
We will also need the following corollary of the well-known Kronecker Theorem \cite[Theorem IV, p.~53]{cassels}.
\begin{lemma}\label{kroneckerlemma}
	(\cite[Theorem B]{maksimova})  Let $\{\alpha_n\}_{n=1}^{N}$ be $N$ real numbers  linearly independent over $\mathbb{Q}$ and $\{\beta_n\}_{n=1}^{N}$ be $N$ arbitrary real numbers. Then for any $\epsilon>0$, there exist an arbitrary large positive real number $t$ such that 
	$||\alpha_nt-\beta_n||<\epsilon$ for all $1\leq n\leq N$, where $||x||$ denotes distance of $x$ from the nearest integer.
\end{lemma}
Using Lemma \ref{kroneckerlemma}, we obtain the following result.
\begin{lemma}\label{lemma3}
	Let $\{C_n\}_{n=1}^{\infty}$ be an infinite sequence of complex numbers such that $\sum_{n=1}^{\infty}|C_n|=S<\infty.$
	Let $\{\alpha_n\}_{n=1}^{\infty}$ be a sequence of real numbers that are linearly independent over $\mathbb{Q}$. Then we have
	\begin{align*}
	\limsup_{x\to \infty}\sum_{n=1}^{\infty}\big(C_n\exp(x\di \alpha_n)+\bar C_n\exp(-x\di \alpha_n) \big)&=2S,\nonumber\\
	 \liminf_{x\to \infty}\sum_{n=1}^{\infty}\big(C_n\exp(x\di \alpha_n)+\bar C_n\exp(-x\di \alpha_n) \big)&=-2S.
	\end{align*}
	\end{lemma}
	\begin{proof}
		Let $C_n=|C_n|\exp(\di \theta_n)$.
		Thus, $C_n\exp(x\di \alpha_n)+\bar C_n\exp(-x\di \alpha_n)=2|C_n|\cos(\theta_n+x\alpha_n).$ Fix $N$ such that $\sum_{n>N}|C_n|<\epsilon$. Since $\{\alpha_n\}_{n=1}^{N}$ are linearly independent over $\mathbb{Q}$, from Kronecker's theorem (Lemma \ref{kroneckerlemma}) for any $\epsilon>0$, we can find an arbitrarily large number $x$ such that $\left|\left|x\frac{\alpha_n}{2\pi}+\frac{\theta_n}{2\pi}\right|\right|<\frac{\epsilon}{2\pi}$,
		where $||x||$ denotes distance of $x$ from the nearest integer.
		Thus $2\pi k-\epsilon<x\alpha_n+\theta_n\leq 2\pi k+\epsilon$ for some integer $k.$
		Thus $\cos(x\alpha_n+\theta_n)=1+O(\epsilon^2)$ for all $n\leq N.$
		So we get
		\begin{equation*}
			\begin{split}
				\sum_{n=1}^{\infty}\big(C_n\exp(x\di \alpha_n)+\bar C_n\exp(-x\di \alpha_n) \big) &=2\sum_{n\leq N}|C_n|(1+O(\epsilon^2))+2\sum_{n>N}|C_n|\cos(x\alpha_n+\theta_n)\\
				&=2\sum_{n\leq N}|C_n|+O\bigg(\sum_{n\leq N}|C_n|\epsilon^2)\bigg)+2\sum_{n>N}|C_n|\cos(x\alpha_n+\theta_n)\\
				&= 2\sum_{n\leq N}|C_n|+O(S\epsilon^2))+O(\epsilon)=2S+O(\epsilon).
			\end{split}
		\end{equation*}
		Now since $\sum_{n=1}^{\infty}\left|2C_n\cos(\theta_n+x\alpha_n)\right|\leq2S$,
		$$\limsup_{x\to \infty}\sum_{n=1}^{\infty}\big(C_n\exp(x\di \alpha_n)+\bar C_n\exp(-x\di \alpha_n) \big)=2S. $$
Also, for any $\epsilon>0$, we can find an arbitrarily large number $x$ such that $\left|\left|x\frac{\alpha_n}{2\pi}+\frac{\theta_n-\pi}{2\pi}\right|\right|<\frac{\epsilon}{2\pi}.$ For such $x$, we have $|\cos(x\alpha_n+\theta_n)+1|<\epsilon$.
Using a similar method as shown above, it can be seen that
$$\liminf_{x\to \infty}\sum_{n=1}^{\infty}\big(C_n\exp(x\di \alpha_n)+\bar C_n\exp(-x\di \alpha_n) \big)=-2S. $$
\end{proof}

\begin{lemma}\label{lemma4}
	Let $\delta>0$. Assuming RH and \eqref{conj1}, as $\delta\to0$, we have
	\begin{equation}\label{8.6}
		A:=\sum_{\gamma_m}\frac{1}{|\zeta'(\rho_m)||\rho_m|^{1+\delta}}\asymp \frac{1}{\delta},
	\end{equation}
	\begin{equation}\label{8.7}
		B:=\sum_{\gamma_m}\frac{|\zeta(2\rho_m)|}{|\zeta'(\rho_m)||\rho_m|^{1+\delta}}\asymp \frac{1}{\delta},
	\end{equation}
	and 
	\begin{equation}\label{8.8}
		C:=\sum_{\gamma_m}\frac{|\zeta^4(\rho_m/2)|}{|\zeta'(\rho_m)||\rho_m|^{2+\delta}}\asymp \frac{1}{\delta}.
	\end{equation}
\end{lemma}
\begin{proof}
	Let $$S(T):=\sum_{0<\gamma_m\leq T}\frac{1}{|\zeta'(\rho_m)|}.$$
	Since $S(t)$ is an increasing function, we have $\dd S(t)\geq0.$
	Writing the sum $A$ as an integral and then performing integration by parts, we see that
	\begin{equation*}
		\begin{split}
			A=2\sum_{\gamma_m>0}\frac{1}{|\zeta'(\rho_m)||\rho_m|^{1+\delta}}=2\int_{1}^{\infty}\frac{\dd S(t)}{(\frac{1}{4}+t^2)^{\frac{1+\delta}{2}}}\asymp\int_{1}^{\infty}\frac{\dd S(t)}{t^{1+\delta}}=\frac{S(t)}{t^{1+\delta}}\bigg|_{1}^{\infty}+(1+\delta)\int_{1}^{\infty}\frac{S(t)}{t^{2+\delta}}\dd t.
		\end{split}
	\end{equation*}
	To obtain the upper bound, we use \eqref{conj1} with  $\epsilon=\delta/2$ so as to get
	\begin{equation}\label{A-one}
		A\ll \frac{t^{1+\epsilon}}{t^{1+\delta}}\bigg|_{1}^{\infty}+(1+\delta)\int_{1}^{\infty}\frac{t^{1+\epsilon}}{t^{2+\delta}}\dd t=\frac{1}{t^{\delta/2}}\bigg|_{1}^{\infty}+(1+\delta)\int_{1}^{\infty}\frac{1}{t^{1+\delta/2}}\dd t\asymp\frac{1}{\delta}.
	\end{equation}
	Thus we have $A\ll\delta^{-1}$. 
	To prove the lower bound, using \eqref{hlz}, we have
	\begin{equation}\label{A-two}
		\begin{split}
			A\gg \frac{t\log^{1/4}t}{t^{1+\delta}}\bigg|_{1}^{\infty}+(1+\delta)\int_{1}^{\infty}\frac{t\log^{1/4}t}{t^{2+\delta}}\dd t=\frac{\log^{1/4}t}{t^{\delta}}\bigg|_{1}^{\infty}+(1+\delta)\int_{1}^{\infty}\frac{\log^{1/4}t}{t^{1+\delta}}\dd t\asymp \frac{1}{\delta}.
		\end{split}
	\end{equation}
	From \eqref{A-one} and \eqref{A-two}, we conclude that $A\asymp \frac{1}{\delta}.$
	
	Next, we prove that $B\asymp\delta^{-1}.$
	We know from the results of Titchmarch \cite{titchmarsh} and Littlewood \cite[Theorem 1]{littlewood} that for any given $\epsilon>0$, $\zeta(1+\di t)\ll t^\epsilon$ and $\zeta(1+\di t)\gg t^{-\epsilon}.$
	Taking $\epsilon=\delta/2$ and analyzing the sum $B$ in a similar way as $A$, we see that
	\begin{equation*}
		\begin{split}
			B=\sum_{\gamma_m}\frac{|\zeta(2\rho_m)|}{|\zeta'(\rho_m)||\rho_m|^{1+\delta}}\asymp\int_{1}^{\infty}|\zeta(1+2\di t)|\frac{\dd S(t)}{t^{1+\delta}}\ll \int_{1}^{\infty}\frac{\dd S(t)}{t^{1+\delta-\epsilon}}\ll \int_{1}^{\infty}\frac{\dd S(t)}{t^{1+\delta/2}},
		\end{split}
	\end{equation*}
thereby leading to $B\ll \delta^{-1}.$
We prove the lower bound as follows:
	\begin{equation*}
		\begin{split}
			B=\sum_{\gamma_m}\frac{|\zeta(2\rho_m)|}{|\zeta'(\rho_m)||\rho_m|^{1+\delta}}\asymp\int_{1}^{\infty}|\zeta(1+2\di t)|\frac{\dd S(t)}{t^{1+\delta}}\gg\int_{1}^{\infty}\frac{\dd S(t)}{t^{1+3\delta/2}}\gg\frac{1}{\delta}.
		\end{split}
	\end{equation*}
	We now prove $C\asymp\delta^{-1}.$ Note that
	\begin{equation*}
		\begin{split}
			C=\sum_{\gamma_m>0}\frac{|\zeta^4(\rho_m/2)|}{|\zeta'(\rho_m)||\rho_m|^{2+\delta}}\asymp\int_{1}^{\infty}\bigg|\zeta^4\bigg(\frac{1}{4}+\frac{\di t}{2}\bigg)\bigg|\frac{\dd S(t)}{t^{2+\delta}}.
		\end{split}
	\end{equation*}
	Using the fact $\zeta(s)\asymp t^{\frac{1}{2}-\sigma}\zeta(1-s)$, we have
	$\zeta(1/4+\di t)\asymp t^{1/4}\zeta(3/4-\di t).$
	Using RH, Titchmarsh\cite[p.~337, Equations (14.2.5), (14.2.6)]{titch} proved that
	$\zeta(3/4-\di t)\ll t^{\epsilon}$ and $\zeta(3/4-\di t)\gg t^{-\epsilon}.$
	Thus $\zeta(1/4+\di t)\ll t^{1/4+\epsilon}$ and $\zeta(1/4+\di t)\gg t^{1/4-\epsilon}.$ 
	Thus we have $\zeta^4(1/4+\di t)\ll t^{1+4\epsilon}$ and $\zeta^4(1/4+\di t)\gg t^{1-4\epsilon}$. We choose $\epsilon=\delta/8.$
	Thus we have,
	$$
	C\asymp \int_{1}^{\infty}\bigg|\zeta^4\bigg(\frac{1}{4}+\frac{\di t}{2}\bigg)\bigg|\frac{\dd S(t)}{t^{2+\delta}}\ll\int_{1}^{\infty}\frac{\dd S(t)}{t^{1+\delta-4\epsilon}}=\int_{1}^{\infty}\frac{\dd S(t)}{t^{1+\delta/2}} \ll \frac{1}{\delta}
	$$
	and
	$$
	C\asymp \int_{1}^{\infty}\bigg|\zeta^4\bigg(\frac{1}{4}+\frac{\di t}{2}\bigg)\bigg|\frac{\dd S(t)}{t^{2+\delta}}\gg\int_{1}^{\infty}\frac{\dd S(t)}{t^{1+3\delta/2}} \gg \frac{1}{\delta}.
	$$
		Thus $C\asymp \frac{1}{\delta}$.
\end{proof}

\begin{theorem}\label{mu riesz sign changes}
	Let $\delta>0$. Assume RH,  simplicity of the non-trivial zeros of $\zeta(s)$, Equation \eqref{conj1}, and the Linear Independence conjecture. Then there exists an absolute positive constant $c$ such that for all sufficiently small $\delta$,
	\begin{equation}\label{final 1}
	\limsup_{y\to \infty} \sum_{n\leq y}\frac{\mu(n)(1-n/y)^\delta}{\sqrt{y}}\geq\frac{c}{\delta}\hspace{4mm}\textup{and}\hspace{4mm}
	\liminf_{y\to \infty} \sum_{n\leq y}\frac{\mu(n)(1-n/y)^\delta}{\sqrt{y}}\leq-\frac{c}{\delta}.
\end{equation}
\end{theorem}
\begin{proof}
	In \ref{eq96}, we take 
	\begin{equation}\label{phi special}
	\phi(x)=
	\begin{cases}
		\left(1-\frac{x}{y}\right)^{\delta}&\text {if $x\leq y$},\\
		0 &\text{otherwise}.
	\end{cases}
	\end{equation}
	For a fixed $a>0$ as $|s|>1$, we have
	\begin{equation}\label{asmp}
		\frac{\Gamma(s)}{\Gamma(s+a)}=\frac{1}{s^a}+O\bigg(\frac{1}{s^{a+1}}\bigg).
	\end{equation}
	From Lemma \ref{beta} and \eqref{asmp},
	\begin{equation*}
		\begin{split}
			\frac{1}{\zeta'(\rho_m)}\int_{0}^{y}\bigg(1-\frac{x}{y}\bigg)^\delta x^{\rho_m-1}\dd x=\frac{y^{\rho_m}}{\zeta'(\rho_m)}\frac{\Gamma(1+\delta)\Gamma(\rho_m)}{\Gamma(1+\rho_m+\delta)} \asymp\frac{y^{\rho_m}}{\zeta'(\rho_m)}\frac{\Gamma(1+\delta)}{\rho_m^{1+\delta}} + O\bigg(\frac{\sqrt{y}}{\zeta'(\rho_m)\rho_m^2}\bigg).
		\end{split}
	\end{equation*}
	Since, using the ideas in the proof of Lemma \ref{conj1} and under the same hypotheses, it is evident that the series $\sum\limits_{\rho_m}\frac{1}{\left|\zeta'(\rho_m)\rho_m^2\right|}$ converges, we get
	\begin{equation*}
		\sum_{\gamma_m}\frac{1}{\zeta'(\rho_m)}\int_{0}^{y}\bigg(1-\frac{x}{y}\bigg)^\delta x^{\rho_m-1}\dd x=\sqrt{y}\sum_{\gamma_m}\frac{y^{\di \gamma_m}}{\zeta'(\rho_m)}\frac{\Gamma(1+\delta)}{\rho_m^{1+\delta}} + O(\sqrt{y}).
	\end{equation*}
	Now we apply Lemma \ref{lemma3} with $x=\log y$, $C_m=\rho^{-\delta-1}_m/\zeta'(\rho_m)$, assume the Linear Independence conjecture so that we can let $\alpha_m=\gamma_m, m>0$, and use \eqref{8.6} to obtain
	\begin{equation}\label{to get final}
	\limsup_{y\to\infty}\sum_{\gamma_m}\frac{y^{\di\gamma_m}}{\rho_m^{1+\delta}\zeta'(\rho_m)}\geq\frac{c'}{\delta}\hspace{4mm}
	\text{and}\hspace{4mm}
	\liminf_{y\to\infty}\sum_{\gamma_m}\frac{y^{\di\gamma_m}}{\rho_m^{1+\delta}\zeta'(\rho_m)}\leq-\frac{c'}{\delta}
	\end{equation}
	for some absolute constant $c'>0.$
	The Mellin transform of $\phi(x)$ is 
	$
	\Phi(s)=y^s\frac{\Gamma(1+\delta)\Gamma(s)}{\Gamma(1+\delta+s)}
	$ and
	\begin{equation*}
		\int_{(-\tau)}\frac{\Phi(s)}{\zeta(s)}\dd s\ll \int_{(-\tau)}\bigg|\frac{\Phi(s)}{\zeta(s)}\bigg| |\dd s|\ll y^{-\tau}\int_{(-\tau)}\bigg|\frac{\Gamma(1+\delta)\Gamma(s)}{\Gamma(1+\delta+s)\zeta(s)}\bigg||\dd s|=O(y^{-\tau}). 
	\end{equation*}
Moreover, the residue of $\Phi(s)$ at $s=0$	is $1$. Therefore, equation \eqref{eq96} becomes
	\begin{equation*}
		\begin{split}
			\sum_{n\leq y}\mu(n)(1-n/y)^\delta&= \lim_{T_n\to \infty}\sum_{|\gamma_m|<T_n}\sqrt{y}\frac{y^{\di \gamma_m}}{\zeta'(\rho_m)}\frac{\Gamma(1+\delta)}{\rho_m^{1+\delta}} + O(\sqrt{y})-2+O(y^{-\tau}). 
		\end{split}
	\end{equation*}
	Dividing both sides by $\sqrt{y}$, we get
	\begin{equation*}
		\begin{split}
			\frac{1}{\sqrt{y}}\sum_{n\leq y}\mu(n)(1-n/y)^{\delta}&= \lim_{T_n\to \infty}\sum_{|\gamma_m|<T_n}\frac{y^{\di \gamma_m}}{\zeta'(\rho_m)}\frac{\Gamma(1+\delta)}{\rho_m^{1+\delta}} + O(1),
		\end{split}
	\end{equation*}
	which, along with \eqref{to get final}, implies \eqref{final 1}.
\end{proof}

\begin{theorem}\label{lambda riesz sign changes}
	Let $\delta>0$. Assume RH,  simplicity of the non-trivial zeros of $\zeta(s)$, Equation \eqref{conj1}, and the Linear Independence conjecture. Then there exists an absolute positive constant $c$ such that for all sufficiently small $\delta$,
	\begin{equation}\label{final 2}
	 \limsup_{y\to \infty} \sum_{n\leq y}\frac{\lambda(n)(1-n/y)^\delta}{\sqrt{y}}\geq\frac{c}{\delta}\hspace{4mm}\textup{and}\hspace{4mm}
	\liminf_{y\to \infty} \sum_{n\leq y}\frac{\lambda(n)(1-n/y)^\delta}{\sqrt{y}}\leq-\frac{c}{\delta}.
\end{equation}
\end{theorem}
\begin{proof}
	From \eqref{first expression for I}, \eqref{added} and \eqref{r0}, for a small positive $\tau$,
	\begin{align*}
		\sum_{n=1}^{\infty}\lambda(n)\phi(n)=\frac{1}{2\pi i}\int_{(-\tau)}\frac{\zeta(2s)}{\zeta(s)}\Phi(s)\dd s+\textup{Res}_{s=0}\frac{\zeta(2s)}{\zeta(s)}\Phi(s)+\frac{\Phi\left(\frac{1}{2}\right)}{2\zeta(\frac{1}{2})}+\lim_{T_n\to\infty}\sum_{|\gamma_m|<T_n}\frac{\zeta(2\rho_m)\Phi(\rho_m)}{\zeta'(\rho_m)}.
	\end{align*}
	Let $\phi$ be defined in \eqref{phi special}. 
From Lemma 
\ref{beta} and  \eqref{asmp},
\begin{equation*}
	\begin{split}
		\frac{\zeta(2\rho_m)}{\zeta'(\rho_m)}\int_{0}^{y}\bigg(1-\frac{x}{y}\bigg)^\delta x^{\rho_m-1}\dd x=y^{\rho_m}\frac{\zeta(2\rho_m)}{\zeta'(\rho_m)}\frac{\Gamma(1+\delta)\Gamma(\rho_m)}{\Gamma(1+\rho_m+\delta)} \asymp y^{\rho_m}\frac{\zeta(2\rho_m)}{\zeta'(\rho_m)}\frac{\Gamma(1+\delta)}{\rho_m^{1+\delta}} + O\bigg(\sqrt{y}\frac{\zeta(2\rho_m)}{\zeta'(\rho_m)\rho_m^2}\bigg).
	\end{split}
\end{equation*}
Now use Lemma \ref{lemma3} with $C_m=\rho^{-\delta-1}_m\zeta(2\rho_m)/\zeta'(\rho_m)$ and $\alpha_m=\gamma_m$ for $m>0$ and $x=\log y$ and use \eqref{8.7} to obtain
\begin{equation}\label{to get final 1}
\limsup_{y\to\infty}\sum_{m=1}^{\infty}y^{\di\gamma_m}\frac{\zeta(2\rho_m)}{\rho_m^{1+\delta}\zeta'(\rho_m)}\geq\frac{c'}{\delta}\hspace{4mm}\text{and}\hspace{4mm}\liminf_{y\to\infty}\sum_{m=1}^{\infty}y^{\di\gamma_m}\frac{\zeta(2\rho_m)}{\rho_m^{1+\delta}\zeta'(\rho_m)}\leq-\frac{c'}{\delta}
\end{equation}
for some absolute constant $c'>0.$
Moreover,
\begin{equation*}
	\int_{(-\tau)}\frac{\zeta(2s)}{\zeta(s)}\Phi(s)\dd s\ll \int_{(-\tau)}\bigg|\frac{\zeta(2s)}{\zeta(s)}\bigg||\Phi(s)||\dd s|\ll y^{-\tau}\int_{(-\tau)}\bigg|\frac{\Gamma(1+\delta)\Gamma(s)\zeta(2s)}{\Gamma(1+k+s)\zeta(s)}\bigg||\dd s|=O(y^{-\tau}). 
\end{equation*}
Also $\Res_{s=0}\frac{\zeta(2s)}{\zeta(s)}\Phi(s)=1$. So
\begin{equation*}
	\begin{split}
		\sum_{n\leq y}\lambda(n)\bigg(1-\frac{n}{y}\bigg)^\delta&=\frac{y^{1/2}}{2\zeta(\frac{1}{2})}\frac{\Gamma(1/2)\Gamma(1+\delta)}{\Gamma(1+\delta+1/2)}+ \lim_{T_n\to \infty}\sqrt{y}\sum_{|\gamma_m|<T_n}y^{\di \gamma_m}\frac{\zeta(2\rho_m)}{\zeta'(\rho_m)}\frac{\Gamma(1+\delta)}{\rho_m^{1+\delta}} + O\left(\sqrt{y}\right)\\
		&\quad+O(y^{-\tau})+1.
	\end{split}
\end{equation*}
Dividing both sides by $\sqrt{y}$, we get
\begin{equation*}
	\begin{split}
		\sum_{n\leq y}\frac{\lambda(n)(1-n/y)^{\delta}}{\sqrt{y}}&= \lim_{T_n\to \infty}\sum_{|\gamma_m|<T_n}\frac{y^{\di \gamma_m}}{\zeta'(\rho_m)}\frac{\Gamma(1+\delta)}{\rho_m^{1+\delta}} + O(1),
	\end{split}
\end{equation*}
which, along with \eqref{to get final 1}, implies \eqref{final 2}.

\end{proof}

\begin{theorem}\label{d(n) squared sign changes}
	Let $\delta>0$. Assume RH,  simplicity of the non-trivial zeros of $\zeta(s)$, Equation \eqref{conj1}, and the Linear Independence conjecture. Then there exists an absolute positive constant $c$ such that for all sufficiently small $\delta$,
	\begin{align}\label{final 3}
	\limsup_{y\to \infty}\frac{\sum_{n\leq y}d^2(n)(1-n/y)^{1+\delta}-g(y)}{y^{1/4}}\geq\frac{c}{\delta}
	\hspace{4mm}\text{and}\hspace{4mm}
	\liminf_{y\to \infty}\frac{\sum_{n\leq y}d^2(n)(1-n/y)^{1+\delta}-g(y)}{y^{1/4}}\leq-\frac{c}{\delta},
\end{align}
	where
	 \begin{align}
	 	g(y)&:=\frac{y}{2+\delta}\sum_{j=0}^{3}B_{\delta, j}(y),\label{g(y)}\\
 B_{\delta, j}(y)&:=A_j\int_{0}^y\left(1-\frac{x}{y}\right)^{1+\delta}\log^{j}(x)\dd x\hspace{8mm}(0\leq j\leq 3),\nonumber\\
 \end{align} 
 with $A_j, 1\leq j\leq 3$, defined in \eqref{A0-A3}.
\end{theorem}
\begin{proof}
	In Theorem \ref{vsf dn squared}, we considered $\Phi(s)$ to be holomorphic in $-1<\textup{Re}(s)<2$ since we were explicitly evaluating the integral $\int_{(-\tau)}\frac{\zeta^4(s)}{\zeta(2s)}\Phi(s)\dd s$ to be the sum on the right-hand side of \eqref{vsf dn squared eqn}, and had we allowed $\Phi(s)$ to have a pole at $s=0$, computing this would have been unwieldy . However, in what follows, we allow $\Phi(s)$ to have a simple pole at $s=0$ because we are only estimating the integral over Re$(s)=-\tau$. We thus have to consider the contribution of the residue at this additional pole, and because of this, and \eqref{I-series}, \eqref{I-shift} and \eqref{1R}, for a small positive $\tau$,
\begin{align*}
\sum_{n=1}^{\infty}d^2(n)\phi(n)&=\frac{1}{2\pi i}\int_{(-\tau)}\frac{\zeta^4(s)}{\zeta(2s)}\Phi(s)\dd s+\int_{0}^{\infty}(A_0+A_1\log(x)+A_2\log^2( x)+A_3\log^3(x))\phi(x)\dd x\nonumber\\
&\quad+\textup{Res}_{s=0}\frac{\zeta^4(s)}{\zeta(2s)}\Phi(s)+\lim_{T_n\to\infty}\sum_{|\gamma_m|<T_n}\frac{\zeta^4(\rho_m)\Phi(\rho_m)}{\zeta(2\rho_m)},
\end{align*}
where $A_0, A_1, A_2$ and $A_3$ are defined in \eqref{A0-A3}. In the above identity, let
$$
\phi(x)=
\begin{cases}
	\left(1-\frac{x}{y}\right)^{1+\delta}&\text {if $x\leq y$},\\
	0 &\text{otherwise}.
\end{cases}
$$
Now
\begin{equation*}
	\begin{split}
		\frac{\zeta^4(\frac{\rho_m}{2})}{\zeta'(\rho_m)}\int_{0}^{y}\bigg(1-\frac{x}{y}\bigg)^{1+\delta}x^{\frac{\rho_m}{2}-1}\dd x&=\frac{\zeta^4(\frac{\rho_m}{2})}{\zeta'(\rho_m)}y^{\rho_m/2}\frac{\Gamma(2+\delta)\Gamma(\rho_m/2)}{\Gamma(2+\delta+\rho_m/2)}\\
		&=y^{1/4}\frac{\zeta^4(\frac{\rho_m}{2})\Gamma(2+\delta)}{\zeta'(\rho_m)(\rho_m/2)^{2+\delta}}y^{\di \gamma_m/2}+O\bigg(y^{1/4}\frac{\zeta^4(\frac{\rho_m}{2})}{\zeta'(\rho_m)\rho_m^3}\bigg).
	\end{split}
\end{equation*}
Also,
$$
\Res_{s=0}\frac{\zeta^4(s)}{\zeta(2s)}\Phi(s)=\Res_{s=0}\frac{\zeta^4(s)}{\zeta(2s)}y^s \frac{\Gamma(\delta+2)\Gamma(s)}{\Gamma(s+\delta+2)}=-\frac{1}{8}.
$$
So we get
\begin{equation*}
	\begin{split}
		\sum_{n\leq y}d^2(n)\bigg(1-\frac{n}{y}\bigg)^{1+\delta}&=g(y)-\frac{1}{8}+\lim_{T_n\to \infty}\sum_{|\gamma_m|\leq T_n}y^{1/4}\frac{2^{2+\delta}\zeta^4(\frac{\rho_m}{2})}{\zeta'(\rho_m)\rho_m^{2+\delta}}y^{\di \gamma_m/2}+O\bigg(y^{1/4}\frac{\zeta^4(\frac{\rho_m}{2})}{\zeta'(\rho_m)\rho_m^3}\bigg),
	\end{split}
\end{equation*}
where $g(y)$ is defined in \eqref{g(y)}. Dividing both sides by $y^{1/4}$, we get
\begin{equation*}
	\begin{split}
		\frac{ \sum_{n\leq y}d^2(n)(1-n/y)^{1+\delta}-g(y)}{y^{1/4}}=
		\lim_{T_n\to \infty}\sum_{|\gamma_m|\leq T_n}\frac{2^{2+\delta}\zeta^4(\frac{\rho_m}{2})}{\zeta'(\rho_m)\rho_m^{2+\delta}}y^{\di \gamma_m/2}+O(1).
	\end{split}
\end{equation*}
Using \eqref{8.8} and ideas similar to those used to prove Theorems \ref{mu riesz sign changes} and \ref{lambda riesz sign changes}, we are led to \eqref{final 3}.
\end{proof}
\begin{remark}
While we do not explicitly write down $B_{\delta, j}(x)$, $1\leq j\leq 3$, we note that $g(y)/y$ is a cubic polynomial in $\log(y)$. 
\end{remark}
\begin{remark}
	The assumption of \eqref{conj1} in Theorems \ref{mu riesz sign changes}, \ref{lambda riesz sign changes} and \ref{d(n) squared sign changes} can be replaced by the weaker assumption that the series $A$, $B$, and $C$ in Lemma \ref{lemma4} converge. 
	\end{remark}
Results analogous to those derived in Theorems \ref{mu riesz sign changes}-\ref{d(n) squared sign changes} can be obtained by choosing $\phi_y(x)=e^{-xy}$, however, the techniques involved in proving them are similar, and hence to avoid repetition, we have chosen to forego the calculations.

\section{Concluding remarks}\label{cr}
Vorono\"{\dotlessi} summation formulas have been historically used to improve the error term associated to the summatory function of the arithmetic function involved. Several articles have been written on refining the error term of $\sum_{n\leq x}d^2(n)$, in particular. Let $E(x)$ denote the error term. The current best upper bound on $E(x)$ is by Jia and Sankaranarayanan \cite{jia-sankaranarayanan} who proved that
$E(x)=O\left(x^{1/2}\log^5(x)\right)$ whereas Chandrasekharan and Narasimhan \cite{chna} have shown that $E(x)=\Omega_{\pm}(x^{1/4})$. See \cite{jia-sankaranarayanan}, \cite[p.~438]{robles-roy} for the history on this topic. It would be nice to see if the Vorono\"{\dotlessi} summation formula that we have obtained for $d^2(n)$ (Theorem \ref{vsf dn squared}) could be used for improving $E(x)$. Although the hypotheses of Theorem \ref{vsf dn squared} does not allow us to take $\phi(x)$ to be the characteristic function of the interval $[1,x]$, where $x>1$, heuristically, if we choose $\phi$ to be so, then the main term of our result turns out to be exactly the one given by Ramanujan \cite{ramanujan-some formulae}, namely\footnote{Our $A_0, A_1, A_2$ and $A_3$ are respectively $D, C, B$ and $A$ of Ramanujan.}, 
$A_0x+A_1x\log(x)+A_2x\log^2(x)+A_3x\log^{3}(x)$. Thus, it is certainly of merit to expand the class of admissible functions $\phi$ for which Theorem \ref{vsf dn squared} is valid and thereby obtain better estimates on the error term.

 Koshliakov \cite{koshliakov} devised an ingenious method to derive the Vorono\"{\dotlessi} summation for $\sum_{\alpha<n<\beta}d(n)\phi(n)$,  where $0<\a<\b, \a,\b\notin\mathbb{Z},$ and $f$ denotes a functional analytic inside a closed contour strictly containing $[\a, \b]$. This technique requires nothing more than the functional equation of $\zeta(s)$ and the Cauchy residue theorem. It was adapted in \cite[Theorem 6.1]{bdrz1} and then in \cite[Theorem 2.2]{dmv} to handle the corresponding formulas for $\sigma_s(n)$ and $\sigma_s^{(k)}(n)$ respectively. In his method, Koshliakov starts from the identity \eqref{obe} (which, as mentioned in the introduction, results from \eqref{voronoi dn} by replacing $x$ by $ix$ and $-ix$ and adding the resulting two identities), replaces $x$ by $iz$ and $-iz$ in \eqref{obe}, adds the corresponding sides of the two resulting identities, applies the functional equation of $\zeta(s)$, and then analyzes a certain contour integral to arrive at the Vorono\"{\dotlessi} summation formula. 
 
 Unfortunately, we were unable to adapt this technique for the arithmetic functions we have considered in our paper. For example, consider the case of $d^{2}(n)$. We \emph{do} have the analogue of \eqref{voronoi dn}, namely, \eqref{cohen square of d(n)}. However, if we replace $x$ in there by $ix$ and $-ix$, and add the resulting two identities to obtain a corresponding analogue of \eqref{obe}, the problem is the resulting series of the non-trivial zeros of $\zeta(s)$ diverges. However, if one is able to modify this step and thereby circumvent this obstacle, one may be able to obtain the Vorono\"{\dotlessi} summation not only for $\sum_{\alpha<n<\beta}d^2(n)\phi(n)$, but also for $\sum_{\alpha<n<\beta}\sigma_a(n)\sigma_b(n)\phi(n)$. Then one may as well try out the same in the setting of $\lambda(n)$ for which, again, the analogue of \eqref{voronoi dn} exists, namely, \eqref{cohen lambda eqn}. This definitely seems to be a worthwhile problem to try. 

It would also be important to obtain analogues of our summation formulas and identities of Ramanujan-Guinand type and Cohen type for different $L$-functions. For example, Berndt, Kim and Zaharescu \cite{bkz-koshliakov} derived the Ramanujan-Guinand type formula for the twisted divisor function $d_{\chi}(n):=\sum_{d|n}\chi(d)$, where $\chi$ is an even primitive Dirichlet character, and used to show that when $\chi$ is real, $L(1, \chi)>0$.

In Section \ref{sign changes}, we obtained results on oscillations of the Riesz sums associated with $\lambda(n), \mu(n)$ and $d^2(n)$. As remarked at the end of that section, one may derive similar results where the test function is $e^{-xy}$. The key thing which makes such results work is the presence of the series over the non-trivial zeros of $\zeta(s)$. In light of this, it may be interesting to find other functions than the ones we have studied which also exhibit oscillations caused due to the series over the non-trivial zeros.
\begin{center}
\textbf{Acknowledgements}
\end{center}
The authors sincerely thank Timothy Trudgian, M. Ram Murty, Andr\'{e}s Chirre, Nathan Ng, Arindam Roy and Nicolas Robles for helpful discussions, and to Rahul Kumar for sending us copies of \cite{heap-li-zhao}, \cite{littlewood} and \cite{titchmarsh}. The second author is supported by the Swarnajayanti Fellowship grant SB/SJF/2021-22/08 of ANRF (Govt. of India). The first author is a postdoctoral fellow funded by the same grant. Both the authors thank ANRF for the support.

\end{document}